\documentclass[11pt]{article}
\usepackage[utf8]{inputenc}
\usepackage{amsmath,amssymb,amsfonts}
\usepackage{graphicx}
\usepackage{geometry}
\usepackage{setspace}
\usepackage{amsthm}
\newtheorem{theorem}{Theorem}

\newtheorem{lemma}{Lemma}
\newtheorem{definition}{Definition}
\newtheorem{remark}{Remark}
\newtheorem{corollary}{Corollary}

\DeclareMathOperator*{\argmin}{argmin}

\RequirePackage{tgtermes}
\RequirePackage{newtxtext}
\RequirePackage{newtxmath}
\RequirePackage{bm}
\RequirePackage{endnotes}

\usepackage{algorithm}
\usepackage{algpseudocode}
\usepackage{tikz}

\usepackage[colorlinks=true,
            breaklinks=true,
            bookmarks=true,
            urlcolor=blue,
            citecolor=blue,
            linkcolor=blue,
            bookmarksopen=false,
            draft=false]{hyperref}
\usepackage{caption}
\usepackage{subcaption}
\usepackage{float}
\usepackage{enumerate}
\usepackage{bbm}

\usepackage{empheq}
\usepackage{xfrac}

\def\balpha{\boldsymbol{\alpha}}
\def\bualpha{\underline{\boldsymbol{\alpha}}}

\def\bmu{\boldsymbol{\mu}}

\newcommand\cF{\mathcal F}
\newcommand\cG{\mathcal G}

\newcommand\cK{\mathcal K}
\newcommand\cL{\mathcal L}

\newcommand\cP{\mathcal P}

\newcommand\cf{\mathfrak f}
\newcommand\cg{\mathfrak g}
\newcommand\cb{\mathfrak b}

\newcommand\bK{\boldsymbol K}

\newcommand\bX{\boldsymbol{X}}
\newcommand\bY{\boldsymbol{Y}}

\newcommand\bZ{\boldsymbol{Z}}
\newcommand\bA{\boldsymbol{A}}
\newcommand\bW{\boldsymbol{W}}

\newcommand\EE{\mathbb E}

\newcommand\PP{\mathbb P}

\newcommand\RR{\mathbb R}

\newcommand\bbA{\mathbb A}

\newcommand{\mixind}{{\hbox{\tiny $\mathrm{MI}$}}}
\newcommand{\mixpop}{{\hbox{\tiny $\mathrm{MP}$}}}
\newcommand{\rmc}{{\hbox{\tiny $\mathrm{C}$}}}
\newcommand{\rmnc}{{\hbox{\tiny $\mathrm{NC}$}}}

\newcommand{\indic}{\mathbbm{1}}

\usepackage{natbib}
 \bibpunct[, ]{(}{)}{,}{a}{}{,}%
 \def\BIBand{and}%

\title{Cooperation, Competition, and Common Pool Resources\\ in Mean Field Games}
\author{G\"ok\c ce Dayan{\i}kl{\i}\footnote{Department of Statistics, University of Illinois at Urbana-Champaign, 
  Champaign, IL 61820, USA 
  (\href{mailto:gokced@illinois.edu}{gokced@illinois.edu}).}
\and Mathieu Lauri\`ere
    \footnote{Shanghai Frontiers Science Center of Artificial Intelligence and Deep Learning; NYU-ECNU Institute of Mathematical Sciences at NYU Shanghai; NYU Shanghai, 567 West Yangsi Road, Shanghai, 200126, People’s Republic of China, 
  (\href{mailto:ml5197@nyu.edu}{ml5197@nyu.edu}).}
}
\date{} %

\begin{document}

\maketitle

\begin{abstract}
Mean field games (MFGs) have been introduced to study Nash equilibria in very large population of self-interested agents. However, when applied to common pool resource (CPR) games, MFG equilibria lead to the so-called tragedy of the commons (TOTC). Empirical studies have shown that in many situations, TOTC does not materialize which hints at the fact that standard MFG models cannot explain the behavior of agents in CPR games. In this work, we study two models which incorporate a mix of cooperative and non-cooperative behaviors, either at the individual level or the population level. After defining these models, we study optimality conditions in the form of forward-backward stochastic differential equations and we prove that the mean field models provide approximate equilibria controls for corresponding finite-agent games. We then show an application to a model of fish stock management, for which the solution can be computed by solving systems of ordinary differential equations, which we prove to have a unique solution. Numerical results illustrate the impact of the level of cooperation at the individual and the population levels on the CPR.
\end{abstract}

\section{Introduction}
\label{sec:intro}

\subsection{Background}

 The notion of the tragedy of the commons (TOTC) was introduced by~\cite{hardin1968tragedy}  and it states that the individual incentives will result in overusing common pool resources\footnote{These are the resources that are available to all individuals by consumption and depending on the resource of interest, they can be replenished.} which in turn may have detrimental future consequences that affect everyone involved negatively. However, in many real-life situations, this does not happen and researchers such as the Nobel laureate Elinor Ostrom suggested, based on empirical studies, that local governance and mutual restraint by individuals can be the preventing factor~\citep{ostrom1992governing,ostrom1999coping}.
 In the context of fisheries, empirical studies have shown that self-organization of fishermen can help explaining the recovery of overexploited resources without the control of a central planner~\citep{brochier2018can}. The topic has also attracted the interest of the management science community, particularly through studies based on empirical data such as~\citep{moxnes1998not} for fisheries in Norway and~\cite{mcgahan2023there} for the Amazon rainforest preservation.

The possibility of seeing common pool resources being over-exploited is particularly serious when individuals believe their influence is negligible compared with the rest of the population. For example when individual fishermen access to the same pool of fish, each fisherman by herself does not significantly decrease the stock of fish, but if a large number of fishermen simultaneously fish a lot, the impact will be tremendous. Likewise, individuals have, taken separately, very small influence on the total pollution, but the collective effect is significant. In such situations, a common resource (e.g., fish or air quality) is over-exploited by a large group of agents, and no individual can save the common pool by changing only her own behavior. This is perfectly captured by mean field games (MFG), introduced by~\citep{lasry2007mean} and~\citep{huang2006large}. In MFGs, individuals are infinitesimal and fully non-cooperative, so it is expected that TOTC would appear naturally. However, in many real-life scenarios, the population exhibits a mixture of selfishness and altruism. This could help explaining why TOTC is manifested less often than what would be the case with purely self-interested agents. 

Motivated by this, we introduce and study new notions of mean field equilibria to model mixtures of cooperative and non-cooperative behaviors, and we study their application to common pool resources.

\begin{figure*}%
    \begin{center}
    \includegraphics[width=0.8\linewidth]{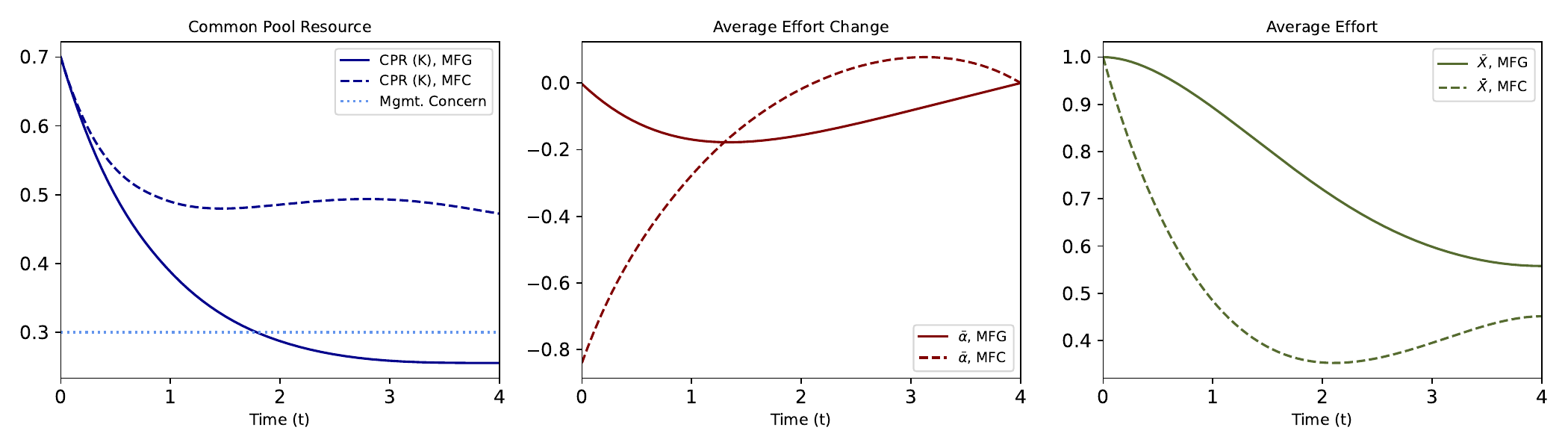}
    \caption{Illustrative example. Evolution of the common pool resource (left), of the average change in effort (middle) and of the average effort (right). In the MFG, agents over-exploit the common pool resource.  }
    \label{fig:mfgVSmfc}
    \end{center}
\end{figure*}

\subsection{Literature review}

Since their introduction, the theory of MFGs have been developed using partial differential equations or stochastic differential equations, see the monographs~\citep{bensoussan2013mean} and~\citep{CarmonaDelarue_book_I} for more background. MFGs have found applications in wide range of fields including management science, such as dynamic auctions~\citep{iyer2014mean}, contract theory~\citep{carmona2021finite,elie2019tale}, pricing and inventory management~\citep{bensoussan2022new}, cryptocurrency mining~\citep{li2024mean}, epidemic control~\citep{aurell2022optimal,elie2020contact}, or energy market regulations~\citep{carbon_StackelbergMFG,shrivats_mfin}. Besides the standard MFG framework which focuses on the notion of Nash equilibrium between purely self-interested players, the scenario with cooperative players looking for a social equilibrium has also been considered under the terminology of mean field control (MFC). We refer to~\citep{bensoussan2013mean} for more details on the theory of MFGs and MFC problems. 
So far, the MFG literature has focused on these two scenarios -- fully cooperative or fully non-cooperative -- with very few exceptions. \cite{carmona2023nash} have studied a model in which one part of the population is non-cooperative and plays a Nash equilibrium while the other part follows a fixed control chosen by a central planner independently of the non-cooperative agents. This is different from our MP model in which the agents of both populations react to each other. \cite{angiuli2022reinforcement,angiuli2023reinforcementmixedmfcg} have introduced a model similar to our MI mean field model but they focus on a specific linear-quadratic (LQ) model and did not study the connection with finite-agent games nor the FBSDE characterization. \cite{li2022dynamic} also considered an LQ model and studied its solvability and the link with finite-agent games using the connection with Riccati equations. In~\cite{dayanikli2024can}, authors studied a class of LQ models for both MI and MP settings, and also solved them using Riccati equations. However, in the present paper, we consider more general models, beyond the LQ case.

One of the most typical example of application motivating the tragedy of the commons is the question of fishing from a common pool of fish. This example has also been revisited in the context of MFGs.  In particular, \cite{McPike_mfg_fish} studied an application to a fisheries model, which inspired the model we propose in Section~\ref{sec:fishing-model}. 
However, they focused on a purely non-cooperative model (i.e., MFG) and did not study mixed models. While this work was completed, \cite{kobeissi2024mean,kobeissi2024tragedy} also studied applications of MFG to fisheries models and the tragedy of the commons. 
\cite{yoshioka2024numerical} proposed a numerical method for an MFG model with a common fishing stock.  However, none of these works studied models with a mixture of cooperation and competition like the ones we propose.

\subsection{Contributions and paper outline}

Our contributions are three-fold. 
First, from the conceptual viewpoint, we introduce  \emph{mixed individual} MFGs and \emph{mixed population} MFGs. The former captures altruistic tendencies at the individual level and the latter models a population that is a mixture of fully cooperative and non-cooperative individuals. Our second contribution is more technical: For both cases, we prove that the mean field model provides a good approximate equilibrium in the corresponding finite-agent game, and we characterize the equilibria using forward-backward stochastic differential equations. 
Lastly, we show how these models can be applied to common pool resources (CPR). After discussing a general CPR model, we study a specific model of fishery, proving the existence and uniqueness of solutions, and illustrating numerically the impact of the altruism level on the tragedy of the commons.

The rest of the paper is organized as follows. In Section~\ref{sec:model}, we introduce both models, for mixed individual and mixed population, without including common pool resources and we give the main results: approximate equilibrium and FBSDE characterization. In Section~\ref{sec:CPR}, we incldude CPRs and present corollaries of the main results in this setting. In Section~\ref{sec:fishing-model}, we study in detail an application to resource management for a fisheries model; in particular, we prove existence and uniqueness of solutions. In Section~\ref{sec:numerics}, we present numerical results and analyze how the new mixed models tackle the tragedy of the commons in this fisheries model.

\section{Mixed mean field models}
\label{sec:model}

In this section, we will introduce two mathematical models to incorporate both competition (i.e., non-cooperation) and cooperation to the problems. The first one is the mixed individual mean field model that assumes each agent has both cooperative and non-cooperative tendencies. The second one is the mixed population mean field model that assumes that the population is a mixture of cooperative and non-cooperative agents. After introducing these two models, we will also introduce how to model common pool resources in the presence of competition and cooperation in Section~\ref{sec:CPR}.

All the models will have finite time horizon where the time horizon is denoted by $T>0$. In the notations, we will use bold letters for functions of time (e.g., stochastic processes or flows of measures).

\subsection{Probabilistic background} 
\label{sec:background-proba}
The space of probability measures on $\mathbb{R}^d$ with a finite second moment is denoted by $\cP_2(\mathbb{R}^d)$. We endow this space with the 2-Wasserstein distance (see e.g.,~\cite[Chapter 5]{CarmonaDelarue_book_I}) where $p$-Wasserstein distance is defined as:
$$
    W_p(\mu,\mu') = \inf_{\pi} \left( \int_{\RR^d \times \RR^d} |x-x'|^p \pi(dx,dx')\right)^{1/p},
$$
where $p\geq1$ and the infimum is taken over $\pi \in \Pi(\mu,\mu')$, with $\Pi(\mu,\mu')$ denotes the set of all couplings of $\mu$ and $\mu'$. When $X$ is a random variable, notation $\mathcal{L}(X)$ is used to denote the law of $X$ and $\EE[X]$ is used to denote the expectation of $X$. The notation $\tilde{\EE}[\varphi(X, \tilde X)]$ means that the expectation is taken (only) over $\tilde{X}$, which is an independent copy of a random variable $X$. Moreover, for an integer $N$, we denote $[N] = \{1,\dots,N\}$. For $i \in [N]$ and for a vector $x$ of length $N$, $x^{-i} = (x^1,\dots,x^{i-1},x^{i+1}, \dots, x^N)$ is the vector of length $N-1$ in which we removed the $i$-th coordinate. When convenient, we will write $(\tilde{x}^i,x^{-i}) = (x^1,\dots,x^{i-1}, \tilde{x}^i, x^{i+1}, \dots, x^N)$.

We will consider a notion of derivatives with respect to measures sometimes referred to as the Lions derivative, which we denote by $\partial_\mu$. A function $U: \cP_2(\RR^d) \to \RR$ is differentiable if there exists a map $\partial_\mu U: \cP_2(\RR^d) \times \RR^d \to \RR$ such that for any $\mu,\mu' \in \cP_2(\RR^d)$, 
\[
\begin{aligned}
    &\lim_{s \to 0^+} \frac{U((1-s)\mu + s\mu') - U(\mu)}{s} 
%    \\
%    &
    = \int_{\RR^d} \partial_\mu U(\mu, x') d(\mu' - \mu)(x').
\end{aligned}
\]
We refer to e.g.,~\cite[Definition 2.2.1]{CardaliaguetDelarueLasryLions} or~\cite[Chapter 5]{CarmonaDelarue_book_I} for more details. 
A simple example is the linear case: when $U$ is of the form $U(\mu) = \int h(x) d\mu(x)$, which can also be written as $\EE[h(X)]$ where the random variable $X$ has distribution $\mu$. In this case, $\partial_\mu U(\mu)(x) = \partial h(x)$, $x \in \RR^d$; see~\cite[Example 1 in Section 5.2.2]{CarmonaDelarue_book_I}.

\subsection{Mixed individual mean field model}
\label{subsec:mi}

We first the setting where each individual has altruistic tendencies in the game model. In the mathematical formulation of the representative agent's model, this setup will manifest itself with the inclusion of two mean field terms. In this section, we start by introducing the finite population model, then we will introduce the limiting mean field model.

\subsubsection{Finite population model}\label{sec:mixed-individual-finite-pop}

We first consider a model with $N$ agents. Given the behavior of other agents, each agent's state dynamics and cost will involve not only the empirical distribution of the population but also her own state distribution. This means that each agent needs to solve a McKean-Vlasov control problem instead of a regular stochastic control problem given the other agents' behaviors.

We assume that the problem is set on a complete filtered probability space that is denoted by $(\Omega, \mathcal{F}, \mathbb{F} = (\mathcal{F}_t)_{t\in[0,T]}, \mathbb{P})$ supporting an $m$-dimensional Wiener process $\bW^i = (W^i_t)_{t\in[0,T]}$. Agent $i$ has a continuous $\RR^d$-valued state process $\bX^i = (X^i_t)_{t\in[0,T]}$ and $\RR^\ell$-valued control process $\balpha^i = (\alpha^i_t)_{t\in[0,T]}$. The control process is assumed to be square-integrable and $(\mathcal{F}_t)_{t\in[0,T]}$-progressively measurable. We denote the collection of such control processes by $\bbA$. We denote by $\mu^i_t = \cL(X^i_t)$ the distribution of the state of agent $i$. We denote by $\mu^N_t = 
\frac{1}{N} \sum_{j=1}^N \delta_{X^j_t}$ the empirical distribution of the population. Typically, the empirical distribution can be observed while the individual's distribution is not necessarily observable, but here (as usual in MFG literature) we assume complete information and perfect rationality, so the agent knows their own distribution and the way it impacts their dynamics and cost. We assume that the agents' states at initial time are i.i.d. with distribution $\mu_0$, i.e., $ X^{i}_0 = \zeta \in L^2(\Omega,\mathcal F_0, \mathbb{P}; \mathbb{R}^{{d}})$ where $\mu_0 = \mathcal{L}(\zeta)$. 

Assume $\bX^i = (X^i_t)_{t \in [0,T]}$ evolves according to the following state dynamics:
\begin{equation*}
    d X^{i}_t = b(t, X^{i}_t, \alpha_t^i, \mu^N_t, \mu_t^{i})dt + \sigma dW_t^i,
\end{equation*}
where the function $b : [0,T] \times \RR^d \times \RR^\ell \times \cP_2(\RR^d) \times \cP_2(\RR^d) \rightarrow \RR^d$ represents the drift of the state of the representative agent, $\sigma \in \RR^{d \times m}$ is the constant volatility, $\bW^i=(W_t^i)_{t\in[0,T]}$ is an $m$-dimensional Wiener process, independent of all the other Wiener processes, $\bW^{-i}$, which represents the idiosyncratic noise in the model. The model could be extended to more general settings (e.g., non-constant volatility which can be implemented as a function of time, state, and the mean field or interactions through the joint distribution of controls and states) in a straightforward way, but we leave such extensions for future work.

Notice that $\bX^i$ depends on the controls $\bualpha = (\balpha^1,\dots,\balpha^N)$ used by all the agents, since they influence the empirical distribution $\bmu^N$.  
To stress the dependence on the controls, we will sometimes write $\bX^{i, \bualpha}$. Similarly, we will sometimes write $\mu^N_t = \mu^{N,\bualpha}_t$ for the empirical distribution and $\mu^{i}_t = \mu^{i, \bualpha}_t=\cL(X^{i, \bualpha}_t)$ for the individual distribution.  We will refer to $\bualpha$ as a control profile. 

Given the controls $\bualpha^{-i}$ used by other agents, agent $i$'s goal is to minimize over $\balpha^i$ the cost:
\begin{equation}
\begin{aligned}
    &J^{\mixind,N}(\balpha^i;\bualpha^{-i}) 
    = \EE \Big[\int_0^T f(t, X_t^{i,\bualpha}, \alpha^i_t, \mu^{N,\bualpha}_t, \mu_t^{i, \bualpha})dt 
%    \\
%    &\hskip4cm 
    + g(X_T^{i,\bualpha}, \mu^{N,\bualpha}_T, \mu_T^{i, \bualpha})\Big].
\end{aligned}
\end{equation}
 Here, the function $f : [0,T] \times \RR^d \times \RR^\ell \times \cP_2(\RR^d) \times \cP_2(\RR^d) \rightarrow \RR$ represents the running cost that depends on the agent's own state (induced by the control profile) $X^{i,\bualpha}_t$, the agent's own action $\alpha^i_t$, the population's empirical distribution $\mu^{N, \bualpha}_t$ and the distribution of the agent's state $\mu_t^{i, \bualpha}$. The function $g : \RR^d \times \cP_2(\RR^d) \times \cP_2(\RR^d) \rightarrow \RR$ represents the terminal cost that depends on the same variables at terminal time $T$, except it does not depend on the control. %

Next, we define the notion of Nash equilibrium in the mixed individual finite population game. 

\begin{definition}%
\label{def:mixed_individual_finitepop_nash}
For $\epsilon>0$, an {\bf $\epsilon$-Nash equilibrium} for the $N$-agent mixed individual game is a control profile $\hat\bualpha$ such that, for every $i$ and every $\balpha^i$, 
\[
    J^{\mixind,N}(\hat\balpha^i;\hat\bualpha^{-i})
    \le 
    J^{\mixind,N}(\balpha^i;\hat\bualpha^{-i}) + \epsilon.
\]
A {\bf Nash equilibrium} is an $\epsilon$-Nash equilibrium with $\epsilon=0$. 
\end{definition}
Intuitively, replacing $\hat\alpha^i$ by $\alpha^i$ strongly affects the individual's distribution flow $\bmu^{i, \bualpha}$ but much less so the population's empirical distribution flow $\bmu^{N, \bualpha}$. In fact, when $N$ goes to infinity, we expect an individual's control to have zero influence on the population's distribution flow. This is the motivation behind the mixed individual mean field problem that is introduced next.

\subsubsection{Mean field model}
\label{sec:mixed-individual-mf}

Now, we let $N \rightarrow \infty$. Since every agent is assumed to be identical, we can focus on a representative agent. We assume that the problem is set on a complete filtered probability space that is denoted by $(\Omega, \mathcal{F}, \mathbb{F} = (\mathcal{F}_t)_{t\in[0,T]}, \mathbb{P})$ supporting a $m$-dimensional Wiener process $\bW = (W_t)_{t\in[0,T]}$. =

Given a population distribution flow $\bmu$, the aim of the representative agent is to minimize the following cost over her control process $\balpha\in \bbA$ where we remind that $\bbA$ denotes the set of control processes that are square-integrable and $(\mathcal{F}_t)_{t\in[0,T]}$-progressively measurable:
\begin{equation}
\begin{aligned}
    J(\balpha;\bmu) 
    &:= \EE \Big[\int_0^T f(t, X_t^{\balpha}, \alpha_t, \mu_t, \mu_t^{\balpha})dt
%    \\
%    &\qquad\qquad\qquad 
    + g(X_T^{\balpha}, \mu_T, \mu_T^{\balpha})\Big],
\end{aligned}
\end{equation}
where $\mu_t^{\balpha} = \cL(X_t^{\balpha})$ for all $t\in[0,T]$, and the running cost $f$ and the terminal cost $g$ are defined as in Section~\ref{sec:mixed-individual-finite-pop}. 
The representative 
agent's state dynamics $\bX^{\balpha}$ evolves according to the following dynamics:
\begin{equation}
    d X^{\balpha}_t = b(t, X^{\balpha}_t, \alpha_t, \mu_t, \mu_t^{\balpha})dt + \sigma dW_t,
\end{equation}
where the drift $b$ and the volatility $\sigma$ are defined as in Section~\ref{sec:mixed-individual-finite-pop}. %
The initial condition of the state process of the representative agent is given as $ X^{\balpha}_0 = \zeta \in L^2(\Omega,\mathcal F_0, \mathbb{P}; \mathbb{R}^{{d}})$ where $\mu_0 = \mu_0^{\balpha}= \mathcal{L}(\zeta)$.
\begin{definition}[MI-MFNE]\label{def:mixed_individual_mfg_nash} We will call $(\hat{\balpha}, \hat{\bmu})$ a {\bf mixed individual mean field Nash equilibrium} (MI-MFNE) if
\begin{itemize}
    \item[i.] $\hat \balpha$ is the best response of the representative agent given the mean field distribution $\hat{\bmu}$. In other words, $\hat{\balpha} \in \argmin_{\balpha\in \bbA} J(\balpha;\hat{\bmu})$,\vskip2mm
    \item[ii.] For all $t\in[0,T]$, we have $\hat{\mu}_t = \mu^{\hat{\balpha}}_t$.
\end{itemize}
\end{definition}
The first condition in Definition~\ref{def:mixed_individual_mfg_nash} requires us to solve a McKean-Vlasov stochastic optimal control problem given $\hat{\bmu}$ since the problem involves the distribution of the state $(\mu_t^{\balpha})_{t\in[0,T]} = \cL(X_t^{\balpha})_{t\in[0,T]}$ which is the main difference from the regular MFG problems. In the regular MFG, the first condition results in solving a regular stochastic optimal control problem. The second condition in Definition~\ref{def:mixed_individual_mfg_nash} is the fixed point condition that comes from the definition of Nash equilibrium and the underlying assumption of agents being identical.
\begin{remark}
    We would like to emphasize that this model is the generalized version of $\lambda$-interpolated mean field games that is introduced in~\citep{carmona2023nash} and analyzed in the linear-quadratic form in~\citep{dayanikli2024can}.  If the running cost, the terminal cost, and the drift have the following special forms in our model, we recover the $\lambda$-interpolated mean field games given in~\citep{carmona2023nash}:
    \begin{equation*}
    \begin{aligned}
        f(t, x, \alpha, \mu, \mu')&:= \frac{1}{2}\alpha^2 +  \lambda f_0(x, \alpha, \mu)
%        \\
%        &\qquad 
        +(1-\lambda) f_0(x, \alpha, \mu'),
        \\%[2mm]
        g(x, \mu, \mu')&:=\lambda g_0(x, \mu)+(1-\lambda) g_0(x, \mu'),
        \\%[2mm]
        b(t, x, \alpha, \mu, \mu') &:= \alpha.
    \end{aligned}
    \end{equation*}
    In this case $\lambda \in [0,1]$ is the level of selfishness, and $(1-\lambda)\in[0,1]$ is the level of altruism of the representative agent. When $\lambda=1$, the model corresponds to the fully game theoretical problem where the agents are non-cooperative (i.e., MFG). When $\lambda=0$, the model corresponds to the control problem where the agents are cooperative (i.e., MFC).
\end{remark}

\subsubsection{Approximate equilibrium result}
\label{subsubsec:app_eq_mi}

For simplicity in the presentation, we focus on the one-dimensional case ($d=m=\ell=1$).

\begin{theorem}[$\epsilon$-Nash for MI-MF]
\label{thm:MI-eps-Nash}
Under technical conditions stated in Appendix~\ref{app:MI-eps-Nash}, the following holds.
 Let $(\hat{\balpha}, \hat{\bmu})$ be a mixed individual mean field Nash equilibrium as in Definition~\ref{def:mixed_individual_mfg_nash}.  
  Then 
  for every $\epsilon>0$, there exists $N$ large enough such that $(\hat{\balpha}, \hat{\bmu})$ is an $\epsilon$-Nash equilibrium for the $N$-agent mixed individual game as in Definition~\ref{def:mixed_individual_finitepop_nash}. 
\end{theorem}
The proof is provided in Appendix~\ref{app:MI-eps-Nash}. 

When the number of agents increases, the finite agent problem will become very hard to solve. Informally, the above approximation result motivates us to focus on solving the MI-MF problem instead of its finite agent counterpart.

\subsubsection{Equilibrium characterization result}

We introduce the following Hamiltonian  for the mixed individual game: 
\[
\begin{aligned}
    &H^\mixind(t, x, \alpha, \mu, \mu', y, z) 
%    \\
%    &
    := b(t, x, \alpha, \mu, \mu') \cdot y 
     + f(t, x, \alpha, \mu, \mu' ) .
\end{aligned}
\]
Following Section~\ref{sec:background-proba}, we denote the Lions  derivative with respect to the second distribution by $\partial_{\mu'} H$.

\begin{theorem}[FBSDE for MI-MFNE]
\label{thm:MI-Pontryagin}
Under technical assumptions stated in Appendix~\ref{app:MI-Pontryagin}, the following holds. 
$(\hat{\balpha}, \hat{\bmu})$ is a mixed individual mean field Nash equilibrium as in Definition~\ref{def:mixed_individual_mfg_nash} if and only if it satisfies:
\[
\begin{cases}
    H^\mixind(t, X_t, \hat\alpha_t, \mu_t, \mu_t, Y_t, Z_t)
%    \\
%    \quad 
    = \inf_{\alpha} H^\mixind(t, X_t, \alpha, \mu_t, \mu_t, Y_t, Z_t)
    \\ 
    \hat\mu_t = \mu_t, 
\end{cases}
\]
where $\mu_t = \cL(X_t)$ and $(\bX, \bY, \bZ)$ solve the MKV FBSDE:
\begin{equation}
\label{eq:FBSDE-MI}
    \begin{cases}
        dX_t = b(t, X_t, \hat{\alpha}_t, \mu_t, \mu_t) dt + \sigma dW_t
        \\
        dY_t 
        = \Big(- \partial_x H^\mixind(t, X_t, \hat{\alpha}_t, \mu_t, \mu_t, Y_t, Z_t) 
%        \\
%        \qquad 
        - \tilde\EE\left[\partial_{\mu_2} H^\mixind(t, \tilde{X}_t, \tilde{\hat{\alpha}}_t, \mu_t, \mu_t, \tilde{Y}_t, \tilde{Z}_t)(X_t) \right] \Big) dt 
%        \\
%        \qquad 
        + Z_t dW_t
        \\
        X_0 \sim \mu_0, \qquad Y_T = g(X_T, \mu_T, \mu_T),
    \end{cases}
\end{equation}
where we recall that $\tilde\EE$ denotes an expectation on the tilded variables only.  
    
\end{theorem}
The proof is provided in Appendix~\ref{app:MI-Pontryagin}. 
We stress that, although at equilibrium the two distributions coincide and are equal to $\bmu$, the Hamiltonian's derivative is with respect to the second distribution only (i.e., the one corresponding to the individual's distribution). This is different from the FBSDE system arising in MFGs (see e.g. \cite[Theorem 3.17]{CarmonaDelarue_book_I}) and is reminiscent of the FBSDE system of MFC (see e.g. \cite[System (6.31)]{CarmonaDelarue_book_I}). 

Informally, the above result says that in order to find the MI-MFNE, we can focus on solving the FBSDE system given in~\eqref{eq:FBSDE-MI}.

\subsection{Mixed population mean field model}

Second, we introduce a model where there are two types of agents in the population, cooperative and non-cooperative. As we did in Section~\ref{subsec:mi}, we start by introducing the finite population model, then we will introduce the limiting mean field model. Since we have two types of agents (cooperative and non-cooperative), in the mean field model we will introduce the models of the \textit{representative} cooperative and non-cooperative agents separately.

\subsubsection{Finite population model}
\label{sec:mixed-population-finite-pop}

In the population, we consider two groups respectively with $N^{\rmnc}$ and $N^{\rmc}$ individuals. The former have a non-cooperative mindset while the latter are cooperative. We denote by $N = N^{\rmnc} + N^{\rmc}$ the total number of agents in the population.

In this model, each agent's state dynamics and cost will involve the empirical distribution of each group of agents. The non-cooperative agents will solve a Nash equilibrium (while taking into account the coalition of cooperative agents), while the cooperative agents will solve a social optimum problem (while taking into account the presence of non-cooperative agents).

We assume that the problem is set on a complete filtered probability space that is denoted by $(\Omega, \mathcal{F} = \mathcal{F}^{\rmnc} \cup \mathcal{F}^{\rmc}, \mathbb{F} = (\mathcal{F}_t^{\rmnc} \cup \mathcal{F}_t^{\rmc} )_{t\in[0,T]}, \mathbb{P})$ supporting an $m$-dimensional Wiener processes $\bW^{i, \rmnc} = (W^{i, \rmnc}_t)_{t\in[0,T]}$ and $\bW^{i, \rmc} = (W^{i, \rmc}_t)_{t\in[0,T]}$. 
Agent $i$ in the non-cooperative population with $i \in [N^\rmnc]$ has a continuous $\RR^d$-valued state process $\bX^{i, \rmnc} = (X^{i, \rmnc}_t)_{t\in[0,T]}$ and $\RR^\ell$-valued control process $\balpha^{i, \rmnc} = (\alpha^{i, \rmnc}_t)_{t\in[0,T]}$. The control process is assumed to be square-integrable and $(\mathcal{F}_t)_{t\in[0,T]}$-progressively measurable i.e., $\balpha^{i, \rmnc} \in \bbA$. 
We denote by $\mu^{N,\rmnc}_t = \frac{1}{N^\rmnc} 
\sum_{j=1}^{N^\rmnc} \delta_{X^{j,\rmnc}_t}$ the empirical distribution of the non-cooperative population. 
Similarly, for the corresponding concepts in the cooperative population, we use the notations $X^{i,\rmc}_t$, $\alpha^{i,\rmc}_t$, with $i \in [N^\rmc]$, and $\mu^{N,\rmc}_t = \frac{1}{N^\rmc} \sum_{j=1}^{N^\rmc} \delta_{X^{j,\rmc}_t}$. The superscript $N$ simply refers to the fact that we have a finite population, although the number of agents in each population could be different.

We assume that the non-cooperative and cooperative agents' states at the initial time are i.i.d. with respective distributions $\mu^\rmnc_0$ and $\mu^\rmc_0$. %
Furthermore, we assume $\bX^{i,\rmnc} = (X^{i,\rmnc}_t)_{t \in [0,T]}$ evolves according to the dynamics:
\[
\begin{aligned}
    d X^{i,\rmnc}_t 
    &= b^{\rmnc}(t, X^{i,\rmnc}_t, \alpha_t^{i,\rmnc}, \mu^{N,\rmnc}_t, \mu_t^{N,\rmc})dt 
%    \\
%    &\quad 
    + \sigma^{\rmnc} dW_t^{i,\rmnc},
\end{aligned}
\]
where $b^{\rmnc} : [0,T] \times \RR^d \times \RR^\ell \times \cP_2(\RR^d) \times \cP_2(\RR^d) \rightarrow \RR^d$ is the drift of the non-cooperative agent, and $\sigma^{\rmnc} \in \RR^{d \times m}$ is a constant volatility.

On the other hand, we assume $\bX^{i,\rmc} = (X^{i,\rmc}_t)_{t \in [0,T]}$ evolves according to the dynamics:
\begin{equation*}
    d X^{i,\rmc}_t = b^{\rmc}(t, X^{i,\rmc}_t, \alpha_t^{i,\rmc}, \mu^{N,\rmnc}_t, \mu_t^{i,\rmc})dt + \sigma^{\rmc} dW_t^{i,\rmc},
\end{equation*}
where $\mu_t^{i,\rmc}=\cL(X_t^{i,\rmc})$. We want to emphasize that the non-cooperative agents observe the mean fields of both non-cooperative and cooperative populations without seeing their own effect on the population directly. However, cooperative agents observe the mean field of the non-cooperative agents and their own effect on the population. 

In the dynamics, we assume that the Brownian motions $\bW^{1,\rmnc}, \dots, $ $ \bW^{N^\rmnc,\rmnc}$, $\bW^{1,\rmc}, \dots, \bW^{N^\rmc,\rmc}$ are a family of independent $m$-dimensional Wiener processes. They represent the idiosyncratic noises in the model for non-cooperative and cooperative agents, respectively. Here, similar to Section~\ref{sec:mixed-individual-finite-pop}, the model could be extended to more general settings, but we leave such extensions for future work.

Notice that the state of one agent depends on the controls used by all the agents (in both groups).   
To stress the dependence on the controls, we will sometimes write $X^{i, \bualpha^{\rmnc}, \bualpha^{\rmc}}$. Likewise, we will sometimes write $\mu^{N,\rmnc}_t = \mu^{N,\rmnc,\bualpha^{\rmnc}, \bualpha^{\rmc}}_t$, $\mu^{N,\rmc}_t = \mu^{N,\rmc,\bualpha^{\rmnc}, \bualpha^{\rmc}}_t$, and $\mu_t^{i, C} = \mu^{i,\rmc,\bualpha^{\rmnc}, \bualpha^{\rmc}}_t$ for the empirical distributions of non-cooperative and cooperative populations, and the state distribution of the cooperative agent $i$ with $i\in[N^{\rmnc}]$, respectively.

We first consider the objective of the non-cooperative agents. 
Given the controls used by other agents, $\bualpha^{-i,\rmnc}$ and $\bualpha^{\rmc}$, non-cooperative agent $i$'s goal is to minimize over $\balpha^{i,\rmnc}$ the cost:
\[
\begin{aligned}
    &J^{\mixpop,N,\rmnc}(\balpha^{i,\rmnc};\bualpha^{-i,\rmnc}, \bualpha^{\rmc}) 
%    \\
%    &
    = \EE \Big[\int_0^T f^{\rmnc}(t, X_t^{i,\rmnc,\bualpha^{\rmnc}, \bualpha^{\rmc}}, \alpha^{i,\rmnc}_t, 
%    \\
%    &\qquad\qquad\qquad 
    \mu^{N,\rmnc,
    \bualpha^{\rmnc}, \bualpha^{\rmc}}_t, \mu_t^{N,\rmc, \bualpha^{\rmnc}, \bualpha^{\rmc}})dt 
    \\
    &\qquad \hskip5cm+ g^{\rmnc}(X_T^{i,\rmnc,\rmnc,\bualpha^{\rmnc}, \bualpha^{\rmc}}, 
%    \\
%    &\qquad\qquad\qquad 
    \mu^{N,\rmnc,\bualpha^{\rmnc}, \bualpha^{\rmc}}_T, \mu_T^{N,\rmc, \bualpha^{\rmnc}, \bualpha^{\rmc}})\Big].
\end{aligned}
\]
 Here the function $f^{\rmnc} : [0,T] \times \RR^d \times \RR^\ell \times \cP_2(\RR^d) \times \cP_2(\RR^d) \rightarrow \RR$ represents the running cost and $g^{\rmnc} : \RR^d \times \cP_2(\RR^d) \times \cP_2(\RR^d) \rightarrow \RR$ represents the terminal cost. They depend on the agent's own state and action (for the running cost), as well as both population distributions.

We then define the cost for the cooperative population. Given the control profile $\bualpha^{\rmnc}$ of the non-cooperative agents, the cooperative agents try to minimize the social cost for their group by choosing the control profile $\bualpha^{\rmc}$. Therefore, we do not define the individual cost and instead focus on the cost averaged over the cooperative population.  We define:
\[
\begin{aligned}
    &J^{\mixpop,N,\rmc}(\bualpha^{\rmc};\bualpha^{\rmnc}) 
%    \\
%    &
    = \frac{1}{N^\rmc} \sum_{i=1}^{N^\rmc} 
    \EE \Big[\int_0^T f^{\rmc}(t, X_t^{i,\rmc, \bualpha^{\rmnc}, \bualpha^{\rmc}}, \alpha^{i,\rmc}_t, 
%    \\
%    &\qquad\qquad\qquad\qquad
     \mu^{N,\rmnc, \bualpha^{\rmnc}, \bualpha^{\rmc}}_t, \mu_t^{i,\rmc, \bualpha^{\rmnc}, \bualpha^{\rmc}})dt 
  \\
&\qquad\hskip5cm
     + g^{\rmc}(X_T^{i,\rmc,\bualpha^{\rmnc}, \bualpha^{\rmc}}, \mu^{N,\rmnc,\bualpha^{\rmnc}, \bualpha^{\rmc}}_T, \mu_T^{i,\rmc, \bualpha^{\rmnc}, \bualpha^{\rmc}})\Big],
\end{aligned}
\]
where $\mu_t^{i,\rmc, \bualpha^{\rmnc}, \bualpha^{\rmc}} = \cL(X_t^{i,\rmc, \bualpha^{\rmnc}, \bualpha^{\rmc}})$, which shows the fact that each cooperative agent $i$ is aware of their own effect in the population.

With these notations, the notion of equilibrium is defined as follows. 
\begin{definition}
\label{def:mixed_population_finitepop_nash}
For $\epsilon^\rmnc \ge 0$ and $\epsilon^\rmc \ge 0$, an {\bf $(\epsilon^\rmnc, \epsilon^\rmc)$-equilibrium} for the mixed population game is a pair $(\hat\bualpha^\rmnc, \hat\bualpha^\rmc)$ such that:
\begin{itemize}
    \item[i.] for every $i \in [N^\rmnc]$ and every control $\balpha^{i,\rmnc}$, 
    \[
    \begin{aligned}
        &J^{\mixpop,N,\rmnc}(\hat\balpha^{i,\rmnc};\hat\bualpha^{-i,\rmnc}, \hat\bualpha^{\rmc})
%        \\
%        &
        \le 
        J^{\mixpop,N,\rmnc}(\balpha^{i,\rmnc};\hat\bualpha^{-i,\rmnc}, \hat\bualpha^{\rmc}) + \epsilon^\rmnc;
    \end{aligned}
    \]
    \item[ii.] for every $\bualpha^{\rmc} = (\bualpha^{1, \rmc}, \dots, \bualpha^{N^\rmc, \rmc})$, 
    \[
        J^{\mixpop,N,\rmc}(\hat\bualpha^{\rmc}; \hat\bualpha^{\rmnc})
        \le 
        J^{\mixpop,N,\rmc}(\bualpha^{\rmc}; \hat\bualpha^{\rmnc}) + \epsilon^\rmc.
    \]
\end{itemize}
An {\bf equilibrium} is an $(\epsilon^\rmnc,\epsilon^\rmc)$-Nash equilibrium with $(\epsilon^\rmnc,\epsilon^\rmc)=(0,0)$. 
\end{definition}

\subsubsection{Mean field model}
\label{sec:mixed-population-mf}

Now, we let informally $N^{\rmnc} \rightarrow \infty$ and $N^{\rmc} \rightarrow \infty$. Since all the cooperative and non-cooperative agents are identical within their groups, we can focus on representative non-cooperative and representative cooperative agents. We assume that the problem is set on a complete filtered probability space that is denoted by $(\Omega, \mathcal{F} = \mathcal{F}^{\rmnc} \cup \mathcal{F}^{\rmc}, \mathbb{F} = (\mathcal{F}_t^{\rmnc} \cup \mathcal{F}_t^{\rmc} )_{t\in[0,T]}, \mathbb{P})$ supporting an $m$-dimensional Wiener processes $\bW^{i, \rmnc} = (W^{\rmnc}_t)_{t\in[0,T]}$ and $\bW^{i, \rmc} = (W^{\rmc}_t)_{t\in[0,T]}$.

\noindent\textbf{Non-cooperative agents.} The aim of the representative non-cooperative agent is to minimize the following cost over her control process $\balpha^{\rmnc}\in \bbA$:
\begin{equation}
    \begin{aligned}
    &J^{\rmnc}(\balpha^{\rmnc};\bmu^{\rmnc},\bmu^{\rmc})
%    \notag\\
%    &
    := \EE \Big[\int_0^T f^{\rmnc}(t, X_t^{\balpha^{\rmnc}}, \alpha^{\rmnc}_t, \mu^{\rmnc}_t, \mu_t^{\rmc})dt 
%    \notag\\
%    &\hskip2.7cm
    + g^{\rmnc}(X_T^{\balpha^{\rmnc}}, \mu^{\rmnc}_T, \mu_T^{\rmc})\Big].
    \end{aligned}
    \label{eq:mixedpop-mfg-cost-NC}
\end{equation}
Here, as introduced in Section~\ref{sec:mixed-individual-finite-pop}, the function $f^{\rmnc} : [0,T] \times \RR^d \times \RR^\ell \times \cP_2(\RR^d) \times \cP_2(\RR^d) \rightarrow \RR$ represents the running cost of the representative non-cooperative agent that depends on the non-cooperative agent's own state (induced by her own control) $X^{\balpha^{\rmnc}}_t$, own control $\alpha^{\rmnc}_t$, the mean field distribution of the non-cooperative agents $\mu^{\rmnc}_t$, and the mean field distribution of cooperative agents $\mu_t^{\rmc}$. The function $g^{\rmnc} : \RR^d \times \cP_2(\RR^d) \times \cP_2(\RR^d) \rightarrow \RR$ represents the terminal cost of the non-cooperative agent and depends on the same inputs at the terminal time except the control of the non-cooperative agent.

The state of the representative non-cooperative agent, $\bX^{\balpha^{\rmnc}}$, evolves according to the following dynamics:
\begin{equation*}
    d X^{\balpha^{\rmnc}}_t = b^{\rmnc}(t, X^{\balpha^{\rmnc}}_t, \alpha^{\rmnc}_t, \mu^{\rmnc}_t, \mu_t^{\rmc})dt + \sigma^\rmnc dW^{\rmnc}_t,
\end{equation*}
where the function $b^{\rmnc} : [0,T] \times \RR^d \times \RR^\ell \times \cP_2(\RR^d) \times \cP_2(\RR^d) \rightarrow \RR^d$ represents the drift of the state of the representative non-cooperative agent and $\sigma^\rmnc \in \RR^{d \times m}$ is the constant volatility. The $m$-dimensional Wiener process $\bW^{\rmnc}$ represents the idiosyncratic noise of the non-cooperative agents. 
\vskip6pt
\noindent\textbf{Cooperative agents.} The aim of the representative cooperative agent is to minimize the following cost over her control process $\balpha^{\rmc}\in \bbA$:
\[
\begin{aligned}
    J^{\rmc}(\balpha^{\rmc};\bmu^{\rmnc})
    &:= \EE \Big[\int_0^T f^{\rmc}(t, X_t^{\balpha^{\rmc}}, \alpha^{\rmc}_t, \mu^{\rmnc}_t, \mu_t^{\balpha^{\rmc}})dt 
%    \\
%    &\qquad\qquad
    + g^{\rmc}(X_T^{\balpha^{\rmnc}}, \mu^{\rmnc}_T, \mu_T^{\balpha^{\rmc}})\Big],
\end{aligned}
\]
where $\mu_t^{\balpha^{\rmc}}= \cL(X_t^{\balpha^{\rmc}})$ for all $t \in [0,T]$. As introduced in Section~\ref{sec:mixed-individual-finite-pop}, the function $f^{\rmc} : [0,T] \times \RR^d \times \RR^\ell \times \cP_2(\RR^d) \times \cP_2(\RR^d) \rightarrow \RR$ represents the running cost of the representative cooperative agent and it depends on the cooperative agent's own state $X^{\balpha^{\rmc}}_t$, own control $\alpha^{\rmc}_t$, the mean field distribution of the non-cooperative agents $\mu^{\rmnc}_t$, and the distribution of the cooperative agent's own state $\mu_t^{\balpha^{\rmc}}$. The function $g^{\rmc} : \RR^d \times \cP_2(\RR^d) \times \cP_2(\RR^d) \rightarrow \RR$ represents the terminal cost of the cooperative agent and depends on the same inputs at the terminal time except the control of the cooperative agent.

The state of the representative cooperative agent, $\bX^{\balpha^{\rmc}}$, evolves according to the following dynamics:
\begin{equation*}
    d X^{\balpha^{\rmc}}_t = b^{\rmc}(t, X^{\balpha^{\rmc}}_t, \alpha^{\rmc}_t, \mu^{\rmnc}_t, \mu_t^{\balpha^{\rmc}})dt + \sigma^\rmc dW^{\rmc}_t,
\end{equation*}
where the function $b^{\rmc} : [0,T] \times \RR^d \times \RR^\ell \times \cP_2(\RR^d) \times \cP_2(\RR^d) \rightarrow \RR^d$ represents the drift of the state of the representative cooperative agent and $\sigma^\rmc \in \RR^{d \times m}$ is the constant volatility. The $m$-dimensional Wiener process $\bW^{\rmc}$ represents the idiosyncratic noise of the cooperative agents. For the sake of simplicity, we assume Wiener processes $\bW^{\rmnc}$ and $\bW^{\rmc}$ are independent. 

\vskip6pt
\begin{definition}[MP-MFE]\label{def:mixedpop_mfg_nash}
We call $(\hat{\balpha}^{\rmnc}, \hat{\balpha}^{\rmc}, \hat{\bmu}^{\rmnc})$ a {\bf mixed population mean field equilibrium} (MP-MFE) if
    \begin{itemize}
        \item[i.] $\hat{\balpha}^{\rmnc}$ is the best response of the representative non-cooperative agent given the mean field distributions of the non-cooperative and cooperative agents, $\hat{\bmu}^{\rmnc}$ and ${\bmu}^{\hat{\balpha}^{\rmc}}$, respectively. In other words, $\hat{\balpha}^{\rmnc} \in \argmin_{\balpha^{\rmnc}\in \bbA} J^{\rmnc}(\balpha^{\rmnc};\hat{\bmu}^{\rmnc},{\bmu}^{\hat{\balpha}^{\rmc}})$,\vskip2mm
        \item[ii.] For all $t\in[0,T]$, we have $\hat{\mu}_t^{\rmnc} = \mu^{\hat{\balpha}^{\rmnc}}_t$,\vskip2mm
        \item[iii.] $\hat{\balpha}^{\rmc}$ is the optimal control of the representative cooperative agent given the mean field distribution of the non-cooperative agents, $\hat\bmu^{\rmnc}$. In other words, $\hat{\balpha}^{\rmc} \in \argmin_{\balpha^{\rmc}\in \bbA} J^{\rmc}(\balpha^{\rmc};\hat{\bmu}^{\rmnc})$.
    \end{itemize}
\end{definition}

    An interesting version of the mixed population mean field model is the case where we analyze a population that consists of a $p$ proportion of non-cooperative agents and a $(1-p)$ proportion of cooperative agents. In this case, we can assume that the agents are interacting with a mixture of population distribution in the following way: 
    \begin{equation*}
        \begin{aligned}
            &f^{\rmnc}(t, x, a, \mu, \mu') 
            := \cf^{\rmnc}(t, x, a, p \mu + (1-p)\mu')
            \\
            &f^{\rmc}(t, x, a, \mu, \mu')
            := \cf^{\rmc}(t, x, a, p\mu+(1-p) \mu'),  
        \end{aligned}
    \end{equation*}
    where $\mu$ and $\mu'$ play respectively the role of $\mu_t^{\rmnc}$ and $\mu_t^{\rmc}$.  $g^{\rmnc}, g^{\rmc}, b^{\rmnc}, b^{\rmc}$ can be defined similarly by using functions $\cg^{\rmnc}(x, p \mu + (1-p)\mu'$, $\cg^{\rmc}(x, p\mu+(1-p) \mu')$, $\cb^{\rmnc}(t, x, z, p \mu + (1-p)\mu')$, and $\cb^{\rmc}(t, x, \alpha^{\rmc}_t, p\mu+(1-p) \mu')$, respectively. An analysis of a linear-quadratic model in this form can be found in~\citep{dayanikli2024can}.
\vskip6pt
    \begin{remark}
    The version with a population $p$ proportion of non-cooperative and $(1-p)$ proportion of cooperative agents is similar to the $p$-partial MFG model introduced by~\cite{carmona2023nash}. However, there is an important difference. 
    In $p$-partial mean field game that is introduced by~\cite{carmona2023nash}, the cooperative agents do not optimize their own objective under the knowledge of the existence of non-cooperative agents in the population. Instead, they follow the social optimum that is found by assuming as if everyone in the population is cooperative. In a way, in the mixed population mean field model, the cooperative agents are informed about the population mixture and actively optimizing their behavior by taking into account this information.
    \end{remark}

\subsubsection{Approximate equilibrium result}

\begin{theorem}[$\epsilon$-equilibrium for MP-MFE]
\label{thm:MP-eps-Nash}
Under technical conditions stated in Appendix~\ref{app:MP-eps-Nash}, the following holds. Let $(\hat{\balpha}^{\rmnc}, \hat{\balpha}^{\rmc}, \hat{\bmu}^{\rmnc})$ be a mixed population mean field equilibrium as in Definition~\ref{def:mixedpop_mfg_nash}. 
For every positive $\epsilon^\rmnc$ and $\epsilon^\rmc$, for large enough $N^\rmnc$ and $N^\rmc$,  $(\hat{\balpha}^{\rmnc}, \hat{\balpha}^{\rmc}, \hat{\bmu}^{\rmnc})$ is an $(\epsilon^\rmnc, \epsilon^\rmc)$-equilibrium for the $(N^\rmnc,N^\rmc)$-agent mixed population game as in Definition~\ref{def:mixed_population_finitepop_nash}.
\end{theorem}
The proof is postponed to Appendix~\ref{app:MP-eps-Nash}. 

Similar to the discussion in Section~\ref{subsubsec:app_eq_mi}, when the number of agents increases, the finite agent problem will become increasingly harder to solve. Informally, this approximation result motivates us to focus on solving the mixed population mean field problem instead of its finite agent counterpart.

\subsubsection{Equilibrium characterization result}

For $\pi \in \{\rmc,\rmnc\}$, we introduce the Hamiltonian:
\[
\begin{aligned}
    &H^{\mixpop, \pi}(t, x, \alpha, \mu, \mu', y, z) 
%    \\
%    &
    = b^\pi(t, x, \alpha, \mu, \mu') \cdot y
    + f^\pi(t, x, \alpha, \mu, \mu' ).
\end{aligned}
\]

\begin{theorem}[FBSDE for MP-MFE]
\label{thm:MP-Pontryagin}
Under technical conditions stated in Appendix~\ref{app:MP-fbsde}, the following holds.
    $(\hat{\balpha}^{\rmnc}, \hat{\bmu}^{\rmnc}, \hat{\balpha}^{\rmc})$ is a mixed population mean field equilibrium as in Definition~\ref{def:mixedpop_mfg_nash}. 
    if and only if it satisfies:
    \begin{equation*}
    \begin{cases}
        H^{\mixpop, \rmnc}(t, X^\rmnc_t, \hat{\alpha}^\rmnc_t , \mu^\rmnc_t, \mu^\rmc_t, Y^\rmnc_t, Z^\rmnc_t) 
%        \\
%        \quad 
        = \inf_\alpha H^{\mixpop, \rmnc}(t, X^\rmnc_t, \alpha, \mu^\rmnc_t, \mu^\rmc_t, Y^\rmnc_t, Z^\rmnc_t),
        \\
        H^{\mixpop, \rmc}(t, X^\rmc_t, \hat{\alpha}^\rmc_t , \mu^\rmnc_t, \mu^\rmc_t, Y^\rmc_t, Z^\rmc_t)
%        \\
%        \quad 
        = \inf_\alpha H^{\mixpop, \rmc}(t, X^\rmc_t, \alpha, \mu^\rmnc_t, \mu^\rmc_t, Y^\rmc_t, Z^\rmc_t), 
        \\
        \hat\mu^\rmnc_t 
        = \mu^\rmnc_t,
    \end{cases}
    \end{equation*}
    where $\mu^\rmnc_t = \cL(X^\rmnc_t)$, $\mu^\rmc_t = \cL(X^\rmc_t)$, and $(\bX^\rmnc, \bY^\rmnc, \bZ^\rmnc, \bX^\rmc, \bY^\rmc, \bZ^\rmc)$ solve the coupled MKV FBSDE:
\begin{equation}
\label{eq:FBSDE-MP}
    \begin{cases}
        dX^\rmnc_t = b^\rmnc(t, X^\rmnc_t, \hat{\alpha}^\rmnc_t, \mu^\rmnc_t, \mu^\rmc_t) dt + \sigma^\rmnc dW^\rmnc_t
        \\[2mm]
        dY^\rmnc_t = - \partial_x H^{\mixpop,\rmnc}(t, X^\rmnc_t, \hat{\alpha}^\rmnc_t, \mu^\rmnc_t, \mu^\rmc_t, 
%        \\
%        \qquad\qquad\qquad 
        Y^\rmnc_t, Z^\rmnc_t) dt + Z^\rmnc_t dW^\rmnc_t
        \\[2mm]
        X^\rmnc_0 \sim \mu^\rmnc_0, \qquad Y^\rmnc_T = g^{\rmc}(X^\rmnc_T, \mu^\rmnc_T, \mu^\rmc_T),
        \\[2mm]
        dX^\rmc_t = b^\rmc(t, X^\rmc_t, \hat{\alpha}^\rmc_t, \mu^\rmnc_t, \mu^\rmc_t) dt + \sigma^\rmc dW^\rmc_t
        \\[2mm]
        dY^\rmc_t = \Big(- \partial_x H^{\mixpop,\rmc}(t, X^\rmc_t, \hat{\alpha}^\rmc_t, \mu^\rmnc_t, \mu^\rmc_t, Y^\rmc_t, Z^\rmc_t) 
%        \\
%        \qquad 
        - \tilde\EE\Big[\partial_{\mu^\rmc} H^{\mixpop,\rmc}(t, \tilde{X}^\rmc_t, \tilde{\hat{\alpha}}^\rmc_t, \mu^\rmnc_t, \mu^\rmc_t,
%        \\
%        \qquad\qquad\qquad 
         \tilde{Y}^\rmc_t, \tilde{Z}^\rmc_t)(X^\rmc_t) \Big] \Big) dt 
        + Z^\rmc_t dW^\rmc_t
        \\[2mm]
        X^\rmc_0 \sim \mu^\rmc_0, \qquad Y^\rmc_T = g^{\rmc}(X^\rmc_T, \mu^\rmnc_T, \mu^\rmc_T),
    \end{cases}
\end{equation}
where we recall that $\tilde\EE$ denotes an expectation on the tilded variables only.  
\end{theorem}

The proof is provided in Appendix~\ref{app:MP-fbsde}.

Informally, the above result says that in order to find the MI-MFNE, we can focus on solving the FBSDE system given in~\eqref{eq:FBSDE-MP}. Different than the mixed individual mean field model, now we have an FBSDE system for non-cooperative and cooperative agents, and due to their interactions through the population distributions, these systems are coupled.

\section{Mixed mean field models with common pool resources}
\label{sec:CPR}

The tragedy of the commons may occur when a common-pool resource is over-used and becomes extinct. This occurs when individual and population interests have a conflict which creates an individual profit through over-consumption in the short term, but in the long term creates negative outcomes when the common-pool resource becomes extinct or depleted below the management concern levels. However, as it is discussed in Section~\ref{sec:intro}, individual efforts and altruistic behaviors may contribute to the sustainability of CPRs.
Motivated by modeling CPRs and the effect of altruism in their sustainability, we introduce mixed individual and mixed population mean field models with CPRs.

\subsection{Mixed individual mean field model with common pool resources} 
\label{subsec:mi_CPR}

We assume that there are $k\in\mathbb{N}_+$ different types of common-pool resources. The state of the CPRs (i.e., the remaining amount of the stock) is a continuous $\RR^k_+$-valued process and is denoted by $\bK = (K_t)_{t\in[0,T]}$. It evolves according to the following dynamics:
\begin{equation}
    d K_t  = b^K(t, K_t, \mu_t, \mu_t^{\balpha})dt, 
\end{equation}
where $K_0=k_0\in\RR^k_+$. The drift is assumed to be of the form $b^K(t, k, \mu,\mu') := b^K_0(t, k) - b^K_1(t, k, \mu, \mu')$, where the function $b_0^K : [0,T] \times \RR_+^k \rightarrow \RR^k$ represents the reproduction or replenishment rate of CPRs and the function $b_1^K : [0,T] \times \RR_+^k \times \cP(\RR^d) \times \cP(\RR^d) \rightarrow \RR^k$ represents the CPR consumption rate.

With the introduction of CPRs, the goal of the representative agent in the \textit{mixed individual mean field model} is: given $\bmu$, find 
\begin{equation*}
\begin{aligned}
    &\inf_{\balpha \in \bbA} J^K(\balpha;\bmu)
%    \\
%    &
    := \inf_{\balpha \in \bbA} \EE \Big[\int_0^T f^K(t, X_t^{\balpha}, \alpha_t, \mu_t, \mu_t^{\balpha}, K_t)dt 
%    \\
%    &\hskip3cm 
    + g^K(X_T^{\balpha}, \mu_T, \mu_T^{\balpha}, K_T)\Big],
\end{aligned}
\end{equation*}
where $\mu_t^{\balpha} = \cL(X_t^{\balpha})$ for all $t\in[0,T]$ and the dynamics of individual state $\bX^{\balpha}$ and CPR state $\bK$ are:
\begin{equation*}
    \begin{aligned}
        d X^{\balpha}_t &= b(t, X^{\balpha}_t, \alpha_t, \mu_t, \mu_t^{\balpha}, K_t)dt + \sigma dW_t,\\
        d K_t &= \big(b^K_0(t, K_t) - b^K_1(t, K_t, \mu_t, \mu_t^{\balpha})\big)dt.
    \end{aligned}
\end{equation*}

We can then introduce the extended state $\tilde{X}^{\balpha}_t = (X^{\balpha}_t, K_t)$, the mean fields $\tilde{\mu}_t = \mu_t \otimes \delta_{K_t}$ and $\tilde{\mu}^{\balpha}_t = \mu^{\balpha}_t \otimes \delta_{K_t}$, and the functions 
\begin{equation*}
\begin{aligned}
    \tilde{f}(t, (x,k), \alpha, \tilde\mu, \tilde\mu') &= f^K(t, x, \alpha, \mu, \mu', k),\\
    \tilde{g}((x,k), \tilde\mu, \tilde\mu') &= g^K(x, \mu, \mu', k)\\
        \tilde{b}(t,(x,k), \alpha, \tilde\mu, \tilde\mu') &= \begin{bmatrix}
        b(t, x, \alpha, \mu, \mu', k)\\ b^K(t,k,\mu,\mu')
    \end{bmatrix},\\
    \tilde{\sigma} = [\sigma, 0]^\top, \qquad \tilde{W}_t &= W_t, 
\end{aligned}
\end{equation*}
where $\mu, \mu'$ are the first marginals of $\tilde\mu, \tilde\mu'$, 
such that the problem rewrites: 
\begin{equation*}
\begin{aligned}
    &\inf_{\balpha \in \bbA} \tilde{J}(\balpha;\tilde\bmu )
    := \inf_{\balpha \in \bbA} \EE \Big[\int_0^T \tilde{f}(t, \tilde{X}_t^{\balpha}, \alpha_t, \tilde{\mu}_t, \tilde{\mu}_t^{\balpha})dt 
%    \\
%    &\hskip3.5cm 
    + \tilde{g}(\tilde{X}_T^{\balpha}, \tilde{\mu}_T, \tilde{\mu}_T^{\balpha})\Big],
\end{aligned}
\end{equation*}
where the dynamics of $\tilde\bX^{\balpha}$ are given as follows:
\begin{equation*}
    \begin{aligned}
        d \tilde{X}^{\balpha}_t &= \tilde{b}(t, \tilde{X}^{\balpha}_t, \alpha_t, \tilde{\mu}_t, \tilde{\mu}_t^{\balpha})dt + \tilde{\sigma} d\tilde{W}_t.
    \end{aligned}
\end{equation*}
This problem is now a mixed individual mean field model as presented earlier in Section~\ref{sec:mixed-individual-mf}, and the same theoretical results can be extended to this setting.

\subsection{Mixed population mean field model with common pool resources} 

In the \textit{mixed population mean field model}, the dynamics of the common-pool resource state will be perceived differently by the non-cooperative and cooperative agents to emphasize that non-cooperative agents are not aware of their individual effect on the system while the cooperative agents are. The CPR dynamics perceived by the representative non-cooperative and cooperative agents will be denoted with $\bK^{\rmnc}$ and $\bK^{\rmc}$, respectively and evolve according to the following dynamics:
\begin{equation}
\begin{aligned}
    d K^{\rmnc}_t  &= b^{K}(t, K_t, \mu^{\rmnc}_t, \mu^{\rmc}_t)dt,\\
    d K^{\rmc}_t  &= b^{K}(t, K_t, \mu^{\rmnc}_t, \mu^{\balpha^{\rmc}}_t)dt,
\end{aligned}
\end{equation}
where $K^{\rmnc}_0=k_0\in\RR^k_+$. In this model, as it can be seen from the dynamics above for the non-cooperative agents, the perceived common-pool resource dynamics include the mean field distribution of both non-cooperative and cooperative agent populations exogenously. On the other hand, the cooperative agents take the mean field of the non-cooperative agent population exogenously but they see their own affect on the common pool resources with the inclusion of $\bmu^{\balpha^{\rmc}}$. Similar to the discussion before, for motivational examples, the drift form can be assumed as $b^{K}(t, k, \mu, \mu') := b^K_0(t, k) - b^{K}_1(t, k, \mu, \mu')$. Then the function $b_1^{K} : [0,T] \times \RR_+^k \times \cP(\RR^d)\times \cP(\RR^d) \rightarrow \RR^k$ represents the CPR consumption rate.

With the introduction of CPRs, the goal of the representative \textbf{non-cooperative} agent in the \textit{mixed population mean field model} is: given $\bmu^{\rmnc}, \bmu^{\rmc}$ find 
\begin{equation*}
\begin{aligned}
    &\inf_{\balpha^\rmnc \in \bbA} J^{K,\rmnc}(\balpha^\rmnc;\bmu^{\rmnc},\bmu^{\rmc})\\
    &:= \inf_{\balpha^{\rmnc} \in \bbA} \EE \Big[\int_0^T f^{K,\rmnc}(t, X_t^{\balpha^{\rmnc}}, \alpha^{\rmnc}_t, 
    \mu^{\rmnc}_t, \mu_t^{\rmc}, K^{\rmnc}_t)dt 
%    \\
%    &\hskip3cm 
    + g^{K,\rmnc}(X_T^{\balpha^{\rmnc}}, \mu^{\rmnc}_T, \mu_T^{\rmc}, K^{\rmnc}_T)\Big],
\end{aligned}
\end{equation*}
where the dynamics of individual state of representative non-cooperative agent $\bX^{\balpha^{\rmnc}}$ and CPR state $\bK^{\rmnc}$ are:
\begin{equation*}
    \begin{aligned}
        d X^{\balpha^{\rmnc}}_t 
        &= b^{\rmnc}(t, X^{\balpha^{\rmnc}}_t, \alpha^{\rmnc}_t, \mu^{\rmnc}_t, \mu_t^{\rmc}, K^{\rmnc}_t)dt 
        + \sigma^{\rmnc} dW^{\rmnc}_t,
        \\
        d K^{\rmnc}_t 
        &= \big(b^K_0(t, K^{\rmnc}_t) - b^{K}_1(t, K^{\rmnc}_t, \mu^{\rmnc}_t, \mu^{\rmc}_t)\big)dt.
    \end{aligned}
\end{equation*}
On the other hand, with the introduction of CPRs, the goal of the representative \textbf{cooperative} agent in the \textit{mixed population mean field model} is: given $\bmu^{\rmnc}$
\begin{equation}
\begin{aligned}
    &\inf_{\balpha^{\rmc} \in \bbA} J^{K,\rmc}(\balpha^{\rmc};\bmu^{\rmnc})
    \\
    &:= \inf_{\balpha^{\rmc} \in \bbA} \EE \Big[\int_0^T f^{K,\rmc}(t, X_t^{\balpha^{\rmc}}, \alpha^{\rmc}_t, \mu^{\rmnc}_t, \mu_t^{\balpha^{\rmc}}, K^{\rmc}_t)dt 
%    \\
%    &\hskip3cm 
    + g^{K,\rmc}(X_T^{\balpha^{\rmc}}, \mu^{\rmnc}_T, \mu_T^{\balpha^{\rmc}}, K^{\rmc}_T)\Big],
\end{aligned}
\end{equation}
where the dynamics of individual state of representative cooperative agent $\bX^{\balpha^{\rmc}}$ and CPR state $\bK$ are given as follows:
\begin{equation*}
    \begin{aligned}
        d X^{\balpha^{\rmc}}_t 
        &= b^{\rmc}(t, X^{\balpha^{\rmc}}_t, \alpha_t, \mu^{\rmnc}_t, \mu_t^{\balpha^{\rmc}}, K^{\rmc}_t)dt  + \sigma^{\rmc} dW^{\rmc}_t,
        \\
        d K^{\rmc}_t 
        &= \big(b^K_0(t, K^{\rmc}_t) - b^{K}_1(t, K^{\rmc}_t, \mu^{\rmnc}_t, \mu_t^{\balpha^{\rmc}})\big)dt,
    \end{aligned}
\end{equation*}
where $K^{\rmc}_0=k_0\in\RR^k_+$. We want to re-emphasize that different than the model of the representative non-cooperative agent, the cooperative agents are aware of their effect on the CPR dynamics and mathematically this is modeled by inputting $\mu_t^{\balpha^{\rmc}}$ in the dynamics of the CPRs instead of $\mu_t^{{\rmc}}$. 

\begin{remark}
\label{rem:common_K_MP}
We emphasize that even if the representative non-cooperative and cooperative agents observe different CPR dynamics, at the equilibrium we will have $\bmu^{\hat{\balpha}^{\rmc}} = \hat \bmu^{\rmc}$ and since because there is no randomness in the CPR dynamics and $K_0^{\rmnc} = K_0^{\rmc} = k_0$, at the equilibrium $\bK^{\rmnc}$ is equivalent to $\bK^{\rmc}$. This means at the equilibrium, there is only one CPR dynamics that are commonly observed by the non-cooperative and cooperative agents, and it can be denoted with $\bK$.    
\end{remark}

We can then introduce the extended states $\tilde{X}^{\balpha^{\rmnc}}_t = (X^{\balpha^{\rmnc}}_t, K^{\rmnc}_t)$ and $\tilde{X}^{\balpha^{\rmc}}_t = (X^{\balpha^{\rmc}}_t, K^{\rmc}_t)$,
and following a discussion similar to the one introduced in Section~\ref{subsec:mi_CPR}, this problem can be written as a mixed population mean field model as in Section~\ref{sec:mixed-population-mf}. Then, the same theoretical results hold for this setting.

\section{Motivating application: Fisheries model}
\label{sec:fishing-model}

Our motivating application is one of the most well known examples of CPRs which is modeling fishers and their interaction with the fish population. Our model is inspired by an MFG model introduced by~\cite{McPike_mfg_fish} and we include the altruistic tendencies of individual fishers or the presence of cooperative fishers in the models. The fish stock at time $t\in[0,T]$ is modeled as
the CPR and for the sake of simplicity in notation, we assume the dimension of the CPR is equal to $1$; in other words, we focus on one fish species. The models can be extended straightforwardly to include different species of fish that have for instance different reproduction rates.

\subsection{Mean field game and control fisheries model}
\label{subsec:fisher_mfg_mfc}

For the sake of completeness, we start by introducing the mathematical models of the fishing application in mean field game (MFG) and mean field control (MFC) setups.
Since MFG and MFC are special cases of the mixed mean field models, theoretical results for MFG and MFC can be deduced from the results of the mixed mean field models so we omit them.

\subsubsection{Mean field game. } \label{subsubsec:fisher_mfg}

We first state the model of the representative fisher in the MFG setup. The representative agent controls her fishing effort, $\bX^{\balpha}=(X^{\balpha}_t)_{t\in[0, T]}$, by choosing her change in effort level, $\balpha= (\alpha_t)_{t\in[0,T]}$:
\begin{equation}
\label{eq:fish_dyn_MFG}
    dX_t^{\balpha} = \alpha_t dt + \sigma dW_t,\quad X_0 \sim \mu_0,
\end{equation}
where $\sigma>0$ is a constant volatility.
The CPR dynamics for the fish stock are given as follows:
\begin{equation}
\begin{aligned}
\label{eq:fish_ODE_MFG}
    \dfrac{dK_t}{dt} = b_0(K_t) - b_1(K_t,\overline{X}_t),
\end{aligned}
\end{equation}
where $b_0(k)=  r k (1- \sfrac{k}{\cK})$ is the reproduction function, $r>0$ is the reproduction rate and $\cK$ is the carrying capacity of fish stock. The function $b_1(k, \overline{x})=q k \overline{x}$ is the fishing function where $q>0$ is a constant that represents the size of fisher population. Function $b_1$ implies that the amount of fish that is caught by fishers is a function of the fish stock level $K_t$ and the aggregate effort the fisher population puts in catching the fish, $q\overline{x}$. In this model, the representative agent take the aggregate effort level in the population exogenously while finding her best response. Intuitively, this means that she does not see her own individual effect on the CPR dynamics. 

The objective of the representative agent in MFG is to minimize the following cost functional by choosing their effort level $\balpha= (\alpha_t)_{t\in[0,T]}$ given the average fishing effort in the population, $\overline{\bX}$:
\begin{equation}
\label{eq:fish_cost_MFG}
    J(\balpha;\overline{\bX}):=\EE\int_0^T \big(f_1(X_t^{\balpha}, \alpha_t) -f_2(X_t^{\balpha},K_t,\overline{X}_t)\Big) dt.
\end{equation}
Here, $f_1(x, a) = c_1 x + \sfrac{c_2}{2}\ x^2 + \sfrac{c_3}{2}\ a^2$; the first two terms represent the cost of putting effort $x$ into fishing and the third term is the cost of changing the effort level, where $c_1, c_2, c_3>0$ are constant coefficients. The function $f_2(x, k, \overline{x})= P(k, \overline{x})k x $ is the revenue earned from fishing where $P(k, \overline{x})$ is the price function for one unit of fish and $kx$ represents the total amount of fish caught by the representative fisher. The price function $P$ is designed as a decreasing function of aggregate fishing at time $t$ to capture market dynamics. For simplicity, we will assume that it is an inverse supply function in the form of $P\big(k, \overline{x}) = p_0 -p_1 (qk \overline{x})$ where $p_0, p_1>0$ are constant coefficients.

Then, MFG Nash equilibrium is the couple $(\hat{\balpha}, \hat{\overline{\bX}})$ such that
\begin{itemize}
    \item[i.] $\hat{\balpha}$ is the minimizer of the cost~\eqref{eq:fish_cost_MFG} given $\hat{\overline{\bX}}$,
    \item[ii.] $\hat{\overline{X}}_t=\EE[X^{\hat{\balpha}}_t]$, for all $t\in[0,T]$.
\end{itemize}
Notice that the aggregate does not change at all when the individual agent changes her control. This is due to the fact that the deviation is unilateral and the agent is infinitesimal.

\subsubsection{Mean field control. }\label{subsubsec:fisher_mfc}

Next, we state the model of the representative fisher in the MFC setup. Unlike in MFG, the representative agent in an MFC model observes their effect on the mean field, i.e., the aggregate fishing effort: when the representative agent deviates from their control, the aggregate effort also changes.

Similar to MFG, the representative agent controls their fishing effort, $\bX^{\balpha}=(X^{\balpha}t){t\in[0, T]}$, by selecting their change in effort level, $\balpha= (\alpha_t)_{t\in[0,T]}$: \begin{equation} \label{eq:fish_dyn_MFC} dX_t^{\balpha} = \alpha_t dt + \sigma dW_t,\quad X_0 \sim \mu_0. \end{equation}

The CPR dynamics for the fish stock are given by: \begin{equation} \begin{aligned} \label{eq:fish_ODE_MFC} \dfrac{dK_t}{dt} &= b_0(K_t) - b_1(K_t,\overline{X}^{\balpha}_t), \end{aligned} \end{equation} where $\overline{X}^{\balpha}_t=\EE[X^{\balpha}_t]$. The functions $b_0$ and $b_1$ are defined similarly to Section~\ref{subsubsec:fisher_mfg}.

The objective of the representative fisher in MFC is to minimize the following cost functional by selecting their effort level $\balpha= (\alpha_t)_{t\in[0,T]}$: 
\begin{equation} \label{eq:fish_cost_MFC}  
    J(\balpha):=\EE\int_0^T \big(f_1(X_t^{\balpha}, \alpha_t) - f_2(X_t^{\balpha},K_t,\overline{X}^{\balpha}_t)\big) dt, 
\end{equation} 
where $\overline{X}^{\balpha}_t=\EE[X^{\balpha}_t]$. The functions $f_1$ and $f_2$ are defined similarly to Section~\ref{subsubsec:fisher_mfg}. The key distinction between the MFC and MFG models is that in MFC, the representative fisher determines their optimal control while accounting for their impact on the mean field $\overline{\bX}^{\balpha}=\mathbb{E}[\bX^{\balpha}]$. The MFC-optimal control is the policy $\hat\balpha$ that minimizes the cost $J(\cdot)$ given in~\eqref{eq:fish_cost_MFC}.

\subsection{Mixed individual fisheries model}
\label{subsubsec:MI_fish}
For the mixed individual (MI) fisheries model, we focus on a representative agent with an altruism level of $\lambda\in[0,1]$. As similar to the models in Section~\ref{subsec:fisher_mfg_mfc}, the representative fisher will choose her change in fishing effort, $\balpha = (\alpha_t)_{t \in [0,T]}$, to control her individual state, namely, fishing effort, $\bX = (X_t)_{t\in[0,T]}$. The individual state dynamics will be the same as in equation~\eqref{eq:fish_dyn_MFG} or equivalently in equation~\eqref{eq:fish_dyn_MFC}. The perceived CPR dynamics for the fish stock are given as follows:
\begin{equation}
\begin{aligned}
\label{eq:fish_ODE_MI}
    \dfrac{dK_t}{dt} 
    &= b_0(K_t) - b_1(K_t, \lambda \overline{X}_t + (1-\lambda)\overline{X}_t^{\balpha}),
\end{aligned}
\end{equation}
where $\overline{X}^{\balpha}_t=\EE[X^{\balpha}_t]$. Functions $b_0$ and $b_1$ are defined similarly to Section~\ref{subsec:fisher_mfg_mfc}.

The objective of the representative fisher in MI MFG is to minimize the following cost functional by choosing her effort level $\balpha= (\alpha_t)_{t\in[0,T]}$ given the average fishing effort in the population, $\overline{\bX}$:
\begin{equation}
\label{eq:fish_cost_MI}
\begin{aligned}
&J^\mixind(\balpha; \overline{\bX}):=\\
    &\EE\int_0^T \big(f_1(X_t^{\balpha}, \alpha_t)-f_2( X_t^{\balpha},K_t,\lambda\overline{X}_t + (1-\lambda)\overline{X}_t^{\balpha})\Big) dt
\end{aligned}
\end{equation}
where $\overline{X}^{\balpha}_t=\EE[X^{\balpha}_t]$. Functions $f_1$ and $f_2$ are defined similarly to Section~\ref{subsec:fisher_mfg_mfc}. In this model, fishers perceive their effect on the fish stock level depending on their altruism level $\lambda\in[0,1]$. We stress that if $\lambda=1$, the MI MFG model corresponds to the MFG model introduced in Section~\ref{subsubsec:fisher_mfg}, and if $\lambda=0$, the MI MFG model corresponds to the MFC model introduced in Section~\ref{subsubsec:fisher_mfc}.

Modifying Definition~\ref{def:mixed_individual_mfg_nash}, we will call a tuple of control, mean field of individual state and CPR state $(\hat\balpha, \overline \bX, \bK)$ a mixed individual mean field Nash equilibrium (MI-NFNE) if $\hat\balpha$ is the minimizer of the cost functional $J^\mixind(\cdot;\overline{\bX})$ given in~\eqref{eq:fish_cost_MI} and if $\overline{X}_t = \overline{X}_t^{\hat\balpha}$ for all $t\in[0,T]$.

\begin{corollary}
\label{cor:fisher_fbsde_mi}
Let $(\hat\balpha, \overline{\bX}, \bK)$ be a mixed individual mean field Nash equilibrium (MI-MFNE). Then, the MI-MFNE control is
given by $\hat\alpha_t = - \sfrac{Y_t^2}{c_3}$ where the tuple of processes $(\bK, \bX, \bY^1, \bY^2, \bZ^1, \bZ^2)$ solves the following MKV FBSDE:
\begin{equation}
\label{eq:fisher_mi_fbsde}
    \begin{aligned}
        dK_t &= \big[rK_t \left(1-K_t/K\right) - q K_t \overline{X}_t\big] dt,
        \\
        dX_t 
        & = -\frac{Y_t^2}{c_3}dt + \sigma dW_t,
        \\
        dY_t^1 
        &= -\Big[\big(r-\frac{2r}{K}K_t - q \overline{X}_t\big) Y_t^1 + 2p_1 q K_t \overline{X}_t X_t
%        \\
%        &\qquad\qquad 
        {-p_0 X_t}\Big] dt + Z_t^1 dW_t,
        \\
        dY_t^2 
        &= -\Big[p_1 q (2-\lambda)K_t^2 \overline{X}_t +c_1+c_2 X_t 
%        \\
%        &\qquad\qquad 
        - q(1-\lambda)K_t\overline{Y}^1_t { -p_0K_t}\Big]dt + Z_t^2 dW_t,
        \\
        K_0 
        &= \rho \cK, \quad X_0 \sim \mu_0, \quad Y_T^1=Y_T^2=0. 
    \end{aligned}
\end{equation}
\end{corollary}

\begin{proof}{Proof of Corollary~\ref{cor:fisher_fbsde_mi}}
    We utilize Theorem~\ref{thm:MI-Pontryagin}, where the Hamiltonian is written as:
    \begin{equation*}
\begin{aligned}
    H^\mixind(t, x, \alpha, \overline{x}, \overline{x}^{\prime}, k, y^1, y^2)
%    \\
%    &
    =& \big(rk(1-k/K)-qk(\lambda\overline{x} + (1-\lambda)\overline{x}^{\prime})\big) y^1 +\alpha y^2 
    \\
    &%\quad 
    +c_1x +(c_2/2) x^2 + (c_3/2) \alpha^2
%    \\
%    &\quad 
    -\big(p_0 -p_1(qk(\lambda\overline{x}+(1-\lambda)\overline{x}^{\prime})\big)kx. 
\end{aligned}
\end{equation*}
Then, we have $\partial_\alpha H = y^2 +c_3 \alpha$, hence $\hat\alpha = -y^2/c_3$.
We conclude by plugging the related information in the FBSDE system given in Theorem~\ref{thm:MI-Pontryagin}.\hfill%

\end{proof}

\begin{theorem}
\label{the:existence_uniqueness_fisher_mi}
Under small time condition, there exists a unique mixed individual mean field Nash equilibrium.

\end{theorem}

\begin{proof}{Proof Sketch of Theorem~\ref{the:existence_uniqueness_fisher_mi}}
    For the sake of brevity, we give only the sketch of the proof; the details can be found in Appendix~\ref{app:fisher_mi_proofs}. The proof consists of two steps:

    \noindent\textbf{Step 1:} We show the existence and uniqueness of the mean processes $(\overline{\hat{\balpha}}, \boldsymbol{K}, \overline{\bX}, \overline{\bY}^1, \overline{\bY}^2)$ by using Banach fixed point theorem. For this reason, we take the expectation of the FBSDE system in~\eqref{eq:fisher_mi_fbsde} and find a forward backward ordinary differential equation (FBODE) system for the tuple $(\overline{\hat{\balpha}}, \boldsymbol{K}, \overline{\bX}, \overline{\bY}^1, \overline{\bY}^2)$. We insist that the expected CPR dynamics in the FBODE system will be the same with the CPR dynamics in the FBSDE system since they do not have randomness. We treat the FBODE as a mapping that takes $\overline{\bX}$ as input and returns $\check{\overline{\bX}}$ as output, and show that it is a contraction mapping under the small time condition to apply the Banach fixed point theorem. Global-in-time solutions can be obtained by patching together solutions over a small time interval, which can be done under suitable conditions; see e.g.~\cite[Section 4.1.2]{CarmonaDelarue_book_I}.
    
    \noindent\textbf{Step 2:} Given unique mean processes $(\overline{\hat{\balpha}}, \boldsymbol{K}, \overline{\bX}, \overline{\bY}^1, \overline{\bY}^2)$, the FBSDE in~\eqref{eq:fisher_mi_fbsde} becomes linear. Therefore, we propose linear ansatz for the adjoint processes $\bY^1$ and $\bY^2$. Then, we conclude by using the existence and uniqueness results for the linear ODEs and Riccati differential equations that results from the proposition of linear ansatz for the adjoint processes. 
    \hfill%
\end{proof}

\subsection{Mixed population fisheries model}
For the mixed population (MP) fisheries model, we focus on a population of fishers in which $p$ proportion of the population is non-cooperative and $(1-p)$ proportion of the population is cooperative. Since cooperative and non-cooperative agents have different models, we need to introduce the model for both the representative non-cooperative and the representative cooperative agent.

Similar to the previous sections, the representative \textbf{non-cooperative} agent will choose her change in fishing effort, $\balpha^{\rmnc}= (\alpha^{\rmnc}_t)_{t\in[0,T]}$, to control her individual state – fishing effort, $\bX^{\balpha^\rmnc}=(X_t^{\balpha^\rmnc})_{t\in[0,T]}$. The individual state dynamics will given as follows:
\begin{equation*}
    dX_t^{\balpha^\rmnc} = \alpha^{\rmnc}_t dt + \sigma^{\rmnc} dW_t^{\rmnc},
\end{equation*}
where $\sigma^{\rmnc}>0$ is the constant volatility term. The CPR dynamics for the fish stock that are perceived by the non-cooperative representative agent are given as follows:
\begin{equation}
    \label{eq:fisher_CPR_mp_nc}
    \dfrac{dK^{\rmnc}_t}{dt} = b_0(K^{\rmnc}_t) - b_1(K^{\rmnc}_t,  p\overline{X}^{\rmnc}_t +(1-p)\overline{X}^{\rmc}_t).
\end{equation}
Functions $b_0$ and $b_1$ are defined similarly to Section~\ref{subsec:fisher_mfg_mfc}. This CPR dynamics emphasize that the non-cooperative representative agent takes the population distributions for the non-cooperative and cooperative agents exogenously, and she does not see her own direct effect on the CPR dynamics.

The objective of the representative non-cooperative fisher in MP MFG is to minimize the following cost functional by choosing her effort level $\balpha^{\rmnc}= (\alpha^{\rmnc}_t)_{t\in[0,T]}$ given the average fishing effort of the non-cooperative and cooperative fishers, $\overline{\bX}^{\rmnc}$ and $\overline{\bX}^{\rmc}$, respectively:
\begin{equation}
\label{eq:fish_cost_MP_NC}
\begin{aligned}
    &J^{\mixpop, \rmnc}(\balpha^{\rmnc}; \overline{\bX}^{\rmnc},\overline{\bX}^{\rmc}):=\EE\int_0^T \big(f_1(X_t^{\balpha^\rmnc}, \alpha^{\rmnc}_t)
%    \\
%    &\hskip1cm
    -f_2( X_t^{\balpha^\rmnc},K^{\rmnc}_t,p\overline{X}^{\rmnc}_t + (1-p)\overline{X}_t^{\rmc})\Big) dt.
\end{aligned}
\end{equation}
 Functions $f_1$ and $f_2$ are defined similarly to Section~\ref{subsec:fisher_mfg_mfc}. Non-cooperative agents take the population distribution of states -- in this model the average fishing effort for both non-cooperative and cooperative agents exogenously.

Similar to the model of the representative non-cooperative agent (and the previous sections), the representative \textbf{cooperative} agent will choose her change in fishing effort, $\balpha^{\rmc}= (\alpha^{\rmc}_t)_{t\in[0,T]}$, to control her individual state – fishing effort, $\bX^{\balpha^\rmc}=(X_t^{\balpha^\rmc})_{t\in[0,T]}$. The individual state dynamics will given as follows:
\begin{equation*}
    dX_t^{\balpha^\rmc} = \alpha^{\rmc}_t dt + \sigma^{\rmc} dW_t^{\rmc},
\end{equation*}
where $\sigma^{\rmc}>0$ is the constant volatility term and $\bW^{\rmc}$ is a Wiener process that is independent of $\bW^{\rmnc}$. The CPR dynamics for the fish stock that are perceived by the cooperative representative agent are given as follows: 
\begin{equation}
    \label{eq:fisher_CPR_mp_c}
    \dfrac{dK^{\rmc}_t}{dt} = b_0(K^{\rmc}_t) - b_1(K^{\rmc}_t,  p\overline{X}^{\rmnc}_t +(1-p)\overline{X}^{\balpha^\rmc}_t),
\end{equation}
where $\overline{X}_t^{\balpha^\rmc} = \EE[X^{\balpha^\rmc}_t]$. Functions $b_0$ and $b_1$ are defined similarly to Section~\ref{subsec:fisher_mfg_mfc}. This CPR dynamics emphasize that the non-cooperative representative agent takes the population distribution for the non-cooperative agents exogenously; however, she sees her own direct effect on the dynamics through the process $\overline{\bX}^{\balpha^\rmc}$.

The objective of the representative cooperative fisher in MP MFG is to minimize the following cost functional by choosing her effort level $\balpha^{\rmc}= (\alpha^{\rmc}_t)_{t\in[0,T]}$ given the average fishing effort of the non-cooperative fishers, $\overline{\bX}^{\rmnc}$:
\begin{equation}
\label{eq:fish_cost_MP_C}
\begin{aligned}
    &J^{\mixpop, \rmc}(\balpha^{\rmc}; \overline{\bX}^{\rmnc}):=\EE\int_0^T \big(f_1(X_t^{\balpha^\rmc}, \alpha^{\rmc}_t)
%    \\&\hskip1cm
    -f_2( X_t^{\balpha^\rmc},K^{\rmc}_t,p\overline{X}^{\rmnc}_t + (1-p)\overline{X}_t^{\balpha^\rmc})\Big) dt,
\end{aligned}
\end{equation}
where $\overline{X}_t^{\balpha^\rmc} = \EE[X^{\balpha^\rmc}_t]$. Functions $f_1$ and $f_2$ are defined similarly to Section~\ref{subsec:fisher_mfg_mfc}. Cooperative agents take the population distribution of states -- in this model the average fishing effort for non-cooperative agents exogenously; however, they see their own effect on the system through $\overline{\bX}_t^{\balpha^\rmc}$.

Following Remark~\ref{rem:common_K_MP}, at the equilibrium we will have $\overline{\bX}^{\rmc} = \overline{\bX}^{\balpha^{\rmc}}$ and; furthermore, since $K_0^{\rmnc}=K_0^{\rmc}=k_0$ and there is no randomness in the CPR dynamics, at the equilibrium there is one common CPR process $\bK^{\rmnc}=\bK^{\rmc}$. To simplify the notation, we will drop the superscript and use notation $\bK$ for the CPR dynamics at the equilibrium of mixed population mean field model.

Modifying Definition~\ref{def:mixedpop_mfg_nash}, we will call a tuple of controls, mean field of individual states, and CPR state $(\hat\balpha^{\rmnc}, \hat\balpha^{\rmc}, \overline{\bX}^{\rmnc}, \bK)$ a mixed population mean field equilibrium (MP-NFE) if $\hat\balpha^{\rmnc}$ is the minimizer of the cost functional $J^{\mixind, \rmnc}(\cdot;\overline{\bX}^{\rmnc}, \overline{\bX}^{\balpha^{\rmc}})$ given in~\eqref{eq:fish_cost_MP_NC}, if $\hat\balpha^{\rmc}$ is the minimizer of the cost functional $J^{\mixind, \rmc}(\cdot;\overline{\bX}^{\rmnc})$ given in~\eqref{eq:fish_cost_MP_C}, and if $\overline{X}^{\rmnc}_t = \overline{X}_t^{\hat\balpha^{\rmnc}}$ for all $t\in[0,T]$.

\begin{corollary}
\label{cor:fisher_fbsde_mp}
Let $(\hat\balpha^{\rmnc}, \hat\balpha^{\rmc}, \overline{\bX}^{\rmnc}, \bK)$ be a mixed population mean field equilibrium (MP-MFE). Then, MP-MFE controls are
given by $\hat\alpha^{\rmnc}_t = - Y_t^{2,\rmnc}/c^{\rmnc}_3$ and $\hat\alpha^{\rmc}_t = - Y_t^{2,\rmc}/c^{\rmc}_3$ where the tuple of processes $(\bK, \bX^{\rmnc}, \bX^{\rmc}, \bY^{1, \rmnc}, \bY^{1, \rmc}, \bY^{2,\rmnc}, \bY^{2,\rmc}, \bZ^{1, \rmnc}, \bZ^{1, \rmc},$ $ \bZ^{2, \rmnc}, \bZ^{2, \rmc})$ solves the following MKV FBSDE:
\begin{equation}
\label{eq:fisher_mp_fbsde}
    \begin{aligned}
        dK_t =&rK_t \big(1-\dfrac{K_t}{K}\big) - q K_t (p\overline{X}^{\rmnc}_t + (1-p)\overline{X}^{\rmc}_t) dt,\\
        dX^{\rmnc}_t  =& -\frac{Y_t^{2,\rmnc}}{c^{\rmnc}_3}dt + \sigma^{\rmnc} dW^{\rmnc}_t,\\
        dX^{\rmc}_t  =& -\frac{Y_t^{2,\rmc}}{c^{\rmc}_3}dt + \sigma^{\rmc} dW^{\rmc}_t,\\
        dY_t^{1,\rmnc} =& -\hskip-0.5mm\big[\hskip-0.5mm\big(r-\tfrac{2r}{K}K_t - q (p\overline{X}^{\rmnc}_t + (1-p)\overline{X}^{\rmc}_t)\big) Y_t^{1,\rmnc} \\
        &- 2p_1 q K_t (p\overline{X}^{\rmnc}_t + (1-p)\overline{X}^{\rmc}_t) X^{\rmnc}_t
%        \\
%        &
        +p_0X_t^{\rmnc}\big] dt + Z_t^{1,\rmnc} dW^{\rmnc}_t,\\
        dY_t^{2,\rmnc} =& -\big[{ -p_0 K_t} + p_1 q K_t^2 (p\overline{X}^{\rmnc}_t + (1-p)\overline{X}^{\rmc}_t)
%        \\ &
        +c^{\rmnc}_1+c^{\rmnc}_2 X^{\rmnc}_t \big]dt+Z_t^{2,\rmnc} dW^{\rmnc}_t,\\
        dY_t^{1,\rmc} =& -\hskip-0.5mm\big[\big(r-\frac{2r}{K}K_t - q (p\overline{X}^{\rmnc}_t + (1-p)\overline{X}^{\rmc}_t)\big) Y_t^{1,\rmc} \\
        &+ 2p_1 q K_t (p\overline{X}^{\rmnc}_t + (1-p)\overline{X}^{\rmc}_t) X^{\rmc}_t
%        \\&
        {-p_0 X_t^{\rmc}}\big] dt + Z_t^{1,\rmc} dW^{\rmc}_t,\\
        dY_t^{2,\rmc} =& -\big[{ -p_0 K_t}+p_1 q K_t^2 (p\overline{X}^{\rmnc}_t + (1-p)\overline{X}^{\rmc}_t) 
        \\&
        +c^{\rmc}_1+c^{\rmc}_2 X^{\rmc}_t -qK_t(1-p)\overline{Y}_t^{1,\rmc} 
%        \\&
        + p_1qK_t^2(1-p)\overline{X}^{\rmc}_t\big]dt+Z_t^{2,\rmc} dW^{\rmc}_t,\\
        K_0 =& \rho \cK, \quad X^{\rmnc}_0 \sim \mu^{\rmnc}_0, X^{\rmc}_0 \sim \mu^{\rmc}_0, \quad
%        \\&
        Y_T^{1,\rmnc}=Y_T^{2,\rmnc}=Y_T^{1,\rmc}=Y_T^{2,\rmc}=0. 
    \end{aligned}
\end{equation}
\end{corollary}

\begin{proof}{Proof of Corollary~\ref{cor:fisher_fbsde_mp}}
    We utilize Theorem~\ref{thm:MP-Pontryagin} where the Hamiltonians for non-cooperative and cooperative representative fishers are written as:
    \begin{equation*}
    \begin{aligned}
        H^{\mixpop,\rmnc}(t, x, \alpha, \overline{x}, \overline{x}^{\prime}, k, y^{1, \rmnc}, y^{2, \rmnc})
%        \\
        &= \big(rk(1-k/K)-qk(p\overline{x} + (1-p)\overline{x}^{\prime})\big) y^{1,\rmnc} +\alpha y^{2,\rmnc} 
        \\
        &\quad 
        +c_1x +(c_2/2) x^2 + (c_3/2) \alpha^2
%        \\
%        &\quad 
        -\big(p_0 -p_1(qk(p\overline{x}+(1-p)\overline{x}^{\prime})\big)kx, 
        \end{aligned}
        \end{equation*}
    \begin{equation*}
    \begin{aligned}
        =H^{\mixpop,\rmc}(t, x, \alpha, \overline{x}, \overline{x}^{\prime}, k, y^{1,\rmc}, y^{2,\rmc})
         &= \big(rk(1-k/K)-qk(p\overline{x} + (1-p)\overline{x}^{\prime})\big) y^{1,\rmc} +\alpha y^{2,\rmc} 
        \\
        &\quad +c_1x +(c_2/2) x^2 + (c_3/2) \alpha^2
%        \\
%        &\quad 
        -\big(p_0 -p_1(qk(p\overline{x}+(1-p)\overline{x}^{\prime})\big)kx.
    \end{aligned}
    \end{equation*}
Then, we have $\hat \alpha^{\rmnc}_t = \argmin_{\alpha} H^{\mixpop, \rmnc}=- y^{2,\rmnc}/c^{\rmnc}_3$ and $\hat \alpha^{\rmc}_t = \argmin_{\alpha} H^{\mixpop, \rmc}=- y^{2,\rmc}/c^{\rmc}_3$.
We conclude by plugging the related information in the FBSDE system given in Theorem~\ref{thm:MP-Pontryagin}. We emphasize that since at the MP-MFE $\overline{X}^{\rmc} = \overline{\bX}^{\balpha^{\rmc}}$, perceived CPR dynamics for non-cooperative and cooperative agents, $\bK^{\rmnc}$ and $\bK^{\rmc}$, will be the same. Therefore, we will have only one CPR dynamics where we denote with $\bK$.\hfill%

\end{proof}

\begin{theorem}
\label{the:existence_uniqueness_fisher_mp}
   Under small time condition, there exists a unique mixed population mean field equilibrium.
\end{theorem}

\begin{proof}{Proof Sketch of Theorem~\ref{the:existence_uniqueness_fisher_mp}}
    For the sake of brevity, we give the sketch of the proof, the details can be found in Appendix~\ref{app:fisher_mp_proofs}. Following Corollary~\ref{cor:fisher_fbsde_mp}, the existence uniqueness analysis of MP-MFE corresponds to the existence-uniqueness analysis of the FBSDE in~\eqref{eq:fisher_mp_fbsde}. The proof consists of two steps:

    \noindent\textbf{Step 1:} We show the existence and uniqueness of the mean processes $(\overline{\hat{\balpha}}^{\rmnc}, \overline{\hat{\balpha}}^{\rmc}, \boldsymbol{K},$ $ \overline{\bX}^{\rmnc}, \overline{\bX}^{\rmc}, \overline{\bY}^{1, \rmnc}, \overline{\bY}^{1, \rmc}, \overline{\bY}^{2,\rmnc},$ $\overline{\bY}^{2,\rmc})$ by using Banach fixed point theorem. For this reason, we take the expectation of the FBSDE system in~\eqref{eq:fisher_mp_fbsde} and find a forward backward ordinary differential equation (FBODE) system for the tuple. We insist that the expected CPR dynamics in the FBODE system will be the same with the CPR dynamics in the FBSDE system since they do not have randomness. We treat the FBODE as a mapping from $(\overline{\bX}^{\rmnc}, \overline{\bX}^{\rmc})$ to $(\check{\overline{\bX}}^{\rmnc}, \check{\overline{\bX}}^{\rmc})$, and show that it is a contraction mapping under the small time condition to apply the Banach fixed point theorem. As mentioned before, global-in-time solutions can be obtained by patching together solutions over a small time interval, which can be done under suitable conditions.
    
    \noindent\textbf{Step 2:} Given unique mean processes $(\overline{\hat{\balpha}}^{\rmnc}, \overline{\hat{\balpha}}^{\rmc}, \boldsymbol{K}, \overline{\bX}^{\rmnc}, \overline{\bX}^{\rmc}, \overline{\bY}^{1, \rmnc}, \overline{\bY}^{1, \rmc}, \overline{\bY}^{2,\rmnc}, \overline{\bY}^{2,\rmc})$ the FBSDE in~\eqref{eq:fisher_mp_fbsde} becomes linear. Therefore, we propose linear ansatz for the adjoint vector processes $\bY^1= [\bY^{1,\rmnc}, \bY^{1,\rmc}]^{\top}$ and $\bY^2= [\bY^{2,\rmnc}, \bY^{2,\rmc}]^{\top}$. Then, we conclude by using the existence and uniqueness results for the linear ODE systems and matrix Riccati differential equations that results from the proposition of linear ansatz for the adjoint processes. 
    \hfill%
\end{proof}

\section{Numerical results}
\label{sec:numerics}

In this section, we will analyze the effect of the mixed mean field model parameters $\lambda$ and $p$ on the results. In order to accomplish, this we will numerically solve the FBODE systems that result in by taking the expectation of the FBSDEs that are given in~\eqref{eq:fisher_mi_fbsde} and~\eqref{eq:fisher_mp_fbsde}, respectively for mixed individual and mixed population mean field models. In this way, following Corollary~\ref{cor:fisher_fbsde_mi}, we can compute the MI-MFNE individual and CPR state flows $\overline{\bX}, \bK$, and average control $\overline{\hat\balpha}=(\overline{\hat\alpha}_t)_{t\in[0,T]}$ where $\overline{\hat\alpha}_t=-\overline{Y}_t^2/c_3$. Similarly, following Corollary~\ref{cor:fisher_fbsde_mp}, we can compute the MP-MFE individual (both for non-cooperative and cooperative agents) and CPR state flows, $\overline{\bX}^{\rmnc}, \overline{\bX}^{\rmc}, \bK$, and average controls $\overline{\hat\balpha}^{\rmnc}=(\overline{\hat\alpha}^{\rmnc}_t)_{t\in[0,T]}$, $\overline{\hat\balpha}^{\rmc}=(\overline{\hat\alpha}^{\rmc}_t)_{t\in[0,T]}$ where $\overline{\hat\alpha}^{\rmnc}_t=-\overline{Y}_t^{2,\rmnc}/c^{\rmnc}_3$ and $\overline{\hat\alpha}^{\rmc}_t=-\overline{Y}_t^{2,\rmc}/c^{\rmc}_3$. For solving the FBODE systems, we will use fixed point algorithm which is guaranteed to converge when $T$ is chosen to be small due to Theorems~\ref{the:existence_uniqueness_fisher_mi} and~\ref{the:existence_uniqueness_fisher_mp}. In order to extend the terminal time, we will use fictitious play approach. For the experiments, we use the following parameters $\cK=1$, $\rho=0.7$, $r=0.75$, $q=1.0$, $c_1=c_1^{\rmnc}=c_1^{\rmc}=0.5$, $c_2=c_2^{\rmnc}=c_2^{\rmc}=0.1$, $c_3=c_3^{\rmnc}=c_3^{\rmc}=1.0$, $\overline{X}_0=\overline{X}_0^{\rmnc}, \overline{X}_0^{\rmc}=1$, $p_1=1, p_0=2$, and $T=4$. To decide what defines the management concern related to possible the tragedy of the commons in the fisheries model, we use the outlines given in the report of~\cite{wwf_fish}. We set the $\cK=1$ as the \textit{unfished biomass} that is denoted by $\text{SSB}_0$ where SSB stands for \textit{Spawning Stock Biomass} of the fish species of interest in the oceans. WWF states that a healthy SSB is around $40\%$ of $\text{SSB}_0$, i.e. and below $20\%$ is called unsustainable levels or critical threshold. Then, levels around $25-35\%$ are the management concern levels. Therefore, we set $30\%$ as our \textit{management concern} level that we aim to stay above to prevent the tragedy of the commons. The results of regular MFG and MFC models can be seen in~\ref{fig:mfgVSmfc}. In this plot, we see that the CPR level goes below the management concern level in the MFG scenario; on the other hand, it stays sustainable around 0.5 for the MFC scenario. Now, we will further analyze the effect of parameters in the mixed MFG models. 

In Figure~\ref{fig:mi}, we can see the results of mixed individual mean field problems with different altruism levels, $(1-\lambda)$. This can be thought as each person is $100(1-\lambda)\%$ altruistic. We can see on the left subplot that when the individual altruism increases (i.e., $\lambda$ decreases), the CPR levels over time move from MFG setting to the MFC setting monotonously. This says that if the fishers are more altruistic, sustainability of the CPR will be more likely. We also see that when the people are individually altruistic (or in other words, if they see their individual effect on the system more clearly), they show more concern about the sustainability of the CPR and has lower fishing effort levels. We can see that the CPR levels stay above the management concern line ($K=0.3$) even when the altruism level ($1-\lambda$) is equal to $0.3$.

In Figure~\ref{fig:mp}, we can see the results of mixed population mean field problems with different levels of non-cooperative fisher proportion in the population, $p$. We can see on the left subplot that when the proportion of non-cooperative fishers in the population decreases, the CPR levels over time move from MFG setting to MFC setting monotonously. This says that when the proportion of non-cooperative fishers in the population decreases, CPR levels stay at more sustainable levels due to the higher population of fishers who are more concerned about the sustainability of CPR. We can also see respectively in the middle and the right subplots the average control (i.e., the fishing effort change) and average states (i.e., the fishing effort) for non-cooperative and cooperative fishers for a given $p$ level. We can see that at the same level of $p$ fishing effort of non-cooperative fishers are much higher then cooperative fishers. For example when 30\% of the population is non-cooperative ($p=0.3)$, we can see that non-cooperative fishers increase their fishing effort and in response to this to protect the sustainability of fish stock cooperative fishers further decrease their fishing efforts. This shows that in the mixed population mean field model, non-cooperative fishers enjoy \textit{free rider} phenomenon.

Finally, in Figure~\ref{fig:miVSmp_v2}, we compare the results of mixed individual (MI) and mixed population (MP) mean field models at the same levels of model parameters, $\lambda$ and $p$. This would be comparing the results of having each individual a specific percentage level of altruistic or having a population with the same specific proportion of fully cooperative fishers. When we analyze the left subplot (and similarly the other plots), we can see that at $\lambda=p=0$ and $\lambda=p=1$, the CPR dynamics are corresponding to each other for MI and MP model, this is because the former one corresponds to the MFC setting in both cases while the latter one corresponds to the MFG setting. However, when a $\lambda$ and $p$ between 0 and 1 are chosen, we can see that the CPR levels are lower for the MP model at the same $\lambda$ and $p$ levels. For example, for the mixed individual model, the Fish Stock (i.e., CPR level) is below the management concern level (shown with the dotted horizontal line) only for the case where $\lambda=1$, i.e., individuals are fully non-altruistic. The other cases with $\lambda=0, 0.3, 0.5, 0.7$, we can see that CPR levels stay above the management concern line. However, for the mixed population model, the CPR level is below the management concern level for the populations with $p=0.5, 0.7, 1.0$, which are the populations that consists of $50\%$, $70\%$, and $100\%$ non-cooperative fishers, respectively. This gives us a very important insight. \textit{It is better to have everyone partially altruistic instead of some proportion of the population fully altruistic (i.e., cooperative) while the others are non-altruistic, due to the free rider phenomenon that is observed in mixed populations.}

Supporting this observation, in the middle and the right subplots of Figure~\ref{fig:miVSmp_v2}, we can see respectively the average control (i.e., the fishing effort change) and average states (i.e., the fishing effort) in the populations for the specific model parameters. For the mixed population mean field model, the average state in the mixed population is calculated by $\overline{X}_t=p\overline{X}_t^{\rmnc} +(1-p)\overline{X}_t^{\rmc}$  where $\overline{\alpha}_t$ is calculated similarly. We can observe that at the same level of $\lambda$ and $p$ parameters, the average fishing effort in the mixed population mean field model is higher than the corresponding mixed individual mean field model due to the free rider behavior of the non-cooperative fishers in the mixed population model.

\begin{figure*}%
\centering
        \includegraphics[width=1\linewidth]{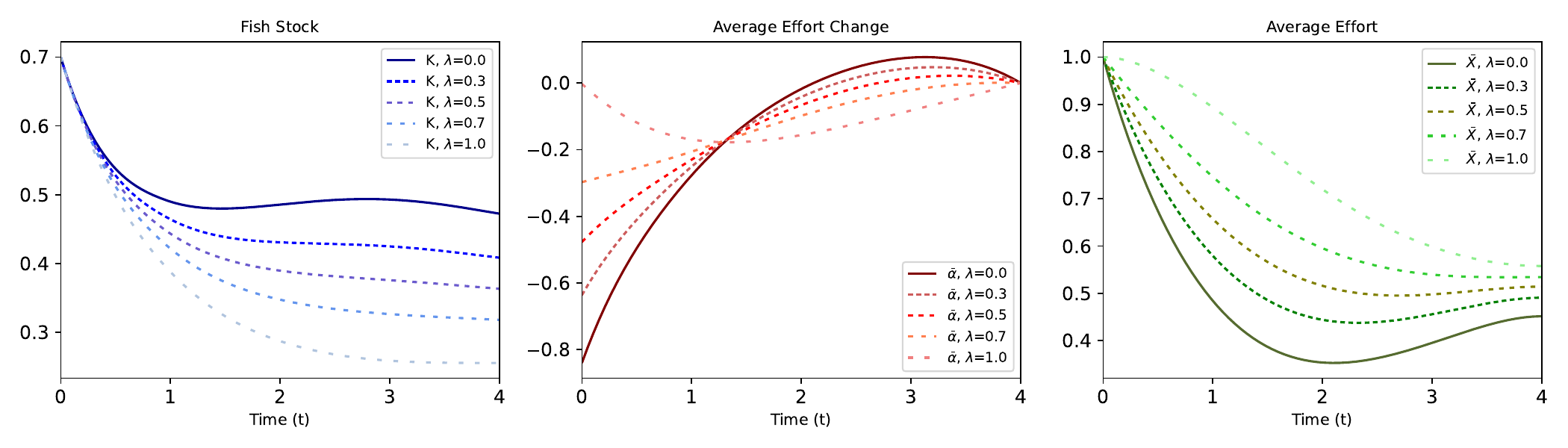}
    \caption{Mixed Individual MFG results under different model parameter, $\lambda$, choices. Left: CPR levels over time; Middle: Average control (i.e., average fishing effort change) in the population over time; Right: Average state (i.e., average fishing effort) over time.}
    \label{fig:mi}
\end{figure*}

\begin{figure*}%
\centering
        \includegraphics[width=1\linewidth]{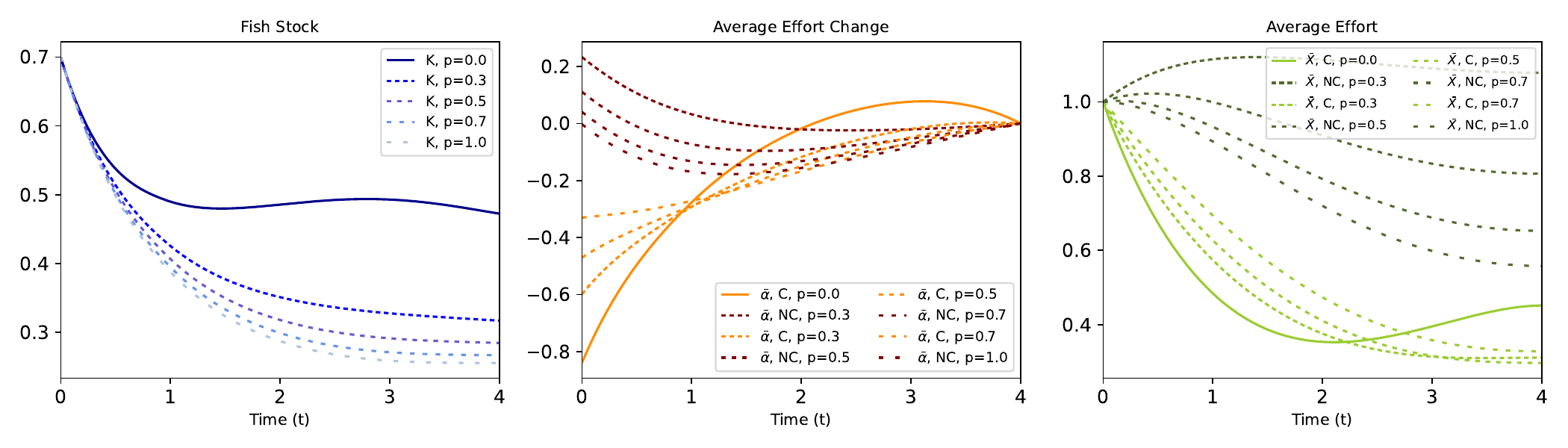}
\caption{Mixed Population MFG results under different model parameter, $p$, choices. Left: CPR levels over time; Middle: Average control (i.e., average fishing effort change) for noncooperative (NC) and cooperative (C) fishers over time; Right: Average state (i.e., average fishing effort) for noncooperative (NC) and cooperative (C) fishers over time. Since $p=0$ and $p=1$ respectively correspond to fully cooperative and fully competitive populations, we only added the related lines in the middle and right subplots.}
    \label{fig:mp}
\end{figure*}

\begin{figure*}%
\centering
        \includegraphics[width=1\linewidth]{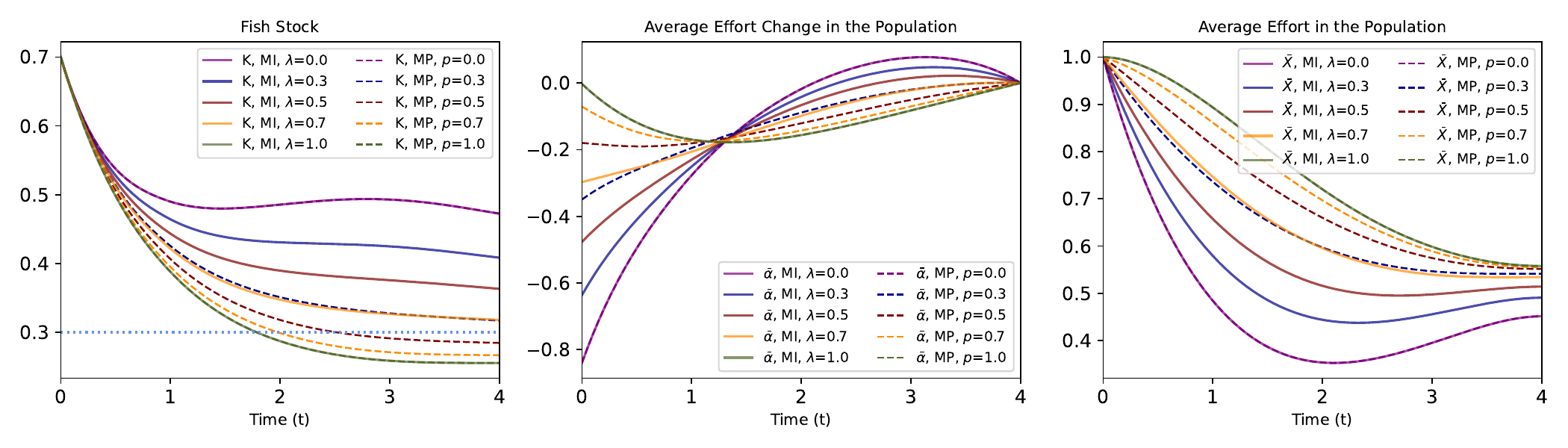}
        \caption{Mixed Individual MFG vs. Mixed Population MFG results under different (but the same level of) model parameters. Left: CPR levels over time; Middle: Average control (i.e., average fishing effort change) in the fisher population over time; Right: Average state (i.e., average fishing effort) in the fisher population over time.}
    \label{fig:miVSmp_v2}
\end{figure*}

\section{Conclusion}

In this paper, we extend the classical MFG framework to better capture the complex behaviors observed in the models where common pool resources are involved, where pure self-interest of individuals often fails to explain empirical observations. By introducing two new classes of models, namely \textit{Mixed Individual MFGs} and \textit{Mixed Population MFGs}, we incorporate non-cooperative and altruistic tendencies at the agent and population levels, respectively. These models allow us to interpolate between purely competitive and purely cooperative dynamics, providing a richer and more realistic description of agent behaviors capturing the realistic behaviors observed in the scenarios where CPRs are involved.

Our theoretical contributions include the derivation of optimality conditions via FBSDEs for both classes of models and proofs that the proposed mean field models provide approximate equilibria for corresponding finite-agent games. Furthermore, we demonstrate how these models can be extended to include common pool resource dynamics, to capture applications such as fisheries management. For our motivating example of fishery modeling, we give the FBSDE systems characterizing the results and we prove existence and uniqueness of solutions. We illustrate the effects of the levels of altruism on the CPR sustainability in our numerical simulations. 

The results show that introducing cooperative behaviors even in small proportions can significantly mitigate the tragedy of the commons. Furthermore, we conclude that the populations with every agent is partially altruistic perform better than the populations where some proportion of agents are fully altruistic and the remaining is competitive due to the \textit{free rider} phenomenon in the latter case. These findings not only align with empirical observations in real-world CPR systems but also suggest new avenues for designing decentralized mechanisms to promote sustainable behavior.

Future research directions include the extension of these frameworks to more general dynamics to include a common randomness for the common pool resources. We are also interested in analyzing mechanism design questions by adding a regulator to see the effects of regulations. Finally, we plan to model different applications with CPRs such as air quality and underground water sustainability.

\onecolumn
\bibliographystyle{informs2014} %
%\bibliography{sn-bibliography} %

\clearpage 

\onecolumn

\appendix 

\section{Proofs for MI mean field model}\label{app:MI-proofs}

\subsection{Approximate equilibrium}
\label{app:MI-eps-Nash}
We start by stating the assumptions we will use in the proof. We suppose that the following conditions hold:
\begin{itemize}
    \item The drift is of the form: $b(t,x,a,\mu,\mu') = a$
    \item The running cost is of the form: $f(t,x,a,\mu,\mu') = \frac{1}{2}|a|^2 + F(x,\mu,\mu')$
	\item The map $F$ is $L_F$-Lipschitz continuous with respect to $(x,\mu)$ in the following sense: for any $(x,\mu,\mu')$ and $(\tilde x,\tilde \mu,\mu')$
$$
	|F(x,\mu,\mu') - F(\tilde x,\tilde \mu,\mu')| \leq L_F [|x-\tilde x| + W_1(\mu,\tilde \mu)]
$$
where $W_1$ is the $1$-Wasserstein distance. Likewise, the terminal cost $g$ is Lipschitz continuous in $(x,\mu)$ with Lipschitz constant $L_g$. %
	\item The maps $F$ and $g$ are bounded: $C_{F} := \|F\|_\infty$ and $ C_g := \|g\|_\infty $ are finite.
	\item The initial distribution has a finite second moment: $\EE|X_0|^2<+\infty$.
\end{itemize}

\begin{proof}{Proof of Theorem~\ref{thm:MI-eps-Nash}}

We denote by $\hat{\bualpha} = (\hat\bualpha^j)_{j=1}^N$ the control profile where $\hat\alpha^j_t = \hat\alpha(t,X^{j,\hat\bualpha}_t)$. We want to show that: for every $i=1,\dots,N$ and every open-loop control $\balpha^i$,
\begin{equation}
    \label{eq:MI-eps-Nash-ineq-eps}
    J^{\mixind,N}(\hat{\bualpha}) \le J^{\mixind,N}(\balpha^i; \hat{\bualpha}^{-i}) + \varepsilon,
\end{equation}
for some $\varepsilon>0$ to be chosen below. In the proof, we introduce $\epsilon \in (0,1)$ and the value of $\varepsilon$ will be equal to $\sqrt\epsilon$, up to a multiplicative constant independent of $N$.

\noindent{\bf Step 1: Bound on the controls. }
Note that, by optimality of $\hat\balpha$, 
\[
    J^{\mixind}(\hat\balpha; \mu^{\hat{\balpha}}) \le J^{\mixind}(0; \mu^{\hat{\balpha}}),
\]
where $0$ denotes the control which is identically $0$ for each time $t$. The left-hand side is bounded from below as:
\[
	\EE\int_0^T \frac{1}{2}|\hat\alpha(t,X^{\hat\balpha}_t)|^2dt - (TC_{F} + C_{g}) 
	\le
	J^{\mixind}(\hat\balpha; \mu^{\hat{\balpha}})
\]
while the right-hand side is bounded from above by $(TC_{F} + C_{g})$. 
As a consequence,
\begin{equation}
\label{eq:MI-eps-Nash-Camax}
    \EE\int_0^T \frac{1}{2}|\hat\alpha(t,X^{\hat\balpha}_t)|^2dt
    \le 2 (TC_{F} + C_{g}) =: C_{a,max}. 
\end{equation}

Without loss of generality, we can assume that 
\begin{equation}
    \label{eq:MI-bound-alphai-kappa}
    \EE \int_0^T \frac{1}{2}|\alpha^i_t|^2 dt 
    \le \epsilon +  (TC_{F} + C_{g}) + C_{a,max} =: \kappa.
\end{equation}
Indeed, if the above inequality is not true, then $J^{\mixind,N}(\balpha^i; \hat{\bualpha}^{-i})$ is large enough that~\eqref{eq:MI-eps-Nash-ineq-eps} holds, which concludes the proof immediately.

\noindent{\bf Step 2: From the finite population cost to the mean field cost. }
We have:
\begin{align*}
    &J^{\mixind,N}(\balpha^i; \hat{\bualpha}^{-i})
    \\
    &= \EE \Big[\int_0^T f(t, X_t^{i,\balpha^i,\hat{\bualpha}^{-i}}, \alpha^i_t, \mu^{N,\balpha^i,\hat{\bualpha}^{-i}}_t, \mu_t^{i, \balpha^i,\hat{\bualpha}^{-i}})dt 
    \\
    &\qquad + g(X_T^{i,\balpha^i,\hat{\bualpha}^{-i}}, \mu^{N,\balpha^i,\hat{\bualpha}^{-i}}_T, \mu_T^{i, \balpha^i,\hat{\bualpha}^{-i}})\Big].
\end{align*}
Note that, since the drift is only the individual control, the process $\bX^{i,\balpha^i,\hat{\bualpha}^{-i}}$ actually does not depend on the control of other agents. So we can write $\bX^{i,\balpha^i}$. Similarly, instead of $\mu_t^{i, \balpha^i,\hat{\bualpha}^{-i}}$, we can write $\mu_t^{\balpha^i}$. 

For the final cost, we split it as:
\begin{align*}
    &\EE\left[g(X_T^{i,\balpha^i}, \mu^{N,\balpha^i,\hat{\bualpha}^{-i}}_T, \mu_T^{i, \balpha^i,\hat{\bualpha}^{-i}})\right]
    \\
    &= \EE\left[\indic_{\left\{\sup_{t \in [0,T]} |X_t^{i,\balpha^i}| > 1/\sqrt{\epsilon} \right\}} g(X_T^{i,\balpha^i}, \mu^{N,\balpha^i,\hat{\bualpha}^{-i}}_T, \mu_T^{\balpha^i})\right]
    \\
    &\qquad +
    \EE\left[\indic_{\left\{\sup_{t \in [0,T]} |X_t^{i,\balpha^i}| \le 1/\sqrt{\epsilon} \right\}} g(X_T^{i,\balpha^i}, \mu^{N,\balpha^i,\hat{\bualpha}^{-i}}_T, \mu_T^{\balpha^i})\right].
\end{align*}
In the right-hand side, the first term is bounded from below by:
\[
    - \|g\|_\infty \PP\left( \sup_{t \in [0,T]} |X_t^{i,\balpha^i}| > 1/\sqrt{\epsilon} \right)
    \ge - C_g K_1 \sqrt\epsilon,
\]
where the inequality holds by Lemma~\ref{lem MI bd Xi out} with $K_1$ depending only on $\EE|X_0|^2,T,\sigma$ and $\kappa$. 

The second term, thanks to Lemma~\ref{lem MI bd cost inside}, is bounded from below by:
\begin{align*}
    &\EE\left[\indic_{\left\{\sup_{t \in [0,T]} |X_t^{i,\balpha^i}| \le 1/\sqrt{\epsilon} \right\}} g(X_T^{i,\balpha^i}, \mu^{\hat{\balpha}}_T, \mu_T^{\balpha^i})\right] - \epsilon
    \\
    &= \EE\left[ g(X_T^{i,\balpha^i}, \mu^{\hat{\balpha}}_T, \mu_T^{\balpha^i})\right] 
    \\
    &\qquad - \EE\left[\indic_{\left\{\sup_{t \in [0,T]} |X_t^{i,\balpha^i}| > 1/\sqrt{\epsilon} \right\}} g(X_T^{i,\balpha^i}, \mu^{\hat{\balpha}}_T, \mu_T^{\balpha^i})\right] - \epsilon
    \\
    &\ge \EE\left[ g(X_T^{i,\balpha^i}, \mu^{\hat{\balpha}}_T, \mu_T^{\balpha^i})\right] - C_g K_1\sqrt\epsilon - \epsilon,
\end{align*}
where $\mu^{\hat{\balpha}}_t$ is the law of $X^{\hat\balpha}_t$, controlled by $t \mapsto \hat\alpha(t,X^{\hat\balpha}_t)$, and we used again Lemma~\ref{lem MI bd Xi out} for the inequality.

We proceed similarly for the running cost $F$ and we obtain:
\begin{align*}
    &\EE \left[F(X_t^{i,\balpha^i},  \mu^{N,\balpha^i,\hat{\bualpha}^{-i}}_t, \mu_t^{\balpha^i})  \right]\\
    &\ge 
    \EE \left[\indic_{\left\{\sup_{t \in [0,T]} |X_t^{i,\balpha^i}| > 1/\sqrt{\epsilon} \right\}} F(X_t^{i,\balpha^i}, \mu^{N,\balpha^i,\hat{\bualpha}^{-i}}_t, \mu_t^{\balpha^i}) \right]
    \\
    & + 
    \EE \left[\indic_{\left\{\sup_{t \in [0,T]} |X_t^{i,\balpha^i}| \le 1/\sqrt{\epsilon} \right\}} F(X_t^{i,\balpha^i}, \mu^{N,\balpha^i,\hat{\bualpha}^{-i}}_t, \mu_t^{\balpha^i}) \right]
    \\
    &\ge 
    -C_F K_1 \sqrt\epsilon 
    \\
    &\qquad + 
    \EE \left[ F(X_t^{i,\balpha^i}, \mu^{\hat\balpha}_t, \mu_t^{\balpha^i}) \right] - C_F K_1 \sqrt\epsilon - \epsilon.
\end{align*}
 
Collecting the terms, we obtain:
\begin{align}
    &J^{\mixind,N}(\balpha^i; \hat{\bualpha}^{-i})
    \notag
    \\
    &= \EE \Big[\int_0^T f(t, X_t^{i,\balpha^i}, \alpha^i_t, \mu^{N,\balpha^i,\hat{\bualpha}^{-i}}_t, \mu_t^{\balpha^i})dt 
    \notag
    \\
    &\qquad + g(X_T^{i,\balpha^i}, \mu^{N,\balpha^i,\hat{\bualpha}^{-i}}_T, \mu_T^{\balpha^i})\Big]
    \notag
    \\
    &= \EE \Big[\int_0^T \frac{1}{2} |\alpha^i_t|^2 dt + \int_0^T F(X_t^{i,\balpha^i}, \mu^{N,\balpha^i,\hat{\bualpha}^{-i}}_t, \mu_t^{\balpha^i})dt 
    \notag
    \\
    &\qquad + g(X_T^{i,\balpha^i}, \mu^{N,\balpha^i,\hat{\bualpha}^{-i}}_T, \mu_T^{\balpha^i})\Big]
    \notag
    \\
    & 
    \ge \EE \Big[\int_0^T \frac{1}{2} |\alpha^i_t|^2 dt + \int_0^T F(X_t^{i,\balpha^i}, \mu^{\hat\balpha}_t, \mu_t^{\balpha^i})dt 
    \notag
    \\
    &\qquad + g(X_T^{i,\balpha^i}, \mu^{\hat\balpha}_T, \mu_T^{\balpha^i})\Big]
    \notag
    \\
    &\qquad - T(2C_FK_1 + 1)\sqrt\epsilon - (2C_gK_1 + 1) \sqrt\epsilon
    \notag
    \\
    &= J^{\mixind}(\balpha^i; \mu^{\hat{\balpha}}) - [T(2C_FK_1 + 1) + 2C_gK_1 + 1] \sqrt\epsilon 
    \notag
    \\
    &\ge J^{\mixind}(\hat\balpha; \mu^{\hat{\balpha}}) - [T(2C_FK_1 + 1) + 2C_gK_1 + 1] \sqrt\epsilon  ,
    \label{eq:MI-epsilon-MF-to-FP}
\end{align}
where the last inequality is by the optimality of $\hat\balpha$  against $\bmu^{\hat{\balpha}}$ due to the mean field Nash equilibrium property.

\noindent{\bf Step 3: From the mean field cost to the the finite population cost. }
Now we want to show that  $J^{\mixind}(\hat\balpha; \mu^{\hat{\balpha}}) \ge J^{\mixind,N}(\hat\balpha; \hat{\balpha}^{-i}) - \dots$
We will prove an upper bound on $J^{\mixind,N}(\hat\balpha; \hat{\balpha}^{-i})$. 

For the final cost, we split it as:
\begin{align*}
    &\EE\left[g(X_T^{i,\hat\balpha}, \mu^{N,\hat\balpha,\hat{\bualpha}^{-i}}_T, \mu_T^{i, \hat\balpha,\hat{\bualpha}^{-i}})\right]
    \\
    &= \EE\left[\indic_{\left\{\sup_{t \in [0,T]} |X_t^{i,\hat\balpha}| > 1/\sqrt{\epsilon} \right\}} g(X_T^{i,\hat\balpha}, \mu^{N,\hat\balpha,\hat{\bualpha}^{-i}}_T, \mu_T^{\hat\balpha})\right]
    \\
    &\qquad +
    \EE\left[\indic_{\left\{\sup_{t \in [0,T]} |X_t^{i,\hat\balpha}| \le 1/\sqrt{\epsilon} \right\}} g(X_T^{i,\hat\balpha}, \mu^{N,\hat\balpha,\hat{\bualpha}^{-i}}_T, \mu_T^{\hat\balpha})\right].
\end{align*}
In the right-hand side, the first term is bounded from above by:
\[
    \|g\|_\infty \PP\left( \sup_{t \in [0,T]} |X_t^{i,\hat\balpha}| > 1/\sqrt{\epsilon} \right)
    \le C_g K_1 \sqrt{\epsilon},
\]
where the inequality holds by Lemma~\ref{lem MI bd Xi out} with $K_1$ depending only on $\EE|X_0|^2,T,\sigma$ and $\kappa$. 

The second term, thanks to Lemma~\ref{lem MI bd cost inside}, can be bounded from above by:
\begin{align*}
    &\EE\left[\indic_{\left\{\sup_{t \in [0,T]} |X_t^{i,\hat\balpha}| \le 1/\sqrt{\epsilon} \right\}} g(X_T^{i,\hat\balpha}, \mu^{\hat{\balpha}}_T, \mu_T^{\hat\balpha})\right] + \epsilon
    \\
    &= \EE\left[ g(X_T^{i,\hat\balpha}, \mu^{\hat{\balpha}}_T, \mu_T^{\hat\balpha})\right] 
    \\
    &\qquad - \EE\left[\indic_{\left\{\sup_{t \in [0,T]} |X_t^{i,\hat\balpha}| > 1/\sqrt{\epsilon} \right\}} g(X_T^{i,\hat\balpha}, \mu^{\hat{\balpha}}_T, \mu_T^{\hat\balpha})\right] + \epsilon
    \\
    &\le \EE\left[ g(X_T^{i,\hat\balpha}, \mu^{\hat{\balpha}}_T, \mu_T^{\hat\balpha})\right] + \|g\|_\infty\PP\left(\sup_{t \in [0,T]} |X_t^{i,\hat\balpha}| > 1/\sqrt{\epsilon} \right) + \epsilon
    \\
    &\le \EE\left[ g(X_T^{i,\hat\balpha}, \mu^{\hat{\balpha}}_T, \mu_T^{\hat\balpha})\right] + \|g\|_\infty K_1 \sqrt{\epsilon} + \epsilon,
\end{align*}
where $\mu^{\hat{\balpha}}_t$ is the law of $X^{\hat\balpha}_t$, controlled by $t \mapsto \hat\alpha(t,X^{\hat\balpha}_t)$, and we used again Lemma~\ref{lem MI bd Xi out} for the last inequality. We proceed similarly for the running cost $F$ (we omit the details for the sake of brevity). Collecting the terms, we obtain:
\begin{align}
\label{eq:MI-epsilon-FP-to-MF}
    J^{\mixind,N}(\hat\balpha; \hat{\bualpha}^{-i})
    &\le J^{\mixind}(\hat\balpha; \mu^{\hat{\balpha}}) + C \sqrt\epsilon,
\end{align}
with $C = T(C_FK_1 + 1) + C_gK_1 + 1$. 

\noindent{\bf Step 4: Conclusion. } Combining~\eqref{eq:MI-epsilon-MF-to-FP} and~\eqref{eq:MI-epsilon-FP-to-MF}, we deduce:
\[
    J^{\mixind,N}(\balpha^i; \hat{\bualpha}^{-i}) 
    \ge
    J^{\mixind,N}(\hat\balpha; \hat{\bualpha}^{-i})
    - C \sqrt\epsilon,
\]
with $C = [T(3C_FK_1 + 2) + 3C_gK_1 + 2]$. 

\end{proof}

\begin{lemma}[Bound on the state]
\label{lem MI bd Xi out}
    Suppose the Assumptions of Theorem~\ref{thm:MI-eps-Nash} hold and $\EE \int_0^T \frac{1}{2} |\alpha^i_t|^2 dt \le C_{a,max}$.
    There exists $K_1$ depending only on $\EE|X_0|^2$, $T$, $\sigma$ and $\kappa$ such that: for every $R>0$, 
    \[
    \PP\left( \sup_{t \in [0,T]} |X_t^{i,\balpha^i}| > R \right)
    \le K_1 \frac{1}{R}.
    \]
\end{lemma}
\begin{proof}{Proof}
    Let us notice that
\begin{align}
	&\EE\left[ \sup_{t \in [0,T]} |X_t^{i,\balpha^i}|\right]^2
    \notag
    \\
    &\le 
    \EE\left[ \sup_{t \in [0,T]} |X_t^{i,\balpha^i}|^2\right] 
	\notag
    \\
    &\leq 
	C \EE\left[ |X_0|^2 + \int_0^T \frac{1}{2} |\alpha^i_t|^2 dt + \sigma^2 \sup_{t \in [0,T]} |W^i_t|^2 \right] 
	\notag
    \\
	&\leq 
	C \left(\EE\left[ |X_0^{i,\balpha^i}|^2 \right] + \EE \int_0^T \frac{1}{2} |\alpha^i_t|^2 dt  + \sigma^2 \EE \left[\sup_{t \in [0,T]} |W^i_t|^2 \right] \right)
	\notag
    \\
	&\leq 
	C \left(\EE\left[ |X_0|^2 \right] + \kappa + \sigma^2 T \right),
    \label{eq:MI-exp-sup-Xi}
\end{align}
where the last inequality comes from~\eqref{eq:MI-bound-alphai-kappa} and Doob's inequality. The result is obtained by Markov's inequality. 

\end{proof}

\begin{lemma}[Cost perturbation]
\label{lem MI bd cost inside}
     Suppose the Assumptions of Theorem~\ref{thm:MI-eps-Nash} hold and $\EE \int_0^T \frac{1}{2} |\alpha^i_t|^2 dt \le C_{a,max}$. Let $\epsilon>0$. There exists $N_0$ such that for all $N>N_0$, for any probability distribution $\mu^I$, we have, for all $t \in [0,T)$,
    \begin{align*}
        \EE\left[ \sup_{|y| \le 1/\sqrt{\epsilon}} \left| F(y,\mu^{N,\balpha^i,\hat{\bualpha}^{-i}}_t,\mu^I) - F(y,\mu^{\hat{\balpha}}_t,\mu^I) \right| \right] &\le \epsilon,
        \\
        \EE\left[ \sup_{|y| \le 1/\sqrt{\epsilon}} \left| g(y,\mu^{N,\balpha^i,\hat{\bualpha}^{-i}}_T,\mu^I) - g(y,\mu^{\hat{\balpha}}_T,\mu^I) \right| \right] &\le \epsilon,
    \end{align*}
    where $\mu^{\hat{\balpha}}_t$ is the distribution of $X^{\hat\balpha}_t$. 
\end{lemma}
\begin{proof}{Proof}
    We start with $F$. Let us define, for all $t \in [0,T)$:
    \[
    	\mathbb{F}_t = \EE\left[ \sup_{|y| \le 1/\sqrt{\epsilon}} \left| F(y,\mu^{N,\balpha^i,\hat{\bualpha}^{-i}}_t,\mu^I) - F(y,\mu^{\hat{\balpha}}_t,\mu^I) \right| \right]. 
    \]
    Let $t \in [0,T)$. Using the Lipschitz continuity of $F$ with respect to $\mu$ uniformly in the other variables, we note that:
    \begin{align*}
        &\mathbb{F}_t
        \\
        & \le L_F \left(\EE\left[ W_1\left( \mu^{N,\balpha^i,\hat{\bualpha}^{-i}}_t, \mu^{N,\hat{\bualpha}}_t \right) \right] + \EE\left[ W_1\left( \mu^{N,\hat{\bualpha}}_t, \mu^{\hat{\bualpha}}_t \right) \right] \right)
        \\ 
        & \le L_F \left( \frac{1}{N} \sum_{j=1}^N \EE\left[ \left| X^{j,\balpha^i,\hat{\bualpha}^{-i}}_t - \tilde X^{j,\hat{\balpha}}_t \right| \right]  + \EE\left[ W_1\left( \mu^{N,\hat{\bualpha}}_t, \mu^{\hat{\bualpha}}_t \right) \right] \right),
    \end{align*}
    where $\mu^{N,\hat{\bualpha}}$ denotes the empirical distribution of  $(\tilde X^{j,\hat{\balpha}}_t)_{j=1}^N$, which are $N$ i.i.d. copies of $X^{\hat{\balpha}}_t$, with law $\mu^{\hat{\balpha}}_t$. The last inequality above uses the Kantorovich-Rubinstein dual representation of $W_1$. Furthermore, by~\cite[Theorem 1]{fournier2015rate}, $\EE\left[ W_1\left( \mu^{N,\hat{\bualpha}}_t, \mu^{\hat{\bualpha}}_t \right) \right]$ converges to $0$ as $N$ goes to infinity. 
    
    Next, we note that:
    \[
    \begin{aligned}
        &\frac{1}{N} \sum_{j=1}^N \EE\left[ \left| X^{j,\balpha^i,\hat{\bualpha}^{-i}}_t - \tilde X^{j,\hat{\balpha}}_t \right| \right] 
        \\
        &\le 
        \frac{N-1}{N} \frac{1}{N-1} \sum_{j \neq i} \EE\left[ \left| X^{j,\hat{\balpha}}_t - \tilde X^{j,\hat{\balpha}}_t \right| \right]
        \\
        &\qquad + \frac{1}{N} \EE\left[ \left| X^{i,\balpha^i}_t - \tilde X^{i,\hat{\balpha}}_t \right| \right].
    \end{aligned}
    \]
    For the first term, note that for $j\neq i$, $X^{j,\hat{\balpha}}_t$ and $\tilde X^{j,\hat{\balpha}}_t$ have the same law. So by the law of large numbers, there exists $N_0(t)$ such that for all $N \ge N_0(t)$, this term is smaller than $\epsilon/2$. For the second term, it is bounded by 
    \[
    \begin{aligned}
        &\frac{1}{N} (\EE[ | X^{i,\hat{\balpha}}_t | ] + \EE[ | \tilde X^{i,\hat{\balpha}}_t | ]) 
        \\
        &\le \frac{1}{N}  C \left(\EE\left[ |X_0|^2 \right] + C_{a,max} + \sigma^2 T \right)^{1/2},
    \end{aligned}
    \]
    where we used~\eqref{eq:MI-exp-sup-Xi} to obtain abound on the first moment of the state process.   
    Hence the first inequality in the statement.

    To finish the proof for $F$, we need to argue that $N_0(t)$ can be chosen independently of $t$. For this, we use the Lipschitz continuity of $F$. We have:
    \begin{align*}
        &\EE\left[ \sup_{|y| \le 1/\sqrt{\epsilon}} \left| F(y,\mu^{N,\balpha^i,\hat{\bualpha}^{-i}}_t,\mu^I) - F(y,\mu^{N,\balpha^i,\hat{\bualpha}^{-i}}_s,\mu^I) \right| \right]
        \\
        & \le L_F \EE\left[ W_1\left( \mu^{N,\balpha^i,\hat{\bualpha}^{-i}}_t, \mu^{N,\balpha^i,\hat{\bualpha}^{-i}}_s \right) \right]
        \\
        & \le L_F \frac{1}{N} \sum_{j=1}^N \EE\left[ \left| X^{j,\balpha^i,\hat{\bualpha}^{-i}}_t - X^{j,\balpha^i,\hat{\bualpha}^{-i}}_s \right| \right].
    \end{align*}
    Next, we note that when $j \neq i$,
    \begin{align*}
        &\EE\left[ \left| X^{j,\hat{\balpha}}_t - X^{j,\hat{\balpha}}_s \right| \right]
    	\\
        &\le 
    	\EE\int_s^t \left| \hat\alpha(r, X^{j,\hat{\balpha}}_r)\right| dr  + \sigma \EE|W^j_t - W^j_s| 
    	\\
    	&\le \left(\EE\int_0^T \left| \hat\alpha(r, X^{j,\hat{\balpha}}_r)\right|^2 dr\right)^{1/2}  \sqrt{t-s}  + \sigma \sqrt{t-s}
	\\
	&\le C \sqrt{t-s},
    \end{align*}
    where we used Cauchy-Schwarz inequality and~\eqref{eq:MI-eps-Nash-Camax}.  For $j=i$, we proceed similarly and use~\eqref{eq:MI-bound-alphai-kappa} instead of~\eqref{eq:MI-eps-Nash-Camax}. We deduce that:
    \begin{align*}
        &\EE\left[ \sup_{|y| \le 1/\sqrt{\epsilon}} \left| F(y,\mu^{N,\balpha^i,\hat{\bualpha}^{-i}}_t,\mu^I) - F(y,\mu^{N,\balpha^i,\hat{\bualpha}^{-i}}_s,\mu^I) \right| \right]
        \\
        &\le C \sqrt{t-s}.
    \end{align*}
    A similar result holds when $\mu^{N,\balpha^i,\hat{\bualpha}^{-i}}$ is replaced by $\mu^{\hat{\balpha}}$. 
    As a consequence, 
    \begin{align*}
    &|\mathbb{F}_t - \mathbb{F}_s|
	\\
    &\le \EE\left[ \sup_{|y| \le 1/\sqrt{\epsilon}} \left| F(y,\mu^{N,\balpha^i,\hat{\bualpha}^{-i}}_t,\mu^I) - F(y,\mu^{N,\balpha^i,\hat{\bualpha}^{-i}}_s,\mu^I) \right| \right]
	\\
	&\qquad\qquad + \EE\left[ \sup_{|y| \le 1/\sqrt{\epsilon}} \left| F(y,\mu^{\hat{\balpha}}_t,\mu^I) - F(y,\mu^{\hat{\balpha}}_s,\mu^I) \right| \right]
	\\
	&\le C \sqrt{t-s}. 
    \end{align*}
    We deduce that it is possible to find $N_0$ independent of $t$ such that, for all $t \in [0,T]$, $\mathbb{F}_t \le \epsilon$.

    The inequality for $g$ in the statement is proved analogously.

\end{proof}

\subsection{Equilibrium characterization}
\label{app:MI-Pontryagin}

We start by writing in detail the Assumption~{\bf Pontryagin Optimality} of~\cite[p.542-543]{CarmonaDelarue_book_I} adapted to the MI setting are satisfied for $b,\sigma,f,g$:
\begin{itemize}
    \item The functions \( b \) and \( f \) are differentiable with respect to \( (x, \alpha) \), the mappings \( (x, \alpha, \mu, \mu') \mapsto \partial_x(b, f)(t, x, \alpha, \mu, \mu') \) and \( (x, \alpha, \mu, \mu') \mapsto \partial_\alpha(b, f)(t, x, \alpha, \mu, \mu') \) being continuous for each \( t \in [0, T] \). The functions \( b \) and \( f \) are also differentiable with respect to the variable \( \mu' \), the mapping \( (x, \alpha, \mu, X') \mapsto \partial_{\mu'}(b,  f)(t, x, \alpha, \mu, \mathcal{L}(X'))(X') \) being continuous for each \( t \in [0, T] \). Similarly, the function \( g \) is differentiable with respect to \( x \), the mapping \( (x, \mu, \mu') \mapsto \partial_x g(x, \mu, \mu') \) being continuous. 
    The function \( g \) is also differentiable with respect to the variable \( \mu' \), the mapping \( (x, \mu, X') \mapsto \partial_{\mu'} g(x, \mu, \mathcal{L}(X'))(X') \) being continuous.
    \item The function \( [0, T] \ni t \mapsto (b, f)(t, 0, 0, \delta_0, \delta_0) \) is uniformly bounded. The derivatives \( \partial_x b \) and \( \partial_\alpha b \) are uniformly bounded and the mapping \( x' \mapsto \partial_{\mu'}b (t, x, \alpha, \mu, \mu')(x') \) has an \( L^2(\mathbb{R}^d; \mu'; \mathbb{R}^d) \)-norm which is also uniformly bounded (i.e., uniformly in \( (t, x, \alpha, \mu, \mu') \)). There exists a constant \( L \) such that, for any \( R \geq 0 \) and any \( (t, x, \alpha, \mu, \mu') \) such that \( |x|, |\alpha|, M_2(\mu), M_2(\mu') \leq R \), \( |\partial_x f(t, x, \alpha, \mu, \mu')| \), \( |\partial_x g(x, \mu)| \), and \( |\partial_{\mu'} f(t, x, \alpha, \mu, \mu')(x')| \) are bounded by \( L(1 + R) \) and the \( L^2(\mathbb{R}^d; \mu'; \mathbb{R}^d) \)-norms of \( x' \mapsto \partial_{\mu'} f(x, \alpha, \mu, \mu')(x') \) and \( x' \mapsto \partial_{\mu'} g(x, \mu, \mu')(x') \) are bounded by \( L(1 + R) \).
\end{itemize}
For the sufficient condition, we also assume:
\begin{itemize}
    \item $(x,\mu,\mu') \mapsto g(x,\mu,\mu')$ is convex;
    \item $(x,\alpha,\mu,\mu') \mapsto H^\mixind(t, x,  \alpha, \mu, \mu', Y^{\hat\balpha}_t, Z^{\hat\balpha}_t)$ is convex $\mathrm{Leb}_1 \otimes \PP$-a.e..
\end{itemize}

\begin{proof}{Proof of Theorem~\ref{thm:MI-Pontryagin}}
First, we consider the MKV control problem for a representative agent, with the population distribution $\hat\bmu$ fixed. By~\cite[Proposition 6.15 and Theorem 6.16]{CarmonaDelarue_book_I}, we obtain that the control is optimal if and only if:
\[
    H^\mixind(t, X^{\hat\balpha}_t, \hat{\alpha}_t, \hat\mu_t, \mu^{\hat\balpha}_t, Y^{\hat\balpha}_t, Z^{\hat\balpha}_t) = \inf_\alpha H^\mixind(t, X^{\hat\balpha}_t, \alpha, \hat\mu_t, \mu^{\hat\balpha}_t, Y^{\hat\balpha}_t, Z^{\hat\balpha}_t),
\]
with
\begin{equation*}
    \begin{cases}
        dX^{\hat\balpha}_t = b(t, X^{\hat\balpha}_t, \hat{\alpha}_t, \hat\mu_t, \mu^{\hat\balpha}_t) dt + \sigma dW_t
        \\
        dY^{\hat\balpha}_t = \Big(- \partial_x H^\mixind(t, X^{\hat\balpha}_t, \hat{\alpha}_t, \hat\mu_t, \mu^{\hat\balpha}_t, Y^{\hat\balpha}_t, Z^{\hat\balpha}_t) 
        \\ \quad - \tilde\EE\left[\partial_{\mu_2} H^\mixind(t, \tilde{X}^{\hat\balpha}_t,  \tilde{\hat{\alpha}}_t, \hat\mu_t, \mu^{\hat\balpha}_t, \tilde{Y}^{\hat\balpha}_t, \tilde{Z}^{\hat\balpha}_t)(X^{\hat\balpha}_t) \right] \Big) dt 
        \\ \quad + Z_t dW_t
        \\
        X^{\hat\balpha}_0 \sim \mu_0, \qquad Y^{\hat\balpha}_T = g(X^{\hat\balpha}_T, \hat\mu_T, \mu^{\hat\balpha}_T),
    \end{cases}
\end{equation*}
where $\mu^{\hat\balpha}_t = \cL(X^{\hat\balpha}_t)$.

From here, we use the consistency condition in the Nash equilibrium definition implies that $\bmu^{\hat\balpha}$ and $\hat\bmu$ must be equal, which yields the system~\eqref{eq:FBSDE-MI}.
\end{proof}

\section{Proofs for MP mean field model}\label{app:MP-proofs}

\subsection{Approximate equilibrium}
\label{app:MP-eps-Nash}

We start by stating the technical conditions that we will use in the proof:
\begin{itemize}
    \item The drifts are of the form: $b^\rmnc(t,x,a,\mu,\mu') = a$, $b^\rmc(t,x,a,\mu,\mu') = a$
    \item The running costs are of the form: $f^\rmnc(t,x,a,\mu,\mu') = \frac{1}{2}|a|^2 + F^\rmnc(x,\mu,\mu')$ and $f^\rmc(t,x,a,\mu,\mu') = \frac{1}{2}|a|^2 + F^\rmc(x,\mu,\mu')$
	\item The maps $F^\rmnc$ and $F^\rmc$ are $L_{F}$-Lipschitz continuous with respect to $(x,\mu,\mu')$ in the following sense: for any $(x,\mu,\mu')$ and $(\tilde x,\tilde \mu,\tilde\mu')$
    \[
    \begin{aligned}
	&|\varphi(x,\mu,\tilde\mu) - \varphi(\tilde x,\tilde \mu,\tilde\mu')| 
    \\
    &\leq L_F [|x-\tilde x| + W_1(\mu,\tilde \mu) + W_1(\mu',\tilde\mu')]
    \end{aligned}
    \]
for $\varphi \in \{F^\rmnc, F^\rmc\}$, 
where $W_1$ is the $1$-Wasserstein distance. Likewise, the terminal costs $g^\rmnc$ and $g^\rmc$ are Lipschitz continuous in $(x,\mu,\tilde\mu)$ with Lipschitz constant $L_{g}$. 
	\item The maps $F^\rmnc, F^\rmc, g^\rmnc$ and $g^\rmc$ are bounded: $C_{F^\rmnc} := \|F^\rmnc\|_\infty$, $C_{F^\rmc} := \|F^\rmc\|_\infty$, $C_{g^\rmnc} := \|g^\rmnc\|_\infty$ and $C_{g^\rmc} := \|g^\rmc\|_\infty$ are finite.
	\item The initial distribution has a finite second moment: $\EE|X_0|^2<+\infty$.
\end{itemize}

\begin{proof}{Proof of Theorem~\ref{thm:MP-eps-Nash}}

We denote by $\hat{\bualpha}^\rmnc = (\hat\bualpha^{j,\rmnc})_{j=1}^{N^\rmnc}$ the control profile where $\hat\alpha^{j,\rmnc}_t = \hat\alpha^\rmnc(t,X^{\rmnc, j,\hat\bualpha^\rmnc,\hat\bualpha^\rmc}_t)$, and likewise for the cooperative control profile, $\hat{\bualpha}^\rmc$. For simplicity, we will denote $(\epsilon^\rmnc(N^\rmnc,N^\rmc), \epsilon^\rmc(N^\rmnc,N^\rmc)) = (\epsilon^\rmnc, \epsilon^\rmc)$. We want to show that: 
for every $i \in [N^\rmnc]$ and every control $\balpha^{i,\rmnc}$, 
    \begin{equation}
    \label{eq:MP-eps-Nash-NC}
    \begin{aligned}
        &J^{\mixpop,N,\rmnc}(\hat\balpha^{i,\rmnc};\hat\bualpha^{-i,\rmnc}, \hat\bualpha^{\rmc})
        \\
        &\le 
        J^{\mixpop,N,\rmnc}(\balpha^{i,\rmnc};\hat\bualpha^{-i,\rmnc}, \hat\bualpha^{\rmc}) + \epsilon^\rmnc;
    \end{aligned}
    \end{equation}
    and for every control profile $\bualpha^{\rmc} = (\bualpha^{1, \rmc}, \dots, \bualpha^{N^\rmc, \rmc})$, 
    \begin{equation}
    \label{eq:MP-eps-Nash-C}
        J^{\mixpop,N,\rmc}(\hat\bualpha^{\rmc}; \hat\bualpha^{\rmnc})
        \le 
        J^{\mixpop,N,\rmc}(\bualpha^{\rmc}; \hat\bualpha^{\rmnc}) + \epsilon^\rmc.
    \end{equation}

Note that, since we assume that the drift is the control, $X_t^{i,\rmnc,\bualpha^{\rmnc}, \bualpha^{\rmc}} = X_t^{i,\rmnc,\balpha^{i,\rmnc}}$ does not depend on other agents' controls, and likewise  $X_t^{i,\rmc,\bualpha^{\rmnc}, \bualpha^{\rmc}} = X_t^{i,\rmc,\balpha^{i,\rmc}}$. As a consequence, the empirical distributions can be denoted: $\mu^{N,\rmnc, \bualpha^{\rmnc}, \bualpha^{\rmc}}_t = \mu^{N,\rmnc,\bualpha^{\rmnc}}_t$ and  $\mu_t^{N,\rmc, \bualpha^{\rmnc}, \bualpha^{\rmc}} = \mu_t^{N,\rmc,\bualpha^{\rmc}}$. 
Then the costs rewrite:
\begin{equation}
\begin{aligned}
    &J^{\mixpop,N,\rmnc}(\balpha^{i,\rmnc};\bualpha^{-i,\rmnc}, \bualpha^{\rmc}) 
    \\
    &= \EE \Big[\int_0^T f^{\rmnc}(t, X_t^{i,\rmnc,\balpha^{i,\rmnc}}, \alpha^{i,\rmnc}_t, \mu^{N,\rmnc, \bualpha^{\rmnc}}_t, \mu_t^{N,\rmc, \bualpha^{\rmc}})dt 
    \\
    &\qquad\qquad + g^{\rmnc}(X_T^{i,\rmnc,\balpha^{i,\rmnc}}, \mu^{N,\rmnc,\bualpha^{\rmnc}}_T, \mu_T^{N,\rmc, \bualpha^{\rmc}})\Big].
\end{aligned}
\end{equation}
and:
\begin{equation}
\begin{aligned}
    &J^{\mixpop,N,\rmc}(\bualpha^{\rmc};\bualpha^{\rmnc}) 
    \\
    &= \frac{1}{N^\rmc} \sum_{i=1}^{N^\rmc} 
    \EE \Big[\int_0^T f^{\rmc}(t, X_t^{i,\rmc, \balpha^{i,\rmc}}, \alpha^{i,\rmc}_t, \mu^{N,\rmnc, \bualpha^{\rmnc}}_t, \mu_t^{N,\rmc, \bualpha^{\rmc}})dt 
    \\
    &\qquad\qquad\qquad\qquad + g^{\rmc}(X_T^{i,\rmc,\balpha^{i,\rmc}}, \mu^{N,\rmnc,\bualpha^{\rmnc}}_T, \mu_T^{N,\rmc, \bualpha^{\rmc}})\Big].
\end{aligned}
\end{equation}

\noindent{\bf Step 1. Non-cooperative population. }  
The analysis of the non-cooperative population is very similar to the proof of Theorem~\ref{thm:MI-eps-Nash}, for the mixed-individual setting, except for the fact that the individual distribution is replaced by the empirical distribution of the cooperative population. The follow the same proof structure but with the modifications induced by this remark.

Fix $i \in [N^\rmnc]$ and consider a control $\balpha^{i,\rmnc}$.

\noindent{\bf Step 1.1: Bound on the controls. }
Note that, by optimality of $\hat\balpha^\rmnc$, 
\[
	J^{\mixpop,\rmnc}(\hat\balpha^{\rmnc};\bmu^{\rmnc,\hat\balpha^{\rmnc}}, \bmu^{\rmc,\hat\balpha^{\rmc}})
	\le J^{\mixpop,\rmnc}(0;\bmu^{\rmnc,\hat\balpha^{\rmnc}}, \bmu^{\rmc,\hat\balpha^{\rmc}}),
\]
where $0$ denotes the control which is identically $0$. The left-hand side is bounded from below as:
\[
\begin{aligned}
	&\EE\int_0^T \frac{1}{2}|\hat\alpha^\rmnc(t,X^{\hat\balpha^\rmnc}_t)|^2dt - (TC_{F^\rmnc} + C_{g^\rmnc}) 
	\\
    &\le
	J^{\mixpop,\rmnc}(\hat\balpha^{\rmnc};\bmu^{\rmnc,\hat\balpha^{\rmnc}}, \bmu^{\rmc,\hat\balpha^{\rmc}})
\end{aligned}
\]
while the right-hand side is bounded from above by $(TC_{F^\rmnc} + C_{g^\rmnc})$. 
As a consequence,
\begin{equation}
\label{eq:MP-eps-Nash-Camax-NC}
    \EE\int_0^T \frac{1}{2}|\hat\alpha^\rmnc(t,X^{\hat\balpha^\rmnc}_t)|^2dt
    \le 2 (TC_{F^\rmnc} + C_{g^\rmnc}) =: C_{a,max}^\rmnc. 
\end{equation}
Without loss of generality, we can assume that 
\begin{equation}
    \label{eq:MP-bound-alphai-kappa-NC}
    \EE \int_0^T \frac{1}{2}|\alpha^{i,^\rmnc}_t|^2 dt 
    \le \epsilon +  (TC_{F^\rmnc} + C_{g^\rmnc}) + C_{a,max}^\rmnc =: \kappa^\rmnc.
\end{equation}
Indeed, if the above inequality is not true, then $J^{\mixpop,N,\rmnc}(\balpha^{i,\rmnc};\hat\bualpha^{-i,\rmnc}, \hat\bualpha^{\rmc})$ is large enough that~\eqref{eq:MP-eps-Nash-NC} holds, which concludes the proof immediately. 

\vskip 6pt

\noindent{\bf Step 1.2: Difference between the finite population cost and the mean field cost. } Proceeding as in the proof of Theorem~\ref{thm:MI-eps-Nash}, we can show that: for any control $\balpha^{i,\rmnc}$, 
\begin{equation}
\label{eq:mp-NC-diff-finite-mf-costs}
\begin{aligned}
    &\Big|J^{\mixpop,N,\rmnc}(\balpha^{i,\rmnc};\hat\bualpha^{-i,\rmnc}, \hat\bualpha^{\rmc})
    \\
    &\quad - 
    J^{\mixpop,\rmnc}(\balpha^{i,\rmnc};\bmu^{\rmnc,\hat\balpha^{\rmnc}}, \bmu^{\rmc,\hat\balpha^{\rmc}})\Big| \le C_1 \epsilon,
\end{aligned}
\end{equation}
Indeed, Lemma~\ref{lem MI bd Xi out} can be used as such because the dynamics is the same. Lemma~\ref{lem MI bd cost inside} needs to be adapted as follows.
\begin{lemma}[Cost perturbation]
\label{lem MP bd cost inside}
     Suppose the Assumptions of Theorem~\ref{thm:MP-eps-Nash} hold and $\EE \int_0^T \frac{1}{2} |\alpha^{i,\rmnc}_t|^2 dt \le C_{a,max}$. Let $\epsilon>0$. There exists $N_0$ such that for all $N^\rmnc>N_0$ and $N^\rmc>N_0$, 
    \begin{align*}
        \EE\left[ \sup_{|y| \le 1/\sqrt{\epsilon}} \left| F^{\rmnc}(y,\check\mu^{i,\rmnc}_t,\hat\mu^{N,\rmc}_t) - F^{\rmnc}(y,\hat\mu^{\rmnc}_t,\hat\mu^{\rmc}_t) \right| \right] &\le \epsilon,
        \\
        \EE\left[ \sup_{|y| \le 1/\sqrt{\epsilon}} \left| g^{\rmnc}(y,\check\mu^{i,\rmnc}_T,\hat\mu^{N,\rmc}_T) - g^{\rmnc}(y,\hat\mu^{\rmnc}_T,\hat\mu^{\rmc}_T) \right| \right] &\le \epsilon,
    \end{align*}
    where we use the shorthand notations $(\check\mu^{i,\rmnc}_t,\hat\mu^{N,\rmc}_t) = (\mu^{N,\rmnc,\balpha^{i,\rmnc},\hat{\bualpha}^{-i,\rmnc}}_t,\hat\mu^{N,\rmc}_t)$, $\hat\mu^{\rmnc}_t = \mu^{\rmnc,\hat{\balpha}^{\rmnc}}_t$, $\hat\mu^{\rmc}_t = \mu^{\rmc,\hat{\balpha}^{\rmc}}_t$, $\hat\mu^{N,\rmc}_t = \mu^{N,\rmc,\hat{\balpha}^{\rmc}}_t$
     and 
    where $\mu^{\rmnc, \hat{\balpha}^\rmnc}_t$ (resp. $\mu^{\rmc, \hat{\balpha}^\rmc}_t$) is the distribution of $X^{\hat\balpha^\rmnc}_t$ (resp. $X^{\hat\balpha^\rmc}_t$). 
\end{lemma}

\begin{proof}{Proof}
            We start with $F$. Let us define:
    \[
    	\mathbb{F}_t = \EE\left[ \sup_{|y| \le 1/\sqrt{\epsilon}} \left| F^{\rmnc}(y,\check\mu^{i,\rmnc}_t,\hat\mu^{N,\rmc}_t) - F^{\rmnc}(y,\hat\mu^{\rmnc}_t, \hat\mu^{\rmc}_t) \right| \right], \qquad t \in [0,T). 
    \]
    Let $t \in [0,T)$. Using the Lipschitz continuity of $F$ with respect to $\mu^\rmnc$ and $\mu^\rmc$ uniformly in $y$, we note that:
    \begin{align*}
        \mathbb{F}_t
        & \le L_F \EE\left[ W_1\left( \check\mu^{i,\rmnc}_t, \hat\mu^{N,\rmnc}_t \right) + W_1\left( \hat\mu^{N,\rmnc}_t, \hat\mu^{\rmnc}_t \right)
        + W_1\left( \hat\mu^{N,\rmc}_t, \hat\mu^{\rmc}_t \right) \right]
        \\ 
        & \le L_F \left(\frac{1}{N^\rmnc} \sum_{j=1}^{N^\rmnc} \EE\left[ \left| X^{j,\balpha^{i,\rmnc},\hat{\bualpha}^{-i,\rmnc}}_t - \tilde X^{j,\hat{\balpha}^{\rmnc}}_t \right| \right] 
        + \EE\left[ W_1\left( \hat\mu^{N,\rmnc}_t, \hat\mu^{\rmnc}_t \right)
        + W_1\left( \hat\mu^{N,\rmc}_t, \hat\mu^{\rmc}_t \right) \right] \right),
    \end{align*}
    where $\hat\mu^{N,\rmnc}_t$ denotes the empirical distribution of  $(\tilde X^{j,\hat{\balpha}^\rmnc}_t)_{j=1}^N$, which are $N$ i.i.d. copies of $X^{\hat{\balpha}^\rmnc}_t$, with law $\mu^{\hat{\balpha}}_t$. The last inequality above uses the Kantorovich-Rubinstein dual representation of $W_1$. Furthermore, by~\cite[Theorem 1]{fournier2015rate}, $\EE\left[ W_1\left( \hat\mu^{N,\rmnc}_t, \hat\mu^{\rmnc}_t \right)
        + W_1\left( \hat\mu^{N,\rmc}_t, \hat\mu^{\rmc}_t \right) \right]$ converges to $0$ as $N^\rmnc$ and $N^\rmc$ go to infinity. 
    
    Proceeding as in the proof of Lemma~\ref{lem MI bd cost inside}, we can show that $\frac{1}{N^\rmnc} \sum_{j=1}^{N^\rmnc} \EE\left[ \left| X^{j,\balpha^{i,\rmnc},\hat{\bualpha}^{-i,\rmnc}}_t - \tilde X^{j,\hat{\balpha}^{\rmnc}}_t \right| \right]$ goes to $0$ as $N^\rmnc$ goes to infinity, and we can also show that the choice of $N_0^\rmnc$ and $N_0^\rmc$ is independent of the time $t$.  
        
    The inequality for $g$ in the statement is proved analogously. 
    
\end{proof}

\noindent{\bf Step 1.3: Conclusion. } 
From~\eqref{eq:mp-NC-diff-finite-mf-costs} we will deduce that:
\[
\begin{aligned}
    &J^{\mixpop,N,\rmnc}(\balpha^{i,\rmnc};\hat\bualpha^{-i,\rmnc}, \hat\bualpha^{\rmc})
    \\
    &\ge 
    J^{\mixpop,\rmnc}(\balpha^{i,\rmnc};\bmu^{\rmnc,\hat\balpha^{\rmnc}}, \bmu^{\rmc,\hat\balpha^{\rmc}}) - C_1 \epsilon,
\end{aligned}
\]
and (taking $\balpha^{i,\rmnc}$ to be the optimal control in~\eqref{eq:mp-NC-diff-finite-mf-costs}):
\[
\begin{aligned}
    &J^{\mixpop,\rmnc}(\hat\balpha^{\rmnc};\bmu^{\rmnc,\hat\balpha^{\rmnc}}, \bmu^{\rmc,\hat\balpha^{\rmc}}) 
    \\
    &\ge 
    J^{\mixpop,N,\rmnc}(\hat\balpha^{\rmnc};\hat\bualpha^{-i,\rmnc}, \hat\bualpha^{\rmc}) - C_1 \epsilon.
\end{aligned}
\]
Combining the last two inequalities above gives:
\[
\begin{aligned}
    &J^{\mixpop,N,\rmnc}(\balpha^{i,\rmnc};\hat\bualpha^{-i,\rmnc}, \hat\bualpha^{\rmc})
    \\
    &\ge 
    J^{\mixpop,N,\rmnc}(\hat\balpha^{\rmnc};\hat\bualpha^{-i,\rmnc}, \hat\bualpha^{\rmc}) - 2C_1 \epsilon.
\end{aligned}
\]

\noindent{\bf Step 2. Cooperative population. } 
We now want to prove~\eqref{eq:MP-eps-Nash-C}. The proof boils down to showing that the costs are Lipschitz in the two distributions.  Here, there are no individual deviations because all the cooperative agents use the same control. The non-cooperative agents do not change their control in this part of the analysis. 

We prove the following result.
\begin{lemma}
\label{lem:MP-C-finite-meanfield-costs}
Under the assumptions of Theorem~\ref{thm:MP-eps-Nash}, there exists $N_0$ such that for all $N^\rmnc \ge N_0$ and $N^\rmc \ge N_0$: for any $\balpha^{\rmc}$, 
\begin{equation}
        |J^{\mixpop,\rmc}(\balpha^{\rmc}; \hat\balpha^{\rmnc})
        - 
        J^{\mixpop,N,\rmc}(\bualpha^{\rmc}; \hat\bualpha^{\rmnc})| \le C \epsilon^{\rmc}.
\end{equation}
\end{lemma}
From here, we will use the inequality in Lemma~\ref{lem:MP-C-finite-meanfield-costs} for $\balpha^{\rmc}$ and $\hat\balpha^{\rmc}$, as well as the optimality of $\hat\balpha^{\rmc}$ in the mean field problem to conclude that: for any $\balpha^{\rmc}$,
\begin{equation}
        J^{\mixpop,N,\rmc}(\hat\bualpha^{\rmc}; \hat\bualpha^{\rmnc})
        \le 
        J^{\mixpop,N,\rmc}(\bualpha^{\rmc}; \hat\bualpha^{\rmnc}) + C \epsilon^{\rmc},
\end{equation}
so $\hat\bualpha^{\rmc}$ is approximately optimal for the cooperative population in the finite-agent game.

\begin{proof}[Proof of Lemma~\ref{lem:MP-C-finite-meanfield-costs}]
Let us introduce the shorthand notations:
\[
\begin{aligned}
    \cG^{\rmc}(\mu^{\rmnc}, \mu^{\rmc}) 
    &= \EE_{X^\rmc \sim \mu^{\rmc}}\left[ g^{\rmc}(X^\rmc, \mu^{\rmnc}, \mu^{\rmc} ) \right]
    \\
    &= \langle \mu^\rmc, g^{\rmc}(\cdot, \mu^{\rmnc}, \mu^{\rmc} ) \rangle.
\end{aligned}
\]
and likewise for $\cF^{\rmc}$ in terms of $F^{\rmnc}$. 

Then, we can rewrite the costs as:
\begin{align*}
    &J^{\mixpop,N,\rmc}(\bualpha^{\rmc};\bualpha^{\rmnc}) 
    \\
    &= \frac{1}{N^\rmc} \sum_{i=1}^{N^\rmc} 
    \EE \Big[\int_0^T f^{\rmc}(t, X_t^{i,\rmc, \balpha^{i,\rmc}}, \alpha^{i,\rmc}_t, \mu^{N,\rmnc, \bualpha^{\rmnc}}_t, \mu_t^{N,\rmc, \bualpha^{\rmc}})dt 
    \\
    &\qquad\qquad\qquad\qquad + g^{\rmc}(X_T^{i,\rmc,\balpha^{i,\rmc}}, \mu^{N,\rmnc,\bualpha^{\rmnc}}_T, \mu_T^{N,\rmc, \bualpha^{\rmc}})\Big]
    \\
    &= 
    \int_0^T \left[ \frac{1}{N^\rmc} \sum_{i=1}^{N^\rmc}  \frac{1}{2} \EE\left[ |\alpha^{i,\rmc}_t|^2 \right] 
    +
     \cF^{\rmc}(\mu^{N,\rmnc, \bualpha^{\rmnc}}_t, \mu_t^{N,\rmc, \bualpha^{\rmc}}) \right] dt
    \\
    &\qquad + \cG^{\rmc}(\mu^{N,\rmnc,\bualpha^{\rmnc}}_T, \mu_T^{N,\rmc, \bualpha^{\rmc}}),
\end{align*}
where for the control we used the fact that the trajectories are i.i.d. so $\EE\left[ |\alpha^{i,\rmc}_t|^2 \right] = \EE\left[ |\alpha^{\rmc}_t|^2 \right]$ for every $i$, 
and
\begin{align*}
    &J^{\mixpop,\rmc}(\balpha^{\rmc};\balpha^{\rmnc}) 
    \\
    &= 
    \EE \Big[\int_0^T f^{\rmc}(t, X_t^{\rmc, \balpha^{\rmc}}, \alpha^{\rmc}_t, \mu^{\rmnc, \balpha^{\rmnc}}_t, \mu_t^{\rmc, \balpha^{\rmc}})dt 
    \\
    &\qquad\qquad\qquad\qquad + g^{\rmc}(X_T^{\rmc,\balpha^{\rmc}}, \mu^{\rmnc,\balpha^{\rmnc}}_T, \mu_T^{\rmc, \balpha^{\rmc}})\Big]
     \\
    &= 
    \int_0^T \left[ \frac{1}{2} \EE\left[ |\alpha^{\rmc}_t|^2 \right] 
    +
     \cF^{\rmc}(\mu^{\rmnc, \bualpha^{\rmnc}}_t, \mu_t^{\rmc, \bualpha^{\rmc}}) \right] dt
    \\
    &\qquad + \cG^{\rmc}(\mu^{\rmnc,\bualpha^{\rmnc}}_T, \mu_T^{\rmc, \bualpha^{\rmc}}).
\end{align*}
We now bound the difference between the absolute value of the two costs.

We start with the terms involving $\cG^\rmc$. We note that $\cG^{\rmc}$ is Lipschitz continuous with respect to $W_1$ because:
\begin{align}
    &|\cG^{\rmc}(\mu^{\rmnc}_1, \mu^{\rmc}_1) - \cG^{\rmc}(\mu^{\rmnc}_2, \mu^{\rmc}_2)|
    \le |\langle \mu^\rmc_1 - \mu^\rmc_2, g^{\rmc}(\cdot, \mu^{\rmnc}_1, \mu^{\rmc}_1 ) \rangle|
    \\
    &\qquad\qquad + |\langle \mu^\rmc_2, g^{\rmc}(\cdot, \mu^{\rmnc}_1, \mu^{\rmc}_1 ) - g^{\rmc}(\cdot, \mu^{\rmnc}_2, \mu^{\rmc}_2 ) \rangle|
    \\
    & \le L_g W_1(\mu^\rmc_1, \mu^\rmc_2) + \langle \mu^\rmc_2, |g^{\rmc}(\cdot, \mu^{\rmnc}_1, \mu^{\rmc}_1 ) - g^{\rmc}(\cdot, \mu^{\rmnc}_2, \mu^{\rmc}_2 )| \rangle
    \\
    & \le L_g W_1(\mu^\rmc_1, \mu^\rmc_2) + \langle \mu^\rmc_2, |g^{\rmc}(\cdot, \mu^{\rmnc}_1, \mu^{\rmc}_1 ) - g^{\rmc}(\cdot, \mu^{\rmnc}_2, \mu^{\rmc}_2 )| \rangle
    \\
    & \le L_g W_1(\mu^\rmc_1, \mu^\rmc_2) + L_g \Big[ W_1(\mu^{\rmnc}_1, \mu^{\rmnc}_2) + W_1( \mu^{\rmc}_1, \mu^{\rmc}_2 ) \Big],
\end{align}
using the assumption that $g^\rmc$ is Lipschitz with respect to $x$ uniformly in $(\mu,\tilde\mu)$, and it is Lipschitz in $(\mu,\tilde\mu)$ (for the Wasserstein-1 distance) uniformly in $x$. We denote the Lipschitz constant of $\cG^\rmc$ by $L_\cG$.

Then, for the terminal cost, we have:
\begin{align*}
    &\left| \frac{1}{N^\rmc} \sum_{i=1}^{N^\rmc}  g^{\rmc}(X_T^{i,\rmc,\balpha^{i,\rmc}}, \mu^{N,\rmnc,\bualpha^{\rmnc}}_T, \mu_T^{N,\rmc, \bualpha^{\rmc}}) 
    - \EE\left[g^{\rmc}(X_T^{\rmc,\balpha^{\rmc}}, \mu^{\rmnc,\balpha^{\rmnc}}_T, \mu_T^{\rmc, \balpha^{\rmc}}) \right]
    \right|
    \\
    &= \left| \cG^{\rmc}(\mu^{N,\rmnc,\bualpha^{\rmnc}}_T, \mu_T^{N,\rmc, \bualpha^{\rmc}}) 
    - \cG^{\rmc}(\mu^{\rmnc,\balpha^{\rmnc}}_T, \mu_T^{\rmc, \balpha^{\rmc}}) 
    \right|
    \\
    & \le L_\cG \Big( W_1(\mu^{N,\rmnc,\bualpha^{\rmnc}}_T, \mu^{\rmnc,\balpha^{\rmnc}}_T) + W_1(\mu_T^{N,\rmc, \bualpha^{\rmc}}, \mu_T^{\rmc, \balpha^{\rmc}}) \Big)
    \\
    & \le L_\cG \epsilon'',
\end{align*}
where we used the Kantorovich-Rubinstein dual representation of the Wasserstein-1 distance for the penultimate inequality, and the last inequality is by the law of large numbers, because $\mu^{N,\rmnc,\bualpha^{\rmnc}}_T$ consists of $N$ i.i.d. samples with law $\mu^{\rmnc,\balpha^{\rmnc}}_T$, and similarly $\mu_T^{N,\rmc, \bualpha^{\rmc}}$ consists of $N$ i.i.d. samples with law $\mu_T^{\rmc, \balpha^{\rmc}}$.

We then proceed similarly for the running cost. 
Notice that $\cF^{\rmc}$ is Lipschitz continuous with respect to $W_1$ and we denote by $L_\cF$ it's Lipschitz constant. Then:
\begin{align}
    &\left| \frac{1}{N^\rmc} \sum_{i=1}^{N^\rmc}  F^{\rmc}(X_t^{i,\rmc,\balpha^{i,\rmc}}, \mu^{N,\rmnc,\bualpha^{\rmnc}}_t, \mu_t^{N,\rmc, \bualpha^{\rmc}}) 
    - \EE\left[F^{\rmc}(X_t^{\rmc,\balpha^{\rmc}}, \mu^{\rmnc,\balpha^{\rmnc}}_t, \mu_t^{\rmc, \balpha^{\rmc}}) \right]
    \right|
    \\
    &= \left| \cF^{\rmc}(\mu^{N,\rmnc,\bualpha^{\rmnc}}_t, \mu_t^{N,\rmc, \bualpha^{\rmc}}) 
    - \cF^{\rmc}(\mu^{\rmnc,\balpha^{\rmnc}}_t, \mu_t^{\rmc, \balpha^{\rmc}}) 
    \right|
    \\
    & \le L_\cF \Big( W_1(\mu^{N,\rmnc,\bualpha^{\rmnc}}_t, \mu^{\rmnc,\balpha^{\rmnc}}_t) + W_1(\mu_t^{N,\rmc, \bualpha^{\rmc}}, \mu_t^{\rmc, \balpha^{\rmc}}) \Big)
    \\
    & \le L_\cF \epsilon'',
\end{align}

Now we estimate:
\begin{align*}
        &|J^{\mixpop,\rmc}(\balpha^{\rmc}; \balpha^{\rmnc})
        - J^{\mixpop,N,\rmc}(\bualpha^{\rmc}; \bualpha^{\rmnc})| 
        \\
        &\le 
         \Big[ \Big| \frac{C}{N^\rmc} \sum_{i=1}^{N^\rmc} \EE[|\balpha^{i,\rmc} |^2] - \EE [|\balpha^{\rmc} | ] \Big| \Big] 
        \\
        &\quad + C 
        \EE \left[\int_0^T \Big|\cF^{\rmc}(t, \mu^{N,\rmnc, \bualpha^{\rmnc}}_t, \mu_t^{N,\rmc, \bualpha^{\rmc}}) - \cF^{\rmc}(t, \mu^{\rmnc, \bualpha^{\rmnc}}_t, \mu_t^{\rmc, \bualpha^{\rmc}})\Big|dt \right]
        \\
        &\qquad\qquad\qquad\qquad + \EE \Big[\Big|\cG^{\rmc}(\mu^{N,\rmnc,\bualpha^{\rmnc}}_T, \mu_T^{N,\rmc, \bualpha^{\rmc}}) - \cG^{\rmc}(\mu^{\rmnc,\bualpha^{\rmnc}}_T, \mu_T^{\rmc, \bualpha^{\rmc}})\Big| \Big],
        \\
        &\le L_\cF \epsilon'' + L_\cG \epsilon' = C \epsilon''',
\end{align*}
where, for the last inequality, we used the fact that the controls $\balpha^{i,\rmc}$ and $\balpha^{\rmc}$ are i.i.d. and the fact that $\cF^\rmc$ and $\cG^\rmc$ as Lipschitz continuous as discussed above. 
\end{proof}

\end{proof}

\subsection{Equilibrium characterization}
\label{app:MP-fbsde}

We start by stating in detail the Assumptions~{\bf SMP} and~{\bf Pontryagin Optimality}. In the MP setting, Assumption~{\bf SMP}  of~\cite[pp.161-162]{CarmonaDelarue_book_I} for the non-cooperative population states that:
\begin{itemize}
    \item The drift is an affine function of $(x,\alpha)$ of the form: $b^\rmnc(t,x,\alpha,\mu,\mu') = b_0^\rmnc(t,\mu,\mu') + b_1^\rmnc(t)x + b_2^\rmnc(t)\alpha$ where $b_0^\rmnc, b_1^\rmnc, b_2^\rmnc$ are measurable and bounded on bounded subsets of their respective domains, namely, $[0,T] \times \cP_2(\RR^d) \times \cP_2(\RR^d)$, $[0,T]$ and $[0,T]$.
    \item There exist constants $\lambda^\rmnc > 0$ and $L^\rmnc \ge 1$ such that the function $(x,\alpha) \mapsto f(t,x,\alpha,\mu,\mu')$  is once continuously differentiable with Lipschitz-continuous derivatives, the Lipschitz constant in $x$ and $\alpha$ being bounded by $L^\rmnc$. Moreover, it satisfies the following strong form of convexity: $f^\rmnc(t,\tilde x,\tilde\alpha,\mu,\mu') - f^\rmnc(t,x,\alpha,\mu,\mu') - (\tilde x-x, \tilde \alpha - \alpha) \cdot \partial_{(x,\alpha)}f^\rmnc(t,x,\alpha,\mu,\mu') \ge \lambda^\rmnc |\tilde\alpha - \alpha|^2$. Finally, $f^\rmnc, \partial_x f^\rmnc, \partial_\alpha f^\rmnc$ are locally bounded over $[0,T] \times \RR^d \times A \times \cP_2(\RR^d) \times \cP_2(\RR^d)$.
    \item The function $\RR^d \times \cP_2(\RR^d) \times \cP_2(\RR^d) \ni (x,\mu,\mu') \mapsto g^\rmnc(x,\mu,\mu')$ is locally bounded. Moreover, for any $\mu,\mu' \in \cP_2(\RR^d)$, the function $x \mapsto g^\rmnc(x,\mu,\mu')$ is once continuously differentiable and convex, and has a $L^\rmnc-$Lipschitz continuous first order derivative. 
\end{itemize}

In the MP setting, Assumption~{\bf Pontryagin Optimality} of~\cite[pp.542-543]{CarmonaDelarue_book_I} for the cooperative population states that:
\begin{itemize}
    \item The functions \( b^\rmc \) and \( f^\rmc \) are differentiable with respect to \( (x, \alpha) \), the mappings \( (x, \alpha, \mu, \mu') \mapsto \partial_x(b^\rmc, f^\rmc)(t, x, \alpha, \mu, \mu') \) and \( (x, \alpha, \mu, \mu') \mapsto \partial_\alpha(b^\rmc, f^\rmc)(t, x, \alpha, \mu, \mu') \) being continuous for each \( t \in [0, T] \). The functions \( b^\rmc \), and \( f^\rmc \) are also differentiable with respect to the variable \( \mu' \), the mapping \( (x, \alpha, \mu, X') \mapsto \partial_{\mu'}(b^\rmc, f^\rmc)(t, x, \alpha, \mu, \mathcal{L}(X'))(X') \) being continuous for each \( t \in [0, T] \). Similarly, the function \( g^\rmc \) is differentiable with respect to \( x \), the mapping \( (x, \mu, \mu') \mapsto \partial_x g^\rmc(x, \mu, \mu') \) being continuous. 
    The function \( g^\rmc \) is also differentiable with respect to the variable \( \mu' \), the mapping \( (x, \mu, X') \mapsto \partial_{\mu'} g^\rmc(x, \mu, \mathcal{L}(X'))(X') \) being continuous.
    \item The function \( [0, T] \ni t \mapsto (b^\rmc, f^\rmc)(t, 0, 0, \delta_0, \delta_0) \) is uniformly bounded. The derivatives \( \partial_x b^\rmc \) and \( \partial_\alpha b^\rmc \) are uniformly bounded and the mapping \( x' \mapsto \partial_{\mu'}b^\rmc(t, x, \alpha, \mu, \mu')(x') \) has an \( L^2(\mathbb{R}^d; \mu'; \mathbb{R}^d) \)-norm which is also uniformly bounded (i.e., uniformly in \( (t, x, \alpha, \mu, \mu') \)). There exists a constant \( L^\rmc \) such that, for any \( R \geq 0 \) and any \( (t, x, \alpha, \mu, \mu') \) such that \( |x|, |\alpha|, M_2(\mu), M_2(\mu') \leq R \), \( |\partial_x f^\rmc(t, x, \alpha, \mu, \mu')| \), \( |\partial_x g^\rmc(x, \mu)| \), and \( |\partial_{\mu'} f^\rmc(t, x, \alpha, \mu, \mu')(x')| \) are bounded by \( L^\rmc(1 + R) \) and the \( L^2(\mathbb{R}^d; \mu'; \mathbb{R}^d) \)-norms of \( x' \mapsto \partial_{\mu'} f^\rmc(x, \alpha, \mu, \mu')(x') \) and \( x' \mapsto \partial_{\mu'} g^\rmc(x, \mu, \mu')(x') \) are bounded by \( L^\rmc(1 + R) \).
\end{itemize}

For the sufficient condition, we also assume:
\begin{itemize}
    \item $(x,\mu,\mu') \mapsto g^\rmnc(x,\mu,\mu')$ and $(x,\mu,\mu') \mapsto g^\rmc(x,\mu,\mu')$ are convex;
    \item $(x,\alpha,\mu,\mu') \mapsto H^{\mixpop, \rmnc}(t, x,  \alpha, \mu, \mu', Y^{\rmnc}_t, Z^{\rmnc}_t)$ and $(x,\alpha,\mu,\mu') \mapsto H^{\mixpop, \rmc}(t, x,  \alpha, \mu, \mu', Y^{\rmc}_t, Z^{\rmc}_t)$ are convex $\mathrm{Leb}_1 \otimes \PP$-a.e..
\end{itemize}

\begin{proof}{Proof of Theorem~\ref{thm:MP-Pontryagin}}
    
    First, we consider the problem for the non-cooperative agents. $(\hat\balpha^\rmnc, \hat\bmu^\rmnc)$ can be viewed as the Nash equilibrium $(\hat\balpha, \hat\bmu)$ for an MFG for the non-cooperative population when the control of the cooperative agents is $\hat\balpha^{\rmc}$. In other words, it is an MFG in which, given the mean field $\hat\bmu$, a representative agent needs to find the minimizers $\hat\balpha$ to the cost $\balpha \mapsto J(\balpha; \hat\bmu) = J^\rmnc(\balpha; \hat\bmu, \hat\bmu^\rmc)$, and $\hat\bmu$ should coincide with $\bmu^{\hat\balpha}$. We recall that $J^\rmnc$ is defined in  see~\eqref{eq:mixedpop-mfg-cost-NC}. In this game, $\hat\bmu^\rmc$ simply enters as a fixed parameter. We can thus use the standard theory of MFGs. By~\cite[Proposition 3.23]{CarmonaDelarue_book_I}, the MFG equilibrium for the non-cooperative population is characterized by: 
    \[
        \hat\alpha^\rmnc_t = \argmin_\alpha H^{\mixpop, \rmnc}(t, X^\rmnc_t, \alpha, \hat\mu^\rmnc_t, \hat\mu^\rmc_t, Y^\rmnc_t, Z^\rmnc_t),
    \] 
    where:
    \begin{equation}
    \begin{cases}
        & dX^\rmnc_t = b^\rmnc(t, X^\rmnc_t, \hat{\alpha}^\rmnc_t, \hat\mu^\rmnc_t, \hat\mu^\rmc_t) dt + \sigma^\rmnc dW^\rmnc_t
        \\
        & dY^\rmnc_t = - \partial_x H^{\mixpop,\rmnc}(t, X^\rmnc_t, \hat{\alpha}^\rmnc_t, \hat\mu^\rmnc_t, \hat\mu^\rmc_t, Y^\rmnc_t, Z^\rmnc_t) dt \\
        &+ Z^\rmnc_t dW^\rmnc_t
        \\
        & X^\rmnc_0 \sim \mu^\rmnc_0, \qquad Y^\rmnc_T = g(X^\rmnc_T, \hat\mu^\rmnc_T, \hat\mu^\rmc_T),
    \end{cases}
    \end{equation}
    and $\hat\mu^\rmnc_t = \cL(X^\rmnc_t)$.

    On the other hand, the cooperative agents need to solve an MKV control problem which consists in finding the minimizer $\hat\balpha^\rmc$ of $\balpha \mapsto J(\balpha) = J^\rmc(\balpha; \hat\mu^\rmnc)$. By~\cite[Proposition 6.15 and Theorem 6.16]{CarmonaDelarue_book_I}, we obtain that the control $\hat{\balpha}^\rmc$ is optimal if and only if: 
    \[
        H^{\mixpop, \rmc}(t, X^\rmc_t, \hat{\alpha}^\rmc_t, \hat\mu^\rmnc_t, \hat\mu^\rmc_t, Y^\rmc_t, Z^\rmc_t) = \inf_\alpha H^{\mixpop, \rmc}(t, X^\rmc_t, \alpha, \hat\mu^\rmnc_t, \hat\mu^\rmc_t, Y^\rmc_t, Z^\rmc_t),
    \] 
    where:
    \begin{equation}
    \begin{cases}
        & dX^\rmc_t = b^\rmc(t, X^\rmc_t, \hat{\alpha}^\rmc_t, \hat\mu^\rmnc_t, \hat\mu^\rmc_t) dt + \sigma^\rmc dW^\rmc_t
        \\
        & dY^\rmc_t = \Big(- \partial_x H^{\mixpop,\rmc}(t, X^\rmc_t, \hat{\alpha}^\rmc_t, \hat\mu^\rmnc_t, \hat\mu^\rmc_t, Y^\rmc_t, Z^\rmc_t) \\
        & - \tilde\EE\left[\partial_{\mu^\rmc} H^{\mixpop,\rmc}(t, \tilde{X}^\rmc_t, \tilde{\hat{\alpha}}^\rmc_t, \hat\mu^\rmnc_t, \hat\mu^\rmc_t, \tilde{Y}^\rmc_t, \tilde{Z}^\rmc_t)(X^\rmc_t) \right] \Big) dt\\
        &+ Z^\rmc_t dW^\rmc_t
        \\
        & X^\rmc_0 \sim \mu^\rmc_0, \qquad Y^\rmc_T = g(X^\rmc_T, \hat\mu^\rmnc_T, \hat\mu^\rmc_T),
    \end{cases}
    \end{equation}
    and $\hat\mu^\rmc_t = \cL(X^\rmc_t)$.

    Combining the above two MKV FBSDEs yield the result. 
\end{proof}

\section{Proofs for fisher model}
\subsection{Existence and uniqueness in MI-MFG model}
\label{app:fisher_mi_proofs}

\begin{proof}{Proof of Theorem~\ref{the:existence_uniqueness_fisher_mi}} Using Corollary~\ref{cor:fisher_fbsde_mi}, showing existence and uniqueness of the mixed individual mean field Nash equilibrium is equivalent to showing the existence and uniqueness of the solution of the FBSDE introduced in Corollary~\ref{cor:fisher_fbsde_mi}. We will prove this in two steps:

\noindent \textbf{Step 1:} We take expectation of the FBSDE in Corollary~\ref{cor:fisher_fbsde_mi}, then use Banach Fixed Theorem to show that it has a unique solution for processes $(K_t, \overline{X_t}, \overline{Y}_t^1, \overline{Y}_t^2)_{t\in[0,T]}$. 
    
    \noindent\textbf{Step 2:} Given processes $(K_t, \overline{X_t}, \overline{Y}_t^1, \overline{Y}_t^2)_{t\in[0,T]}$, the FBSDE in Corollary~\ref{cor:fisher_fbsde_mi} becomes linear. We propose an ansatz and solve for $X_t, Y_t^1, Y_t^2$ given unique $K_t, \overline{X_t}, \overline{Y}_t^1, \overline{Y}_t^2$.
The details of the proof are as follows:

\noindent\textbf{Step 1. Showing existence and uniqueness of the mean processes using Banach Fixed Point Theorem.}

We take the expectation of the FBSDE in Corollary~\ref{cor:fisher_fbsde_mi} and write the following forward backward ordinary differential equation (FBODE) system for the mean processes:
\begin{equation*}
    \begin{aligned}
        dK_t &= \big[rK_t \left(1-K_t/\cK\right) - q K_t \overline{X}_t\big] dt\\
        d\overline{X}_t & = -\frac{\overline{Y}_t^2}{c_3}dt\\
        d\overline{Y}_t^1 &= -\Big[\big(r-\frac{2r}{\cK}K_t - q \overline{X}_t\big) \overline{Y}_t^1 + 2p_1 q K_t \overline{X}^2_t\Big] dt\\
        d\overline{Y}_t^2 &= -\Big[p_1 q (2-\lambda)K_t^2 \overline{X}_t +c_1+c_2 \overline{X}_t 
        - q(1-\lambda)K_t\overline{Y}^1_t\Big]dt\\
        K_0 &= \rho \cK, \quad \overline{X}_0 = x_0, \quad \overline{Y}_T^1=\overline{Y}_T^2=0 
    \end{aligned}
\end{equation*}

Our aim is to find a fixed point for the following mapping:
\begin{equation*}
\begin{aligned}
    \overline{X}^i_t 
    &\mapsto h_1(\overline{X}^i_t) = K^i_t \mapsto h_2(K^i_t) = \overline{Y}^{1,i}_t 
    \mapsto h_3(\overline{Y}^{1,i}_t) = \overline{Y}^{2,i}_t \mapsto h_4(\overline{Y}^{2,i}_t) = \check{\overline{X}}^{i}_t
\end{aligned}
\end{equation*}

We start with two processes $(\overline{X}_t^1)_{t\in[0,T]}$ and $(\overline{X}_t^2)_{t\in[0,T]}$ and we show that the mapping $h=h_4 \circ h_3 \circ h_2 \circ h_1$ is a contraction mapping, i.e.,
\begin{equation*}
    \lVert h(\overline{X}^1) - h(\overline{X}^2)\rVert_T \leq C \lVert \overline{X}^1 - \overline{X}^1\rVert_T 
\end{equation*}
with $C<1$ and $\lVert \overline{X} \rVert_T := \sup_{t\in[0,T]} \lVert \overline{X}_t \rVert$ where $\lVert \cdot \rVert$ is the $L_2$ norm.

We define $\widetilde{Z}_t = {Z}^1_t - {Z}^2_t$ where $Z_t \in \{\overline{X}_t, K_t, \overline{Y}^1_t, \overline{Y}^2_t, \check{\overline{X}}_t\}$.

\noindent\textbf{Step 1.1. Showing mapping $h_1$ is contraction.}

\begin{equation*}
    \begin{aligned}
        d\lVert \widetilde{K}_t\rVert^2 
        &=2 \widetilde{K}_t  d  \widetilde{K}_t\\
        &= 2 \widetilde{K}_t \big[\big(r \widetilde{K}_t -(r/\cK) \widetilde{K}_t (K_t^1+K_t^2) 
        - q\big(K_t^1 \widetilde{\overline{X}}_t + \overline{X}_t^2  \widetilde{K}_t\big)\big)\big] dt,
    \end{aligned}
\end{equation*}
where we added and subtracted $K_t^1\overline{X}_t^2$ to conclude. 

Then, we have:
\begin{equation*}
    \begin{aligned}
        &\lVert \widetilde{K}_t\rVert^2 
        \\
        &\leq 2 \int_0^t r \lVert\widetilde{K}_s\rVert^2 -\frac{2r \lVert K\rVert_T}{\cK} \lVert\widetilde{K}_s\rVert^2
        -q\lVert \overline{X}\rVert_T\lVert\widetilde{K}_s\rVert^2-q \lVert K\rVert_T<\widetilde{K}_s,\widetilde{\overline{X}}_s> ds
        \\
        &\leq \exp\Big(2t \big(r-\frac{2r \lVert K\rVert_T}{\cK}- (q/2\lVert K\rVert_T+q\lVert \overline{X}\rVert_T)
        \big)\Big) q\lVert K\rVert_T\int_0^t \lVert\widetilde{\overline{X}}_s\rVert^2 ds
        \\
        &\leq C^{(1)} \int_0^T \lVert\widetilde{\overline{X}}_s\rVert^2 ds
    \end{aligned}
\end{equation*}
$C^{(1)} = \exp\big(2T(r+q\overline{X}_{max})\big) q\cK$, where we denote the uniform bound on process $\overline{\bX}$ by $\overline{X}_{max}$. Above, in the second inequality, we use geometric-arithmetic mean inequality and Gr\"onwall inequality.

\noindent\textbf{Step 1.2. Showing mapping $h_2$ is contraction.}
\begin{equation*}
    \begin{aligned}
        &d\lVert \widetilde{\overline{Y}}^{1}_t\rVert^2 
        =2 \widetilde{\overline{Y}}^{1}_t  d  \widetilde{\overline{Y}}^{1}_t
        \\
        &=- 2 \widetilde{\overline{Y}}^{1}_t \big[r \widetilde{\overline{Y}}_t^1 - \frac{2r}{K} (K_t^1 \widetilde{\overline{Y}}_t^{1}+ \overline{Y}_t^{1,2} \widetilde{K}_t )
        -q(\overline{X}_t^1 \widetilde{\overline{Y}}_t^{1} + \overline{Y}_t^{1,2}\widetilde{\overline{X}}_t) 
        \\
        &\qquad\qquad 
        + 2p_1 q (K_t^1 (\overline{X}_t^1+\overline{X}_t^2)\widetilde{\overline{X}}_t+ \widetilde{K}_t (\overline{X}_t^2)^2)\big]dt,
    \end{aligned}
\end{equation*}
where we added and subtracted $K_t^1 \overline{Y}_t^{1,2}$, $\overline{X}_t^1 \overline{Y}_t^{1,2}$, and $K_t^1 (\overline{X}_t^2)^2$ to conclude.

Then, we have:
\begin{equation*}
    \begin{aligned}
        &\lVert \widetilde{\overline{Y}}^1_t\rVert^2 
        \\
        &\leq 2 \int_t^T r \lVert \widetilde{\overline{Y}}^1_s\rVert^2 +\frac{2r\lVert K\rVert_T}{\cK} \lVert \widetilde{\overline{Y}}^1_s\rVert^2 +\frac{2r}{\cK} \lVert \overline{Y}^1\rVert_T<\widetilde{\overline{Y}}^1_s,\widetilde{K}_s> 
        \\
        &\qquad 
        + q\lVert \overline{X}\rVert_T \lVert \widetilde{\overline{Y}}^1_s\rVert^2 + q \lVert \overline{Y}^1\rVert_T<\widetilde{\overline{Y}}^1_s,\widetilde{\overline{X}}_s>\\
        &\qquad+4p_1 q \lVert \overline{X}\rVert_T \lVert K\rVert_T <\widetilde{\overline{Y}}^1_s,\widetilde{\overline{X}}_s>
        +2p_1 q \lVert \overline{X}\rVert_T^2 <\widetilde{\overline{Y}}^1_s,\widetilde{K}_s> ds\\
        &\leq \exp\Big(2t\Big( r + \big(\dfrac{r}{\cK}+\frac{q}{2} \big)\lVert \overline{Y}^1\rVert_T + \frac{2r}{\cK} \lVert K\rVert_T
        + 2p_1 q \lVert K\rVert_T \lVert \overline{X}\rVert_T+ q \lVert \overline{X}\rVert_T + p_1 q \lVert \overline{X}\rVert_T^2 \Big)\Big) \times\\
        &\hskip7mm \Big(\big(\frac{r}{\cK}\lVert \overline{Y}^1\rVert_T + p_1 q \lVert \overline{X}\rVert^2_T\big)T C^{(1)} 
        +\big(\frac{q}{2}\lVert \overline{Y}^1\rVert_T + 2p_1 q \lVert \overline{X}\rVert_T \lVert K \rVert_T\big)\Big) \int_0^T \lVert\widetilde{\overline{X}}_s\rVert^2 ds\\
        =& C^{(2)} \int_0^T \lVert\widetilde{\overline{X}}_s\rVert^2 ds,
    \end{aligned}
\end{equation*}  

where 
\begin{equation*}
    \begin{aligned}
        C^{(2)}=&\exp\Bigg(2T\Big( r + \big(\frac{r}{\cK}+\frac{q}{2} \big) \overline{Y}^1_{max} + 2r
        + 2p_1 q \cK \overline{X}_{max} + q  \overline{X}_{max} + p_1 q  \overline{X}_{max}^2 \Big)\Bigg) \times\\
        &\Big(\big(\frac{r}{\cK} \overline{Y}^1_{max} + p_1 q  \overline{X}^2_{max}\big)T C^{(1)} 
        +\big(\frac{q}{2} \overline{Y}^1_{max} + 2p_1 q  \overline{X}_{max}  \cK\big)\Big),
    \end{aligned}
\end{equation*} 
where we used Gr\"onwall inequality to conclude. We denote the uniform bounds on processes on $\overline{\bX}$ and $\overline{\bY}^1$ with $\overline{X}_{max}$ and $\overline{Y}^1_{max}$, respectively.

\noindent\textbf{Step 1.3. Showing mapping $h_3$ is contraction.}

\begin{equation*}
    \begin{aligned}
        d\lVert \widetilde{\overline{Y}}^{2}_t\rVert^2 &=2 \widetilde{\overline{Y}}^{2}_t  d  \widetilde{\overline{Y}}^{2}_t\\
        & = -2 \widetilde{\overline{Y}}^{2}_t \big[p_1 q (2-\lambda)\big((K_t^1)^2 \widetilde{\overline{X}}_t + \widetilde{K}_t (K_t^1 +K_t^2) \overline{X}^2_t\big) \\
        &\quad+c_2\widetilde{\overline{X}}_t  -q(1-\lambda)\big(K_t^1\widetilde{\overline{Y}}_t^{1}+\widetilde{K}_t\overline{Y}_t^{1,2}\big) \big] dt,
    \end{aligned}
\end{equation*}
where we concluded by adding and subtracting $(K_t^1)^2\overline{X}_t^2 $, and $K_t^1 \overline{Y}_t^{1,2}$.

Then, by using Gr\"onwall inequality, we can write:
\begin{equation*}
    \begin{aligned}
        &\lVert \widetilde{\overline{Y}}^2_t\rVert^2 
        \\
        &\leq2 \int_t^T p_1 q (2-\lambda) \lVert K \rVert_T^2 <\widetilde{\overline{Y}}^2_s,\widetilde{\overline{X}}_s> 
        + 2 p_1 q (2-\lambda) \lVert K \rVert_T \lVert \overline{X}\rVert_T  <\widetilde{\overline{Y}}^2_s,\widetilde{K}_s>
        \\
        & 
        + c_2 <\widetilde{\overline{Y}}^2_s,\widetilde{\overline{X}}_s> +q (1-\lambda) \lVert K \rVert_T <\widetilde{\overline{Y}}^2_s,\widetilde{\overline{Y}}^1_s>
        +q (1-\lambda) \lVert \overline{Y}^1\rVert_T<\widetilde{\overline{Y}}^2_s,\widetilde{K}_s> ds\\
        &\leq \exp(tq (1-\lambda) \lVert K \rVert_T) \Big(T C^{(1)}\big( 2p_1 q (2-\lambda) \lVert K\rVert_T \lVert \overline{X}\rVert_T         \\
        & + q(1-\lambda) \lVert \overline{Y}^1\rVert_T\big) + \big(p_1 q (2-\lambda)\lVert K\rVert_T^2 + c_2\big)\\
        & + TC^{(2)} \big(p_1 q(2-\lambda) \lVert K\rVert_T^2 + 2p_1 q (2-\lambda)\lVert K\rVert_T \lVert \overline{X}\rVert_T +c_2\\
        &+ q(1-\lambda) \lVert K\rVert_T + q(1-\lambda) \lVert \overline{Y}^1\rVert_T)\big)\Big) \int_0^T \lVert \widetilde{\overline{X}}_s \rVert^2  ds\\
        &= C^{(3)} \int_0^T \lVert\widetilde{\overline{X}}_s\rVert^2 ds,
    \end{aligned}
\end{equation*}  
where 
\begin{equation*}
    \begin{aligned}
    C^{(3)}&= \exp(Tq (1-\lambda) \cK) \Big(T C^{(1)}\big( 2p_1 q (2-\lambda) \cK  \overline{X}_{max}        
        + q(1-\lambda)  \overline{Y}^1_{max}\big) + \big(p_1 q (2-\lambda)\cK^2 + c_2\big)\\
        &+ TC^{(2)} \big(p_1 q(2-\lambda) \cK^2 + 2p_1 q (2-\lambda)\cK \overline{X}_{max}         
        +c_2+ q(1-\lambda) \cK+ q(1-\lambda) \overline{Y}^1_{max})\big)\Big). 
    \end{aligned}
\end{equation*}

\noindent\textbf{Step 1.4. Showing mapping $h_4$ is contraction.}

\begin{equation*}
    \begin{aligned}
        d\lVert \widetilde{\overline{X}}^{\prime}_t\rVert^2 =2 \widetilde{\overline{X}}^{\prime}_t  d  \widetilde{\overline{X}}^{\prime}_t
    \end{aligned}
\end{equation*}

Then, we have:

\begin{equation*}
    \begin{aligned}
        \lVert \widetilde{\overline{X}}^{\prime}_t\rVert^2 &=  \frac{2}{c_3} \int_0^t < \widetilde{\overline{X}}^{\prime}_s,\widetilde{\overline{Y}}_s^2> ds\\
        &\leq \frac{1}{c_3} \int_0^t (\lVert \widetilde{\overline{X}}^{\prime}_s\rVert^2 +\lVert\widetilde{\overline{Y}}_s^2\rVert^2) ds\\
        &\leq \exp(t/c_3) 1/c_3  \int_0^t \lVert\widetilde{\overline{Y}}_s^2\rVert^2 ds\\
        &\leq  \exp(t/c_3) 1/c_3\int_0^t C^{(3)} \int_0^T \lVert\widetilde{\overline{X}}_s\rVert^2 ds \\
        &\leq C^{(4)} \int_0^T \lVert\widetilde{\overline{X}}_s\rVert^2 ds
    \end{aligned}
\end{equation*}
where $C^{(4)} =  \exp(T/c_3)   C^{(3)}T/c_3$.

Then we can conclude that:
\begin{equation*}
    \lVert \widetilde{\overline{X}}^{\prime}\rVert_T \leq C^{(4)}T \lVert \widetilde{\overline{X}}\rVert_T.
\end{equation*}
Under small terminal time $T$, we have $C^{(4)}T<1$ which shows that mapping $h$ is contraction. By using Banach fixed point theorem, we can conclude that there exists unique mean process flows $(\overline{\bX}, \boldsymbol{K}, \overline{\bY}^1, \overline{\bY}^2)$. Since $\overline{\hat{\alpha}}_t = -\overline{Y}_t^2/ c_3$, we also conclude existence and uniqueness for the mean equilibrium control process $\overline{\hat{\balpha}}$.
\vskip3mm
\noindent \textbf{Step 2: Showing existence and uniqueness of the state and adjoint processes in the linear form.}

In \textbf{Step 1}, we showed that under small time assumption, mean processes exist and unique. Given mean processes $\overline{\bX}, \boldsymbol{K}, \overline{\bY}^1, \overline{\bY}^2$, then the FBSDE that characterizes the MI mean field Nash equilibrium is coupled linearly. Firstly, we realize that the process $\boldsymbol{K}$ is deterministic and already proven to exist uniquely. Then we focus only on the FBSDE system regarding $\bX, \bY^1, \bY^2$. We see that given the mean processes the dynamics for $\bX$ and $\bY^2$ are linearly coupled. Therefore, we first start by proposing a linear ansatz for $\bY^2$ such that $Y_t^2 = A_t X_t +B_t$ where $(A_t, B_t)_{t\in[0,T]}$ are deterministic functions of time. After taking derivative of this ansatz and plugging it in the equations and matching the terms, we conclude that:
\begin{equation*}
    \begin{aligned}
        &\dot{A}_t -\frac{A_t^2}{c_3} + c_2 =0,\\
        &\dot{B}_t -\frac{A_tB_t}{c_3} + p_1q(2-\lambda) K_t^2 \overline{X}_t +c_1 - q(1-\lambda)K_t \overline{Y}_t^1=0,\\
         &A_T=0, \qquad B_T =0
    \end{aligned}
\end{equation*}

Above, the first equation is a scalar Riccati equation in which there exists a unique (and continuous) solution in the following form:
\begin{equation*}
    A_t =\frac{-c_2 \exp(2\sqrt{c_2/ c_3}(T-t))-1}{- \sqrt{c_2/ c_3}\exp(2\sqrt{c_2/ c_3}(T-t))- \sqrt{c_2/ c_3}}.
\end{equation*}
The second equation is a linear ODE with time dependent coefficients. Since $\bK, \overline{\bX}, \overline{\bY}^1, \bA$ are continuous and bounded, 
the linear ODE has a unique (and continuous) solution.

By plugging in the ansatz for $\bY^2$ (i.e., $Y_t^2 = A_t X_t +B_t$) in the SDE for $\bX$ we can write
\begin{equation*}
    dX_t = -\frac{A_t X_t + B_t}{c_3} dt +\sigma dW_t.
\end{equation*}
Since the drift and volatility terms are Lipschitz continuous in $x$ and they satisfy the linear growth condition in $x$, there exist a unique strong solution.

Finally for solving for $\bY^1$, we assume a ansatz for $\bY^1$ such that $Y_t^1 = C_t X_t + D_t$ where $(C_t, D_t)_{t\in[0,T]}$ are deterministic functions of time. After taking derivative of the ansatz and plugging it in the equations and matching the terms, we conclude that:
\begin{equation*}
    \begin{aligned}
        &\dot{C}_t +\big(r-\frac{2r}{\cK}K_t -q\overline{X}_t - \frac{A_t}{c_3}\big) C_t + 2p_1 q K_t \overline{X}_t = 0, && C_T=0\\
        &\dot{D}_t +\big(r-\frac{2r}{\cK}K_t -q\overline{X}_t \big) D_t -\frac{B_t C_t}{c_3}=0, && D_T=0
    \end{aligned}
\end{equation*}
The equations above are linear ODEs with time dependent coefficients. Since $\bK, \overline{\bX}, \bA$ are continuous, the first linear ODE has a unique (and continuous) solution. Since $\bK, \overline{\bX}, \boldsymbol{B}, \boldsymbol{C}$ are continuous, the second linear ODE also has a unique (and continuous) solution. 

This concludes that there exists a unique MI mean field Nash equilibrium for the fisher problem. 

\end{proof}

\subsection{Existence and uniqueness in MP-MFG model}
\label{app:fisher_mp_proofs}

\begin{proof}{Proof of Theorem~\ref{the:existence_uniqueness_fisher_mp}}

The proof follows similar ideas to the proof of Theorem~\ref{the:existence_uniqueness_fisher_mi}. This means, we first show the existence and uniqueness of the mean processes using Banach fixed point theorem. Then given the mean processes (including the common pool resource process), the FBSDE system becomes linear and we can propose linear ansatz for the adjoint processes. 

\noindent\textbf{Step 1. Showing existence and uniqueness of the mean processes using Banach Fixed Point Theorem.}

We take the expectation of the FBSDE system to write the FBODE system for the mean processes:
\begin{equation*}
\label{eq:FBODE_meanprocess_mixedpop}
    \begin{aligned}
        dK_t &= \big[rK_t \left(1-K_t/\cK\right) - q K_t (p\overline{X}^{\rmnc}_t + (1-p)\overline{X}^{\rmc}_t)\big] dt\\
        d\overline X^{\rmnc}_t & = -\frac{\overline Y_t^{2,\rmnc}}{c^{\rmnc}_3}dt \\
        d\overline X^{\rmc}_t & = -\frac{\overline Y_t^{2,\rmc}}{c^{\rmc}_3}dt \\
        d\overline Y_t^{1,\rmnc} &= -\Big[\big(r-\frac{2r}{\cK}K_t - q (p\overline{X}^{\rmnc}_t + (1-p)\overline{X}^{\rmc}_t)\big) \overline Y_t^{1,\rmnc} + 2p_1 q K_t (p\overline{X}^{\rmnc}_t + (1-p)\overline{X}^{\rmc}_t) \overline X^{\rmnc}_t{-p_0\overline X_t^{\rmnc}}\Big] dt \\
        d\overline Y_t^{2,\rmnc} &= -\Big[{ -p_0 K_t} + p_1 q K_t^2 (p\overline{X}^{\rmnc}_t + (1-p)\overline{X}^{\rmc}_t) +c^{\rmnc}_1+c^{\rmnc}_2 \overline X^{\rmnc}_t \Big]dt\\
        d\overline Y_t^{1,\rmc} &= -\Big[\big(r-\frac{2r}{\cK}K_t - q (p\overline{X}^{\rmnc}_t + (1-p)\overline{X}^{\rmc}_t)\big) \overline Y_t^{1,\rmc} + 2p_1 q K_t (p\overline{X}^{\rmnc}_t + (1-p)\overline{X}^{\rmc}_t) \overline X^{\rmc}_t{-p_0 \overline X_t^{\rmc}}\Big] dt \\
        d\overline Y_t^{2,\rmc} &= -\Big[{ -p_0 K_t}+p_1 q K_t^2 (p\overline{X}^{\rmnc}_t + (1-p)\overline{X}^{\rmc}_t) +c^{\rmc}_1+c^{\rmc}_2 \overline X^{\rmc}_t -qK_t(1-p)\overline{Y}_t^{1,\rmc} + p_1qK_t^2(1-p)\overline{X}^{\rmc}_t\Big]dt\\
        K_0 &= \rho \cK, \quad \overline X^{\rmnc}_0 = \overline\mu^{\rmnc}_0,\ \overline X^{\rmc}_0 = \overline \mu^{\rmc}_0, \quad \overline Y_T^{1,\rmnc}=\overline Y_T^{2,\rmnc}=\overline Y_T^{1,\rmc}=\overline Y_T^{2,\rmc}=0 
    \end{aligned}
\end{equation*}

Our aim is to find a fixed point for the following mapping:
\begin{equation*}
    \overline{X}^i_t \mapsto h_1(\overline{X}^i_t) = K^i_t \mapsto h_2(K^i_t) = \overline{Y}^{1,i}_t \mapsto h_3(\overline{Y}^{1,i}_t) = \overline{Y}^{2,i}_t \mapsto h_4(\overline{Y}^{2,i}_t) = \check{\overline{X}}^{i}_t.
\end{equation*}

Differently from the proof of Theorem~\ref{the:existence_uniqueness_fisher_mi}, we define the vector $\overline Z_t = (\overline Z_t^{\rmnc}, \overline Z_t^{\rmc})^\top$ where $\overline{Z}_t \in\{\overline{X}_t, \overline{Y}^1_t, \overline{Y}_t^2, \check{\overline{X}}_t\}$, the remaining of the proof follows similarly.

We start with two processes $(\overline{X}_t^1)_{t\in[0,T]}$ and $(\overline{X}_t^2)_{t\in[0,T]}$ and we show that the mapping $h=h_4 \circ h_3 \circ h_2 \circ h_1$ is a contraction mapping, i.e.,
\begin{equation*}
    \lVert h(\overline{X}^1) - h(\overline{X}^2)\rVert_T \leq C \lVert \overline{X}^1 - \overline{X}^1\rVert_T 
\end{equation*}
with $C<1$ and $\lVert \overline{X} \rVert_T := \sup_{t\in[0,T]} \lVert \overline{X}_t \rVert$ where $\lVert \cdot \rVert$ is the $L_2$ norm. As before we define $\widetilde{Z}_t = {Z}^1_t - {Z}^2_t$ where $Z_t \in \{\overline{X}_t, K_t, \overline{Y}^1_t, \overline{Y}^2_t, \check{\overline{X}}_t\}$.

\noindent\textbf{Step 1.1. Bounding the mapping $h_1$.}

\begin{equation*}
    \begin{aligned}
        d\lVert \widetilde{K}_t\rVert^2 &=2  \widetilde{K}_t  d  \widetilde{K}_t\\
        &= 2 \widetilde{K}_t \big[r \widetilde{K}_t -(r/\cK) \widetilde{K}_t (K_t^1+K_t^2) - qp\big(K_t^1 \widetilde{\overline{X}}^{\rmnc}_t + \overline{X}_t^{\rmnc,2}  \widetilde{K}_t\big) \\
        &\hskip1.5cm -q(1-p) \big(K_t^1 \widetilde{\overline{X}}^{\rmc}_t + \overline{X}_t^{\rmc,2}  \widetilde{K}_t\big) \big] dt
    \end{aligned}
\end{equation*}
where we added and subtracted $qpK_t^1\overline{X}_t^{\rmnc,2}$ and $q(1-p)K_t^1\overline{X}_t^{\rmc,2}$ to conclude. 

Then, we have:
\begin{equation*}
    \begin{aligned}
        \lVert \widetilde{K}_t\rVert^2 
        &\leq 2\int_0^t \Big(\big\lVert r-\dfrac{2r\lVert K\rVert_T}{\cK}-qp \lVert \overline{X}^{\rmnc}\rVert_T -q(1-p) \lVert \overline{X}^{\rmc}\rVert_T\big\rVert\lVert\widetilde{K}_s\rVert^2 \\
        &\hskip1.5cm+ \big\lVert qp\lVert K\rVert_T \big\rVert<\widetilde{K}_s,\widetilde{\overline{X}}^{\rmnc}_s> + \big\lVert q(1-p)\lVert K\rVert_T \big\rVert<\widetilde{K}_s,\widetilde{\overline{X}}^{\rmc}_s> \Big) ds
        \\
        \\
        &\leq \exp\Big(2t \big(\big\lVert r-\dfrac{2r\lVert K\rVert_T}{\cK}-qp \lVert \overline{X}^{\rmnc}\rVert_T -q(1-p) \lVert \overline{X}^{\rmc}\rVert_T\big\rVert + \frac{q  \lVert K\rVert_T}{2} \big)\Big) q  \lVert K\rVert_T \int_0^t \lVert\widetilde{\overline{X}}_s\rVert^2 ds\\
        &\leq C^{(1)} \int_0^T \lVert\widetilde{\overline{X}}_s\rVert^2 ds
    \end{aligned}
\end{equation*}
where $C^{(1)} = \exp\Big(2T \big(\big\lVert r+q \overline{X}_{max} \big\rVert + \frac{q  \cK}{2} \big)\Big) q  \cK$ where we assume the processes $\overline{\bX}^{\rmnc}$ and $\overline{\bX}^{\rmc}$ are uniformly bounded by $\overline{X}_{max}$. In the second inequality, Gr\"onwall inequality is used.

\noindent\textbf{Step 1.2. Bounding the mapping $h_2$.}

We first start by reminding the notations that will be used in this section. 
We define $\widetilde {\overline Y}_t^{1, \rmnc} = {\overline Y}_t^{1, \rmnc, 1} - {\overline Y}_t^{1, \rmnc, 2}$ (and similarly we define $\widetilde {\overline Y}_t^{1, \rmc}$). Then the vector $\widetilde{\overline Y}_t^1 = (\widetilde {\overline Y}_t^{1, \rmnc}, \widetilde {\overline Y}_t^{1, \rmc})^{\top} = \overline{Y}_t^{1,1} - \overline{Y}_t^{1,2}$ where $\overline{Y}_t^{1,i} = (\overline{Y}_t^{1, \rmnc, i}, \overline{Y}_t^{1, \rmc, i})^{\top}$ for $i=1,2$. 

Then we have:
\begin{equation*}
    \begin{aligned}
        d\lVert \widetilde{\overline{Y}}^{1}_t\rVert^2 &=2 (\widetilde{\overline{Y}}^{1}_t)^{\top}  d  \widetilde{\overline{Y}}^{1}_t\\
        &=-2 \widetilde{\overline{Y}}^{1, \rmnc}_t \Big[r \widetilde{\overline{Y}}_t^{1,\rmnc} - \frac{2r}{\cK} (K_t^1 \widetilde{\overline{Y}}_t^{1,\rmnc}+  \overline{Y}_t^{1,\rmnc, 2}\widetilde{K}_t)- qp (\overline{X}_t^{\rmnc,1}\widetilde{\overline{Y}}_t^{1,\rmnc} + \overline{Y}_t^{1,\rmnc,2}\widetilde{\overline{X}}_t^{\rmnc})\\
        &\hskip0.5cm - q(1-p) (\overline{X}_t^{\rmc,1}\widetilde{\overline{Y}}_t^{1,\rmnc} + \overline{Y}_t^{1,\rmnc,2}\widetilde{\overline{X}}_t^{\rmc}) +2p_1qp\big(K_t^1 (\overline{X}_t^{\rmnc,1} + \overline{X}_t^{\rmnc,2})\widetilde{\overline{X}}_t^{\rmnc}+\tilde{K}_t (\overline{X}_t^{\rmnc,2})^2\big)\\
        &\hskip0.5cm +2p_1q(1-p)\big( K_t^1(\overline{X}_t^{\rmc,1} \widetilde{\overline{X}}_t^{\rmnc} + \overline{X}_t^{\rmnc,2}\widetilde{\overline{X}}_t^{\rmc})+\widetilde{K}_t \overline{X}_t^{\rmc,2} \overline{X}_t^{\rmnc,2})\big)\Big]dt
        \\
        &\hskip0.5cm -2 \widetilde{\overline{Y}}^{1, \rmc}_t \Big[r \widetilde{\overline{Y}}_t^{1,\rmc} - \frac{2r}{\cK} (K_t^1 \widetilde{\overline{Y}}_t^{1,\rmc}+  \overline{Y}_t^{1,\rmc, 2}\widetilde{K}_t)- qp (\overline{X}_t^{\rmnc,1}\widetilde{\overline{Y}}_t^{1,\rmc} + \overline{Y}_t^{1,\rmc,2}\widetilde{\overline{X}}_t^{\rmnc})\\
        &\hskip0.5cm - q(1-p) (\overline{X}_t^{\rmc,1}\widetilde{\overline{Y}}_t^{1,\rmc} + \overline{Y}_t^{1,\rmc,2}\widetilde{\overline{X}}_t^{\rmc}) +2p_1qp\big(K_t^1(\overline{X}_t^{\rmc,2} \widetilde{\overline{X}}_t^{\rmnc} + \overline{X}_t^{\rmnc,1}\widetilde{\overline{X}}_t^{\rmc})+\widetilde{K}_t \overline{X}_t^{\rmc,2} \overline{X}_t^{\rmnc,2})\big)\\
        &\hskip0.5cm +2p_1q(1-p)\big( K_t^1 (\overline{X}_t^{\rmc,1} + \overline{X}_t^{\rmc,2})\widetilde{\overline{X}}_t^{\rmc} +\widetilde{K}_t (\overline{X}_t^{\rmc,2})^2\big)\Big]dt
    \end{aligned}
\end{equation*}

To conclude the equation above, we added and subtracted the interaction terms. Then, we have after using the inequality arithmetic-geometric mean inequality:
\begin{equation*}
    \begin{aligned}
        \lVert \widetilde{\overline{Y}}^{1}_t\rVert^2 &\leq 2\int_t^T \Big[\Big(\big\lVert r-\dfrac{2r\lVert K\rVert_T}{\cK}-qp \lVert \overline{X}^{\rmnc}\rVert_T -q(1-p) \lVert \overline{X}^{\rmc}\rVert_T\big\rVert\\
        &\hskip1cm+\big\lVert p_1 q\lVert \overline{X}^{\rmnc}\rVert_T (p\lVert \overline{X}^{\rmnc}\rVert_T+ (1-p)\lVert \overline{X}^{\rmc}\rVert_T) -\dfrac{r}{\cK}\lVert \overline{Y}^{1,\rmnc}\rVert_T\big\rVert\\
        &\hskip1cm +\big\lVert p_1 q \lVert K\rVert_T (2p\lVert \overline{X}^{\rmnc}\rVert_T+ (1-p)\lVert \overline{X}^{\rmc}\rVert_T) -0.5qp\lVert \overline{Y}^{1,\rmnc}\rVert_T\big\rVert
        \\
        &\hskip1cm +\big\lVert p_1 q \lVert K\rVert_T (1-p)\lVert \overline{X}^{\rmnc}\rVert_T-0.5q(1-p)\lVert \overline{Y}^{1,\rmnc}\rVert_T\big\rVert \Big) \lVert \widetilde{\overline{Y}}^{1,\rmnc}_s\rVert^2 \\
        &\hskip0.5cm+\Big(\big\lVert r-\dfrac{2r\lVert K\rVert_T}{\cK}-qp \lVert \overline{X}^{\rmnc}\rVert_T -q(1-p) \lVert \overline{X}^{\rmc}\rVert_T\big\rVert\\
        &\hskip1cm+\big\lVert p_1 q\lVert \overline{X}^{\rmc}\rVert_T (p\lVert \overline{X}^{\rmnc}\rVert_T+ (1-p)\lVert \overline{X}^{\rmc}\rVert_T) -\dfrac{r}{K}\lVert \overline{Y}^{1,\rmc}\rVert_T\big\rVert\\
        &\hskip1cm +\big\lVert p_1 q \lVert K\rVert_T p\lVert \overline{X}^{\rmc}\rVert_T -0.5qp\lVert \overline{Y}^{1,\rmc}\rVert_T\big\rVert
        \\
        &\hskip1cm +\big\lVert p_1 q \lVert K\rVert_T ((p\lVert \overline{X}^{\rmnc}\rVert_T+(1-p)\lVert \overline{X}^{\rmc}\rVert_T)-0.5q(1-p)\lVert \overline{Y}^{1,\rmc}\rVert_T\big\rVert\Big)\lVert \widetilde{\overline{Y}}^{1,\rmc}_s\rVert^2\\
        &\hskip0.5cm+\Big(\big\lVert p_1 q \lVert K\rVert_T (2p\lVert \overline{X}^{\rmnc}\rVert_T+ \lVert \overline{X}^{\rmc}\rVert_T) -0.5qp(\lVert \overline{Y}^{1,\rmnc}\rVert_T+ \lVert \overline{Y}^{1,\rmc}\rVert_T)\big\rVert\Big) \lVert \widetilde{\overline{X}}^{\rmnc}_s\rVert^2\\
        &\hskip0.5cm+\Big(\big\lVert p_1 q \lVert K\rVert_T (\lVert \overline{X}^{\rmnc}\rVert_T+ 2(1-p)\lVert \overline{X}^{\rmc}\rVert_T) -0.5q(1-p)(\lVert \overline{Y}^{1,\rmnc}\rVert_T+ \lVert \overline{Y}^{1,\rmc}\rVert_T)\big\rVert\Big) \lVert \widetilde{\overline{X}}^{\rmc}_s\rVert^2\\
        &\hskip0.5cm+\Big(\big\lVert p_1 q( p\lVert \overline{X}^{\rmnc}\rVert_T(\lVert \overline{X}^{\rmnc}\rVert_T+ \lVert \overline{X}^{\rmc}\rVert_T) +(1-p)\lVert \overline{X}^{\rmc}\rVert_T(\lVert \overline{X}^{\rmnc}\rVert_T+ \lVert \overline{X}^{\rmc}\rVert_T)) \\
        &\hskip1cm-\dfrac{r}{\cK}(\lVert \overline{Y}^{1,\rmnc}\rVert_T+ \lVert \overline{Y}^{1,\rmc}\rVert_T)\big\rVert\Big) \lVert \widetilde{K}_s\rVert^2 \Big]ds\\
        &= 2\int_t^T C^{(2,1)} \lVert \widetilde{\overline{Y}}^{1,\rmnc}_s\rVert^2  + C^{(2,2)} \lVert \widetilde{\overline{Y}}^{1,\rmc}_s\rVert^2 + C^{(2,3)} \lVert \widetilde{\overline{X}}^{\rmnc}_s\rVert^2 + C^{(2,4)} \lVert \widetilde{\overline{X}}^{\rmc}_s\rVert^2 + C^{(2,5)} \lVert \widetilde{K}_s\rVert^2 ds\\
        &\leq 2\int_t^T C^{(2,6)} \lVert \widetilde{\overline Y}_s^1\rVert^2ds + \int_0^T C^{(2,7)} \lVert \widetilde{\overline X}_s\rVert^2 ds\\
        &\leq C^{(2)} \int_0^T \lVert\widetilde{\overline{X}}_s\rVert^2 ds,
    \end{aligned}
\end{equation*}
where we simplified the notations by introducing labels for the constants in the first inequality. Then, we define $C^{(2,6)} = \max (C^{(2,1)}, C^{(2,2)}) $ and $C^{(2,7)}= C^{(2,5)} C^{(1)} T + \max (C^{(2,3)}, C^{(2,4)})$.
Here, $C^{(2)}=\exp(2TC^{(2,6)}) C^{(2,7)}$ is a constant depending on model parameters $p_1, q, p, r, K$ and the uniform bounds on processes $\overline{\bX}^{\rmnc}$ (similarly for $\overline{\bX}^{\rmc}$), $\bK$, and $\overline{\bY}^{1,\rmnc}$ (similarly for $\overline{\bY}^{1,\rmc}$), $\overline{X}_{max} (=\lVert \overline{X}^{\rmnc} \rVert_T=\lVert \overline{X}^{\rmc} \rVert_T)$, $K_{max}= \lVert K \rVert_T$, and $\overline{Y}^1_{max} (=\lVert \overline{Y}^{1,\rmnc} \rVert_T=\lVert \overline{Y}^{1,\rmc} \rVert_T)$, respectively.

\noindent\textbf{Step 1.3. Bounding the mapping $h_3$.}

Similar to the notations and approach in Step 1.2 of this proof, we have:
\begin{equation*}
    \begin{aligned}
        d\lVert \widetilde{\overline{Y}}^{2}_t\rVert^2 &=2 (\widetilde{\overline{Y}}^{2}_t)^{\top}  d  \widetilde{\overline{Y}}^{2}_t\\
        &=-2 \widetilde{\overline{Y}}^{2, \rmnc}_t \Big[ -p_0 \widetilde{K}_t + p_1 q \big( (K_t^1)^2 (p\overline{X}_t^{\rmnc,1}+(1-p)\overline{X}_t^{\rmc,1}) - (K_t^2)^2 (p\overline{X}_t^{\rmnc,2}+(1-p)\overline{X}_t^{\rmc,2}) \big)+ c_2^{\rmnc} \widetilde{\overline X}_t^{\rmnc}\Big]dt\\
        \\
        &\hskip3mm -q(1-p)\big(K_t^1\overline{Y}_t^{1,\rmc,1}-K_t^2\overline{Y}_t^{1,\rmc,2}\big) + p_1 q (1-p) \big((K_t^1)^2 \overline{X}_t^{\rmc,1}-(K_t^2)^2 \overline{X}_t^{\rmc,2}\big)\Big]dt
    \end{aligned}
\end{equation*}
By adding and subtracting the cross terms, we end up with:
\begin{equation*}
    \begin{aligned}
        \lVert \widetilde{\overline{Y}}_t^2\rVert^2 &= 2\int_t^T \Bigg(\Big\lVert-p_0 +2p_1q p\lVert K\rVert_T \lVert X^{\rmnc}\rVert_T + 2p_1 q(1-p)K\rVert_T \lVert X^{\rmc}\rVert_T \Big\rVert < \widetilde{\overline Y}_s^{2, \rmnc}, \widetilde{K}_s>\\
        &\hskip1cm +\Big\lVert p_1 q p \lVert K\rVert^2_T +c_2^{\rmnc}\Big\rVert< \widetilde{\overline Y}_s^{2, \rmnc}, \widetilde{\overline X}_s^{\rmnc}>+ \Big\lVert p_1 q (1-p) \lVert K\rVert^2_T\Big\rVert< \widetilde{\overline Y}_s^{2, \rmnc}, \widetilde{\overline X}_s^{\rmc}>\\
        &\hskip1cm+\Big\lVert-p_0 +2p_1q p\lVert K\rVert_T \lVert X^{\rmnc}\rVert_T + 4p_1 q(1-p)K\rVert_T \lVert X^{\rmc}\rVert_T -q(1-p) \lVert \overline{Y}^{1,\rmc}\rVert_T  \Big\rVert < \widetilde{\overline Y}_s^{2, \rmc}, \widetilde{K}_s>\\
        &\hskip1cm+\Big\lVert p_1 q p \lVert K\rVert^2_T\Big\rVert< \widetilde{\overline Y}_s^{2, \rmc}, \widetilde{\overline X}_s^{\rmnc}>+ \Big\lVert 2p_1 q (1-p) \lVert K\rVert^2_T+c_2^{\rmc }\Big\rVert< \widetilde{\overline Y}_s^{2, \rmc}, \widetilde{\overline X}_s^{\rmc}>\\
        &\hskip1cm+ \Big\lVert  q (1-p) \lVert K\rVert_T\Big\rVert < \widetilde{\overline Y}_s^{2, \rmc}, \widetilde{\overline Y}_s^{1, \rmc}> \Bigg)ds\\
        &\leq\int_t^T\Big(C^{(3,1)} \lVert \widetilde{\overline{Y}}^{2,\rmnc}_s\rVert^2  + C^{(3,2)} \lVert \widetilde{\overline{Y}}^{2,\rmc}_s\rVert^2 + C^{(3,3)} \lVert \widetilde{\overline{X}}^{\rmnc}_s\rVert^2+ C^{(3,4)} \lVert \widetilde{\overline{X}}^{\rmc}_s\rVert^2 + C^{(3,5)} \lVert \widetilde{\overline{Y}}^{1,\rmc}_s\rVert^2 + C^{(3,6)} \lVert \widetilde{K}_s\rVert^2\Big) ds\\
        &\leq \int_t^T C^{(3,7)} \lVert \widetilde{\overline Y}_s^2\rVert^2 ds+\int_0^T C^{(3,8)} \lVert \tilde{\overline X}_s\rVert^2 ds\\
        &\leq C^{(3)} \int_0^T \lVert\widetilde{\overline{X}}_s\rVert^2 ds.\\ 
    \end{aligned}
\end{equation*}
where we used the fact that $a^2+b^2\geq 2ab$ and simplified the notations by introducing labels for the coefficients. Then we define $C^{(3,7)} = \max (C^{(3,1)}, C^{(3,2)}) $ and 
$C^{(3,8)}= T(C^{(3,5)} C^{(2)}+C^{(3,6)} C^{(1)})  + \max (C^{(3,3)}, C^{(3,4)})$.
Here, $C^{(3)}=\exp(TC^{(3,7)}) C^{(3,8)}$ is a constant depending on model parameters $p_1, q, p, r, K$ and the uniform bounds on processes $\overline{\bX}^{\rmnc}$ (similarly for $\overline{\bX}^{\rmc}$), $\bK$, and $\overline{\bY}^{1,\rmnc}$ (similarly for $\overline{\bY}^{1,\rmc}$), $\overline{X}_{max} (=\lVert \overline{X}^{\rmnc} \rVert_T=\lVert \overline{X}^{\rmc} \rVert_T)$, $K_{max}= \lVert K \rVert_T$, and $\overline{Y}^1_{max} (=\lVert \overline{Y}^{1,\rmnc} \rVert_T=\lVert \overline{Y}^{1,\rmc} \rVert_T)$, respectively.

\noindent\textbf{Step 1.4. Bounding the mapping $h_4$.}

Finally, we have
\begin{equation*}
    \begin{aligned}
        d\lVert \widetilde{\overline{X}}^{\prime}_t\rVert^2 &=2(\widetilde{\overline{X}}^{\prime}_t)^{\top} d\widetilde{\overline{X}}^{\prime}_t\\
        &= 2 \widetilde{\overline{X}}^{\rmnc,\prime}_t\Big(-\dfrac{\widetilde{Y}_t^{2,\rmnc}}{c_3^{\rmnc}}\Big)dt + 2\widetilde{\overline{X}}^{\rmc,\prime}_t\Big(-\dfrac{\widetilde{Y}_t^{2,\rmc}}{c_3^{\rmc}}\Big)dt 
    \end{aligned}
\end{equation*}    
        Then, we have
        \begin{equation*}
            \begin{aligned}
                \lVert \widetilde{\overline{X}}^{\prime}_t\rVert^2&=2\int_0^t-\frac{1}{c_3^{\rmnc}}<\widetilde{\overline{X}}_s^{\rmnc,\prime},\widetilde{Y}_s^{2,\rmnc}> - \frac{1}{c_3^{\rmc}}<\widetilde{\overline{X}}_s^{\rmc,\prime}, \widetilde{Y}_s^{2,\rmc}> ds\\
                &\leq \frac{1}{\min(c_3^{\rmnc}, c_3^{\rmc})} \int_0^t \big(\lVert\widetilde{\overline{X}}_s^{\prime}\rVert^2+\lVert\widetilde{Y}_s^{2}\rVert^2\big)
                \\
                &\leq \exp(T/\min(c_3^{\rmnc}, c_3^{\rmc}))\frac{T C^{(3)}}{\min(c_3^{\rmnc}, c_3^{\rmc})}\int_0^T \lVert\widetilde{\overline{X}}_s\rVert^2 ds
            \end{aligned}
        \end{equation*}

        We define $C^{(4)} = \exp(T/\min(c_3^{\rmnc}, c_3^{\rmc}))\frac{T C^{(3)}}{\min(c_3^{\rmnc}, c_3^{\rmc})}$. Then we conclude that
\begin{equation*}
    \lVert \widetilde{\overline{X}}^{\prime}\rVert_T \leq C^{(4)}T \lVert \widetilde{\overline{X}}\rVert_T.
\end{equation*}
Under small time $T$, we have $C^{(4)}T<1$ which shows that mapping $h$ is contraction. By using Banach fixed point theorem, we can conclude that there exist unique mean process flows $(\overline{\bX}, \boldsymbol{K}, \overline{\bY}^1, \overline{\bY}^2)$ where $\overline{\bX} = (\overline{bX}^{\rmnc}, \overline{bX}^{\rmc})$, $\overline{\bY}^1 = (\overline{\bY}^{1,\rmnc}, \overline{\bY}^{1,\rmc})$, and $\overline{\bY}^2 = (\overline{\bY}^{2,\rmnc}, \overline{\bY}^{2,\rmc})$. Since $\overline{\hat{\alpha}}^{\rmnc}_t = -\overline{Y}_t^{2, \rmnc}/ c^{\rmnc}_3$ and $\overline{\hat{\alpha}}^{\rmc}_t = -\overline{Y}_t^{2, \rmc}/ c^{\rmc}_3$, we also conclude existence and uniqueness for the mean equilibrium control processes $\overline{\hat{\balpha}}= (\overline{\hat{\balpha}}^{\rmnc}, \overline{\hat{\balpha}}^{\rmc})$.
\vskip3mm
\noindent \textbf{Step 2: Showing existence and uniqueness of the state and adjoint processes in the linear form.}

In the \textbf{Step 1}, we showed that under small time assumption, mean processes and the common pool resource process exist and unique. Given mean processes $\overline{\bX} = (\overline{\bX}^{\rmnc}, \overline{\bX}^{\rmc}), \overline{\bY}^1=(\overline{\bY}^{1,\rmnc}, \overline{\bY}^{1,\rmc}), \overline{\bY}^2=(\overline{\bY}^{2,\rmnc}, \overline{\bY}^{2,\rmc})$ and the common pool resource process $\bK$, then the FBSDE that characterizes the MP solution is coupled linearly. In this way, we focus only on the FBSDE system regarding $\bX, \bY^1, \bY^2$ and we write them in the vector form as follows:
\begin{equation*}
    \begin{aligned}
        dX_t &= M_t^1 Y_t^2 dt + \Sigma_1 dW_t\\
        dY^1_t &= (M_t^2 X_t + M_t^3 Y_t^1) dt + \Sigma_2 dW_t\\
        dY^2_t &= (M_t^4 X_t + M_t^5) dt + \Sigma_3 dW_t\\
    \end{aligned}
\end{equation*}
where $M_t^1, M_t^2, M_t^3, M_t^4$ are matrices and $M_t^5$ is a vector depending on the model parameters, mean processes $ \overline{\bX}, \overline{\bY}^{1}, \overline{\bY}^{2}$, and the common pool resource process $\bK$ such that
{\small\begin{equation*}\arraycolsep3pt 
    M_t^1 =  \begin{bmatrix}
        -1/c_3^{\rmnc} & 0\\
        0 & -1/c_3^{\rmc}
    \end{bmatrix},
    M_t^2 =- \begin{bmatrix}
        2p_1 q K_t (p \overline{X}_t^{\rmnc} +(1-p)\overline{X}_t^{\rmc}) -p_0 & 0\\
        0 & 2p_1 q K_t (p \overline{X}_t^{\rmnc} +(1-p)\overline{X}_t^{\rmc}) -p_0 
        \end{bmatrix},
\end{equation*}\vskip0mm
\begin{equation*}\arraycolsep3pt 
    M_t^3 = -\begin{bmatrix}
       r-(2r/K) K_t-q (p \overline{X}_t^{\rmnc} +(1-p)\overline{X}_t^{\rmc}) & 0\\
       0& r-(2r/K) K_t-q (p \overline{X}_t^{\rmnc} +(1-p)\overline{X}_t^{\rmc})
        \end{bmatrix},
        M_t^4 =\begin{bmatrix}
        -c_2^{\rmnc} & 0\\
        0 & -c_2^{\rmc}
    \end{bmatrix},
\end{equation*}\vskip0mm
\begin{equation*}\arraycolsep3pt 
M_t^5 = - \begin{bmatrix}
-p_0K_t +p_1 q K_t^2 (p \overline{X}_t^{\rmnc} +(1-p)\overline{X}_t^{\rmc}) +c_1^{\rmnc}\\
-p_0K_t +p_1 q K_t^2 (p \overline{X}_t^{\rmnc} +(1-p)\overline{X}_t^{\rmc}) +c_1^{\rmc} -qK_t (1-p) \overline{Y}_t^{1, \rmc} + p_1 q K_t^2 (1-p) \overline{X}_t^{\rmc}
\end{bmatrix}.
\end{equation*}}

We see that given the mean processes the dynamics for $\bX$ and $\bY^2$ are linearly coupled. Therefore, we first start by proposing a linear ansatz for $\bY^2$ such that $Y_t^2 = A_t X_t +B_t$ where $(A_t, B_t)_{t\in[0,T]}$ are deterministic functions of time that are matrix and vector valued, respectively. After taking derivative of this ansatz and plugging it in the equations and matching the terms, we conclude that:
\begin{equation}
\begin{aligned}
    \dot{A}_t + A_tM_t^1 A_t - M_t^4 &= 0\\
    \dot{B}_t+A_tM_t^1 B_t - M_t^5 &=0
\end{aligned}
\end{equation}
where $A_T=\boldsymbol{0}$ and $B_T=\boldsymbol{0}$. The first equation is a matrix Riccati equation with constant coefficients. This equation has a unique, positive symmetric, continuous, solution, see~\cite[Chapter 14.3]{kucera_2009} and~\cite{Abou-Kandil_Freiling_Ionescu_Jank_2003}. After plugging in $\bA$, the second equation is a linear ODE system that has time dependent coefficients. Since the coefficients are continuous and locally bounded in the time interval $[0, T]$, there exists a unique continuous solution for $\boldsymbol{B}$, see~\cite[Chapter 1]{Abou-Kandil_Freiling_Ionescu_Jank_2003}.

By plugging in the ansatz for $\bY^2$ (i.e., $Y_t^2 = A_t X_t +B_t$) in the SDE for $\bX$ we can write
\begin{equation*}
    dX_t = (M_t^1 A_t X_t + M_t^1 B_t) dt +\sigma dW_t.
\end{equation*}
Since the drift and volatility terms are Lipschitz continuous in $x$ and they satisfy the linear growth condition in $x$, there exists a unique strong solution.

Finally for solving for $\bY^1$, we assume an ansatz for $\bY^1$ such that $Y_t^1 = C_t X_t + D_t$ where $(C_t, D_t)_{t\in[0,T]}$ are deterministic functions of time that are matrix and vector valued, respectively. After taking derivative of the ansatz and plugging it in the equations and matching the terms, we conclude that:
\begin{equation*}
    \begin{aligned}
        &\dot{C}_t + C_t M_t^1 A_t - M_t^3 C_t -M_t^2 = 0, && C_T=\boldsymbol{0}\\
        &\dot{D}_t +C_t M_t^1 B_t - M_t^3 D_t=0, && D_T=\boldsymbol{0}.
    \end{aligned}
\end{equation*}
The equations above are linear ODE systems with time dependent coefficients. The first one is uncoupled and it can be solved itself, then plugging in $\boldsymbol{C}$, the second system one can be solved. Since $\bK, \overline{\bX}, \bA$ are continuous, the first linear ODE has a unique (and continuous) solution. Since $\bK, \overline{\bX}, \boldsymbol{B}, \boldsymbol{C}$ are continuous, the second linear ODE also has a unique (and continuous) solution. 

This concludes that there exists a unique mixed population mean field equilibrium for the fisher problem.
\end{proof}

\end{document}